\documentclass[12pt, a4paper]{article}

\usepackage[T1]{fontenc}
\usepackage{amsmath}
\usepackage{amsfonts}
\usepackage{amssymb}
\usepackage{amsthm}
\usepackage{enumitem}
\usepackage[normalem]{ulem}
\usepackage[mathscr]{euscript}
\usepackage{mathtools}
\usepackage{dsfont}
\usepackage{float}
\usepackage{times}
\usepackage{bbm}
\usepackage{xcolor}
\usepackage{xparse}
\usepackage[colorlinks,linkcolor=blue,citecolor=blue,urlcolor=blue]{hyperref}
\usepackage[capitalise]{cleveref}
\usepackage[left=3cm, right=3cm, bottom=2cm, top=2cm]{geometry}
\usepackage{tcolorbox}
\usepackage{tikz}
\usepackage{tikz-cd} 
\usetikzlibrary{calc}
\usetikzlibrary{arrows}
\usetikzlibrary{decorations.markings}
\usetikzlibrary{3d}
\usetikzlibrary{shapes.geometric}
\usepackage{tkz-euclide}
\usepackage{pgfplots}
\pgfplotsset{compat=1.17}
\def\centerarc[#1](#2)(#3:#4:#5)
{ \draw[#1] ($(#2)+({#5*cos(#3)},{#5*sin(#3)})$) arc (#3:#4:#5); }

\usepackage{calc}
\usepackage{accents}

\usepackage{thmtools}
\usepackage{thm-restate}

\newtheorem{theorem}{Theorem}[section]
\newtheorem*{theorem*}{Theorem}
\newtheorem{lemma}[theorem]{Lemma}
\newtheorem{proposition}[theorem]{Proposition}
\newtheorem{corollary}[theorem]{Corollary}

\theoremstyle{definition}

\newtheorem{notation}[theorem]{Notation}
\newtheorem{definition}[theorem]{Definition}
\newtheorem{remark}[theorem]{Remark}

\theoremstyle{remark}

\numberwithin{equation}{section}

\crefname{example}{Example}{Examples}
\Crefname{example}{Example}{Examples}

\crefname{assumption}{Assumption}{Assumptions}
\Crefname{assumption}{Assumption}{Assumptions}

\crefname{condition}{Condition}{Conditions}
\Crefname{condition}{Condition}{Conditions}

\usepackage[framemethod=TikZ]{mdframed}
\newmdtheoremenv{ftheorem}[theorem]{Theorem}
\newmdtheoremenv{fdefinition}[theorem]{Definition}
\newmdtheoremenv{flemma}[theorem]{Lemma}

\NewDocumentEnvironment{mytheorem}{m}%
  {\renewcommand{\thetheorem}{$\mathbf{#1}$}%
    \begin{theorem}
  \end{theorem}
  }

\setlist{topsep=1ex, itemsep=0.5ex, before={\setlist{topsep=-.5ex}}}
\setlist[itemize]{label=\textbullet}

\newcommand{\R}{\ensuremath{\mathbb{R}}}

\newcommand{\II}{\ensuremath{I \! \! I}}

\def\f{\frac}
\def\epsilon{\varepsilon}

\NewDocumentCommand{\Lip}{om}{\IfNoValueTF{#1}{|#2|_{\mathrm{Lip}}}{|#2|_{\mathrm{Lip};\,#1}}}

\def\E{\mathbb E}

\renewcommand{\geq}{\geqslant}
\renewcommand{\leq}{\leqslant}

\def\${|\!|\!|}

\def\<{\left\langle}
\def\>{\right\rangle}

\newcommand{\vertiii}[1]{{\left\vert\kern-0.25ex\left\vert\kern-0.25ex\left\vert #1
\right\vert\kern-0.25ex\right\vert\kern-0.25ex\right\vert}}
\newcommand{\rom}[1]{(\textup{\uppercase\expandafter{\romannumeral#1}})}
\def\sgn{{\mathop {\rm sign}}}

\makeatletter
\newcommand{\substackal}[1]{%
\vcenter{%
\Let@ \restore@math@cr \default@tag
\baselineskip\fontdimen10 \scriptfont\tw@
\advance\baselineskip\fontdimen12 \scriptfont\tw@
\lineskip\thr@@\fontdimen8 \scriptfont\thr@@
\lineskiplimit\lineskip
\ialign{\hfil$\m@th\scriptstyle##$&$\m@th\scriptstyle{}##$\hfil\crcr
#1\crcr
}%
}%
}
\makeatother

\newcommand{\ps}{\frac{\partial}{\partial s}}

\newcommand{\pt}{\frac{\partial}{\partial t}}

\newcommand{\Dps}{\frac{D}{\partial s}}

\newcommand{\vol}{\mathrm{vol}}

\makeatletter
\newcommand{\newparallel}{\mathrel{\mathpalette\new@parallel\relax}}
\newcommand{\new@parallel}[2]{%
  \begingroup
  \sbox\z@{$#1T$}
  \resizebox{!}{\ht\z@}{\raisebox{\depth}{$\m@th#1/\mkern-5mu/$}}%
  \endgroup
}
\makeatother

\def\Xint#1{\mathchoice
{\XXint\displaystyle\textstyle{#1}}%
{\XXint\textstyle\scriptstyle{#1}}%
{\XXint\scriptstyle\scriptscriptstyle{#1}}%
{\XXint\scriptscriptstyle\scriptscriptstyle{#1}}%
\!\int}
\def\XXint#1#2#3{{\setbox0=\hbox{$#1{#2#3}{\int}$ }
\vcenter{\hbox{$#2#3$ }}\kern-.6\wd0}}

\def\dashint{\Xint-}

\definecolor{LB}{rgb}{0.29, 0.63, 0.73}

\def\second-form{ I\!\! I }

\providecommand{\keywords}[1]
{
  \small	
  \textbf{\textit{Keywords: }} #1
}

\begin{document}

\title{Coarse extrinsic curvature of Riemannian submanifolds} 
\author{Marc Arnaudon \thanks{Univ. Bordeaux, CNRS, Bordeaux INP, IMB, UMR 5251,
F-33400 Talence, France. \newline  \texttt{marc.arnaudon@math.u-bordeaux.fr}}, Xue-Mei Li \thanks{ Dept. of Maths, Imperial College London, U.K.  \& EPFL, Switzerland \newline  \texttt{xue-mei.li@imperial.ac.uk} or  \texttt{xue-mei.li@epfl.ch}}, Benedikt Petko\thanks{Dept. of Maths,  Imperial College London, U.K. \newline  \texttt{benedikt.petko15@imperial.ac.uk} or \texttt{benedikt.petko@gmail.com},\\ corresponding author}}
\date{\today}
\maketitle

\begin{abstract}

We introduce a novel concept of coarse extrinsic curvature for Riemannian submanifolds, inspired by Ollivier's notion of coarse Ricci curvature. This curvature is derived from the Wasserstein 1-distance between probability measures supported in the tubular neighborhood of a submanifold, providing new insights into the extrinsic curvature of isometrically embedded manifolds in Euclidean spaces. The framework also offers a method to approximate the mean curvature from statistical data, such as point clouds generated by a Poisson point process. This approach has potential applications in manifold learning and the study of metric embeddings, enabling the inference of geometric information from empirical data.
\end{abstract}


\tableofcontents

\section{Introduction}

 Synthetic lower bounds on Ricci curvature are a powerful tool in the study of classical geometric analysis and metric measure spaces. Ollivier's notion of Coarse Ricci Curvature is distinct in that it approximates the curvature itself, rather than merely providing a lower bound. By selecting as test measures weighted localized volume measures, supported on a ball of radius $\epsilon$, the Wasserstein 1-distance between two such measures reveals the generalized Ricci tensor;  applying to random geometric graphs sampled from a Poisson point process with non-uniform intensity leads to similar conclusions \cite{arnaudon2023coarse}.

Inspired by the concept of coarse Ricci curvature, in this article we seek a suitable notion of extrinsic curvature for embedded manifolds. With manifold learning applications in mind, we initially work with curves and surfaces and subsequently define a concept of coarse extrinsic curvature for general embedded manifolds. This notion captures the inner product between the mean curvature and the second fundamental form in a principal curvature direction. It may prove useful for studying embedded metric spaces and could be relevant in manifold learning contexts.

Let $M$ be a smooth manifold isometrically embedded in another Riemannian manifold. We propose a family of test measures $\{\mu^{\sigma, \varepsilon}_{x}, x\in M\}$, where $\sigma, \varepsilon$ are small parameters, whose `derivative' in the 1-Wasserstein distance with respect to variation of the point $x$ describes some kind of curvature.

This consideration leads to a novel concept of coarse curvature in the setting of Riemannian submanifolds. Within the applicable range of the parameters, we have an approximation of the mean curvature and the second fundamental form,  providing a valuable tool for evaluating these extrinsic curvatures.
In more practical applications, we can take test measures built from statistical data and simulations; for instance through the empirical measures of point cloud samples. There is scope for extending to metric embeddings of metric spaces.

In contrast to the intrinsic Riemannian curvature, which characterizes the geometry of a manifold independently of its embedding, the second fundamental form of submanifolds is an extrinsic concept. It provides a means for describing the shape of a submanifold in relation to its ambient space, offering views into its bending properties. For instance, a surface embedded in $\mathbb{R}^3$ is locally isometric to a plane if and only if its second fundamental form vanishes.

The extrinsic curvature of $M$, isometrically embedded in $N$, is expressed by the second fundamental form, which we recall to be defined as the bilinear form
\begin{equation}
\label{sff}
\II_x(w,w) := \nabla^N_w W(x) - \nabla^M_w W(x),
\end{equation}
where $W$ is an arbitrary vector field on $M$ with $W(x) = w$.  Letting $m$ denote the dimension of $M$,  the mean curvature is defined as the vector field
\begin{equation}
H(x) := \sum_{i=1}^m \nabla^N_{e_i} e_i(x) -\nabla^M_{e_i} e_i(x).
\end{equation}
Here $(e_i)_{i=1}^m$ is an arbitrary local orthonormal frame on a neighbourhood of $x$ in $M$. Note that we omit the factor of $\frac{1}{m}$ that usually appears in this definition in the literature in order to simplify the statement of our results. It is a standard fact that both $\II_x(w,w)$ and $H(x)$ are vectors which are perpendicular to the submanifold $M$. We refer to e.g. \cite[Chap. 5]{MR3726907} for a detailed treatment of these objects.
For instance, one of the examples we consider below is that of a planar curve $\gamma$ with radius of osculating circle $R(\alpha)$. A simple computation shows that in this case
$$
\|H(\gamma(\alpha))\| = \|\II_{\gamma(\alpha)}
(\dot{\gamma}(\alpha), \dot{\gamma}(\alpha))\| = \f{1}{R(\alpha)},
$$
where $\|\cdot\|$ is the Euclidean magnitude.

There exists a considerable body of literature on description of submanifold properties by tubular volume, of which we name a few representatives. The early work of Weyl \cite{MR1507388} proved the classical tube formula for submanifolds embedded in Euclidean spaces, which is an expansion with respect to the width of the tubular volume and its coefficients are geometric invariants of the submanifold.
Federer \cite{MR0110078} introduced the notion of boundary measures, which lead to generalization of the tube formula to compact subsets of Euclidean spaces.
More recent works of Chazal et al. \cite{MR2594445} \cite{MR3727578} studied geometric inference via point cloud approximations to boundary measures using Monte Carlo methods.
For a comprehensive treatment on properties of tubular neighbourhoods, we refer to the monograph \cite{gray-tubes}. The approach in our present work differs from the above in that it gives a local and directional information about the second fundamental form, and also the mean curvature.

Notions of synthetic Ricci curvature were motivated by the study of geometry of metric measure spaces and were pioneered by the seminal works \cite{ Sturm1, Lott-Villani,MR2502429}, see also the survey \cite{MR2408268}. In a metric measure space, a global lower bound on the synthetic Ricci curvature leads to properties of the metric measure space which are analogous to the Riemannian setting, such as the Poincar\'e and log-Sobolev inequalities, the concentration of measure phenomenon, and closure under measured Gromov-Hausdorff convergence \cite{MR2142879,MR3675945,Erbar-Maas}. We note also the related direction of the works \cite{MR2502429} \cite{Sturm21}.

To our understanding, there has not been a notion of a synthetic extrinsic curvature.  Our notion of coarse extrinsic curvature is inspired by coarse Ricci curvature of Ollivier \cite{MR2484937}, which is defined in the Riemannian setting through the expansion of the 1-Wasserstein distance of two uniform measures supported on geodesic balls of a small radius, the radius being the variable of expansion \cite[Example 7]{MR2484937}, see also the survey \cite{MR3060504}. This is different from the above mentioned synthetic Ricci curvature lower bounds in that it puts a precise number on the value of curvature at a point. Moreover, it can be applied to general metric spaces by choosing a family of measures indexed by points in the space for the evaluation of the 1-Wasserstein distance. Coarse Ricci curvature can be computed explicitly for a number of examples on graphs, where the measures are provided by a Markov chain. We adopt and modify Ollivier's approach to the submanifold setting by choosing suitable measures for the expansion of the 1-Wasserstein distance, showing that this yields a geometrically meaningful information.

As an immediate application of our result, we venture into the setting of \cite{hoorn-2023} and \cite{arnaudon2023coarse} to explore retrieval of curvature information from point clouds generated by a Poisson point process. In the first of the mentioned works, Hoorn et al. proved that Ollivier's coarse Ricci curvature of random geometric graphs sampled from a Poisson point process with increasing intensity on a Riemannian manifold converges in expectation at every point to the classical Ricci curvature of the manifold. This was extended in the second mentioned work to weighted Riemannian manifolds. In the present work, we show that coarse extrinsic curvature can recover the mean curvature in expectation at a point. In this case, it is not necessary to impose a graph structure to connect points of the sample.

In the context of deep learning, it is noteworthy that computational algorithms for effectively computing optimal transport maps have been proposed, as discussed in \cite{rout2022generativemodelingoptimaltransport, Gu-Yau}. Additionally, relevant work in the fields of manifold learning and inverse problems is worth mentioning. One particularly interesting inverse problem is whether an embedded manifold can be learned from a set of samples  $x_j+\xi_j$ where $x_j$ belongs to a submanifold $M\subset \R^{m+p}$ and $\xi_j$ are independent Gaussian random variables on $\R^{m+p}$.
 The reconstruction of embedded manifolds has been studied in \cite{Puchkin-Spokoiny-Trevisan, aigenbaum-Golovin-Levin, fefferman2021reconstructioninterpolationmanifoldsii, Fefferman-Ivanov-Kurylev-Lassas-Naraynan}. An algorithm for constructing an embedded submanifold is provided in \cite{fefferman2022fittingmanifoldlargereach}. Although manifold learning is still in its early stages, manifold approximation and reconstruction have a longer history, we point out some more recent publications on this topic \cite{Eilat-Klartag, genovese-Perone-Pacifico-Verdinelli, Boissonnat-Guibas-Oudot, Eilat-Klartag, Gold-Rosenberg}.

\subsection*{Main Results}

In our setting, $M$ is an $m$-dimensional compact Riemannian manifold embedded isometrically in a Euclidean space $\R^{m+k}$ and $M_\sigma$ is the local $\sigma$-tubular neighbourhood of $M$ in $\R^{m+k}$, defined for $\sigma$ sufficiently small as
$$
M_\sigma = \{x + v: x \in M, v \in T_xM^\perp, \|v\| \leq \sigma\}.
$$
For any compact subset $U \subset M$, the projection mapping from its tubular neighbourhood 
$$
\pi: U_\sigma \rightarrow U, \quad \pi(z) := \textrm{argmin}_{x \in U} \|z-x\|
$$
is well-defined for all $\sigma >0$ sufficiently small, with the same notation for $U_\sigma$ as above.

Denote by $\exp_{M,x}: T_xM \rightarrow M$ the exponential mapping in $M$ with base point $x$. Fix a point $x_0 \in M$, a unit tangent vector $v \in T_{x_0}M$ and denote $y := \exp_{M,x_0}(\delta v)$ for $\delta >0$. Fix a constant $\varepsilon_0 > 0$ smaller than the uniform injectivity radius of some fixed compact neighbourhood of $x_0$ in $M$. Assume $\delta, \varepsilon < \varepsilon_0 / 3$ so that $B_\varepsilon(x_0) \cup B_\varepsilon(y)$ lie within the uniform injectivity radius away from $x_0$ and assume $\sigma$ is small enough so that the projection $\pi$ is well-defined on the $\sigma$-tubular neighbourhood of the $\varepsilon_0$-geodesic ball at $x_0$ in $M$. These requirements on the parameters $\delta, \varepsilon, \sigma$ will henceforth be encapsulated in the assumption that they are "sufficiently small". This ensures that all locally defined maps are well-defined, in particular the projection map (smallness of $\sigma$) and the Fermi coordinates (smallness of $\delta$ and $\varepsilon$) used later on.

As our test measures, we choose the probability measures
$$
\mu^{\sigma, \varepsilon}_{x} (A)=\f{\mu(\pi^{-1}(B^M_\varepsilon(x)) \cap A \cap M_\sigma)}{\mu(\pi^{-1}(B^M_\varepsilon(x)) \cap M_\sigma)} \qquad \forall A \in \mathcal{B}(\R^{m+k}),
$$
where $B^M_\varepsilon(x)$ denotes the $\varepsilon$-geodesic ball at $x$ in $M$.
Note that these measures are supported on compact subsets of $M_\sigma$.  We seek to obtain the expansion of $W_1(\mu_{x_0}^{\sigma,\varepsilon}, \mu_y^{\sigma,\varepsilon})$ with respect to the parameters $\delta,\sigma$ and $\varepsilon$.

To relate the Wasserstein distance to the second fundamental form, we first localize to a tubular neighbourhood of a fixed open set on the submanifold. We expand the densities of the test measures in Fermi coordinates, and for the subsequent computations we rely on a crucial observation developed in \cref{approximate-transport-maps}:
if $T$ is an approximate transport map from $\mu_{x_0}^{\sigma,\varepsilon}$ to $\mu_y^{\sigma,\varepsilon}$, in a sense defined later, then $W_1(\mu_{x_0}^{\sigma,\varepsilon}, \mu_y^{\sigma,\varepsilon})$ is close to $W_1(\mu_{x_0}^{\sigma,\varepsilon}, T_*\mu_{x_0}^{\sigma,\varepsilon})$. The remaining task involves proposing a concrete approximate transport map, which is at the same time close enough to optimal.

When dealing with test measures on an embedded manifold, accounting for the effect of the bending of the submanifold in the ambient space becomes crucial. The proposed transport map is thus formulated in terms of the Fermi frame along $\gamma$, adapted to the submanifold $M$ in a way that separates tangent and normal coordinate directions at every point.

We give a rough outline of the proposed transport map, made precise in \cref{section-proposed-transport-map}. In terms of Fermi coordinates, if $\alpha =(\alpha_1,\ldots,\alpha_m)$ represent submanifold tangent directions with $\alpha_1$ being associated with the direction of $\gamma$, and if $\beta=(\beta_1,\ldots,\beta_k)$ represent the normal directions, an initial proposal informed by the circle example (\cref{section-circle-example}) was
$$
(\alpha, \beta) \mapsto (\delta-\alpha_1, \alpha_2, \dots, \alpha_m, \beta_1,\ldots,\beta_k).
$$
This can be construed as translation by $\delta$ in the direction of the first coordinate, together with reflection in the first coordinate.
From studying the planar curve example (\cref{planar-curve}), it turned out that an additional bending correction needs to be put on top of the $\beta$ components of the transport by adding terms involving the derivative of the mean curvature. Favourably, such a correction contributes to the final estimate of the Wasserstein distance only at the fourth order and higher, and hence does not interfere with the mean curvature term, which will appear at third order of the expansion. The test measures are first expressed in Fermi coordinates in \cref{section-fermi-coordinates}. The proposed transport map is then presented in \cref{section-proposed-transport-map}, where we prove that it is indeed an approximate transport map of degree $3$, i.e.
$$
\f{d(T_* \mu_{x_0}^{\sigma,\varepsilon})}{d\mu_y^{\sigma,\varepsilon}}(\phi(\alpha,\beta)) = 1 + O(\delta^3).
$$
This precision is sufficient for obtaining the 1-Wasserstein distance approximation (see \cref{approximate-transport-maps}):
$$
W_1(\mu_{x_0}^{\sigma,\varepsilon}, \mu_{y}^{\sigma,\varepsilon}) = W_1(\mu_{x_0}^{\sigma,\varepsilon}, T_* \mu_{x_0}^{\sigma,\varepsilon}) + O(\delta^4).
$$
From here the strategy is to construct a test function $f: B_{2\delta}(x_0) \subset \R^{m+k} \rightarrow \R$ with Lipschitz norm approximately $1$ and satisfying the estimate
$$
f(Tz) - f(z) = \|Tz - z\| + O(\delta^4) = O(\delta),
$$  which allows us to estimate the distance between the original measure and its transport by means of the relation  
$$
W_1(\mu_{x_0}^{\sigma,\varepsilon}, T_*\mu_{x_0}^{\sigma,\varepsilon}) = \int (f(Tz) - f(z)) d\mu_{x_0}^{\sigma,\varepsilon}(z) + O(\delta^4).
$$
On the whole, we find that the Wasserstein distance between the initial measure $\mu_{x_0}^{\sigma,\varepsilon}$ and the target measure $\mu_y^{\sigma,\varepsilon}$ is approximated by $\int_M \|Tz-z\| d\mu_{x_0}^{\sigma,\varepsilon}$ up to $O(\delta^4)$ (see \cref{test-function-gradient}), which is explicitly computable as an expansion in $\delta, \sigma$ and $\varepsilon$ with geometric quantities as coefficients.

Using the above tools, in Section 3 we thus compute the expansion of $W_1(\mu_{x_0}^{\sigma,\varepsilon}, \mu_y^{\sigma,\varepsilon})$, beginning with the case of a planar curve:

\begin{restatable*}{proposition}{planarcoarsecurvature}
Let $\gamma$ be a smooth unit speed curve in $\R^2$ such that $\gamma(0) = x_0$ and $\gamma(\delta)=y$. For all $\delta, \varepsilon, \sigma >0$ sufficiently small with $\sigma \vee \varepsilon \leq \f{\delta}{4}$, it holds that
$$
\begin{aligned}
W_1(\mu_{x_0}^{\sigma,\varepsilon}, \mu_{y}^{\sigma,\varepsilon})
&= \|x_0-y\| \left( 1-\f{\varepsilon^2}{6R^2} + \f{\sigma^2}{3R^2} \right) + O(\delta^4)
\end{aligned}
$$
where $R$ is the radius of the osculating circle of the curve at $x_0$.
\end{restatable*}

This expansion can be rearranged as
$$
\begin{aligned}
1- \f {W_1(\mu_{x_0}^{\sigma,\varepsilon}, \mu_{y}^{\sigma,\varepsilon})} { \|x_0-y\| }
&=\f{\varepsilon^2}{6R^2}   - \f{\sigma^2}{3R^2} + O(\delta^3).
\end{aligned}
$$
We refer to the quantity on the left as the coarse extrinsic curvature of $\gamma$ between $x_0$ and $y$ at scales $\sigma,\varepsilon$. A version of this result for spatial curves  is presented in \cref{space-curve}.
In \cref{surface-coarse-curvature}, we then proceed to study the case of coarse extrinsic curvature along a geodesic on a surface embedded in $\R^3$.

This work culminates with the most general form:

\begin{restatable*}{theorem}{coarseextrinsiccurvature}
\label{general-submanifold}
Let $M$ be an isometrically embedded submanifold of $\R^{m+k}$,and $\gamma$ a unit speed geodesic in $M$ such that $\gamma(0) = x_0$ and $\gamma(\delta)=y$. Let $(e_j)_{j=1}^m$ be an orthonormal basis of $T_{x_0}M$ with $e_1 = \dot{\gamma}(0)$ and assume that $\II_{x_0}(e_1,e_j) = \mathbf{0}$ for all $j=2,\ldots,m$.  Then for every $\sigma,\varepsilon, \delta >0$ sufficiently small with $\sigma \vee \varepsilon \leq \f{\delta}{4}$ it holds that
$$
\begin{aligned}
&W_1(\mu_{x_0}^{\sigma,\varepsilon}, \mu_y^{\sigma,\varepsilon}) = \|y-x_0\| \left(1 + \left(\f{\sigma^2}{k+2} - \f{\varepsilon^2}{2(m+2)} \right) \<\II_{x_0}(e_1,e_1), H(x_0)\> \right) + O(\delta^4).
\end{aligned}
$$
\end{restatable*}

The assumption on the second fundamental form is necessary for optimality of our proposed transport map up to sufficient order and can always be satisfied for submanifolds of codimension 1, in particular surfaces embedded in $\R^3$, by choosing the basis of principal curvature directions. Further commentary is provided in \cref{rmk:special-cases}.

To interpret such expansions in terms of mean curvature, we can remove the directionality of the above result caused by transport in the direction of $\gamma$. Denoting the square norm of the mean curvature vector as
$$
\|H(x_0)\|^2 = \sum_{i=1}^k \<H(x_0), \mathbf{n}_i(x_0)\>^2
$$
for an arbitrary orthonormal basis $(\mathbf{n}_i(x_0))_{i=1}^k$ of the normal space $T_{x_0}M^\perp \subset T_{x_0}N$, we deduce the following:

\begin{restatable*}{corollary}{coarsemeancurvature}
\label{removing-directionality}
Let $(e_j)_{j=1}^m$ be an orthonormal basis of $T_{x_0}M$, and for $j=1, \dots, m$,  let $y_j = \exp_{M, x_0}(\delta e_j)$. Assume that $\II_{x_0}(e_i,e_j) = 0$ for $i\neq j$. 
Then for all $\sigma,\varepsilon,\delta>0$ sufficiently small with $\sigma \vee \varepsilon \leq \frac{\delta}{4}$ it holds that
$$
\label{eq:removing-directionality}
\begin{aligned}
 \sum_{j=1}^m \left( 1- \f{W_1(\mu_{x_0}^{\sigma,\varepsilon}, \mu_{y_j}^{\sigma,\varepsilon})}{\|x_0-y_j\|}\right) &= \left( \f{\varepsilon^2}{2(m+2)} - \f{\sigma^2}{k+2} \right) \|H(x_0)\|^2 + O(\delta^3).
\end{aligned}
$$
\end{restatable*}

Observe that the left side of the equation is independent of the choice of orthonormal basis $(e_j)_{j=1}^m$ because the norm on the right side is basis-invariant. Moreover, the assumption on the second fundamental form always holds for submanifolds of codimension 1 (see \cref{rmk:special-cases}).

In \cref{point-cloud-curvature}, we deduce that the coarse extrinsic curvature of suitable test measures on Poisson point clouds sampled from the tubular neighbourhood retrieves the same extrinsic geometric information consistent with \cref{general-submanifold}.

One key ingredient in the proofs of the above theorems is the geometric approximate transport map introduced in \cref{proposed-transport-map}, defined by means of Fermi coordinates (as per \cref{definition-fermi-coordinates}) adapted to the submanifold. Test measures in these coordinates encode information about the second fundamental form of the submanifold. The proposed map is verified to be an approximate transport map between the test measures with sufficient order of accuracy, as specified and motivated in \cref{approximate-transport-maps}. The optimality up to fourth order is proved by choosing a concrete test function for the Wasserstein lower bound by the Kantorovich-Rubinstein duality.

In the resulting expansion of the Wasserstein distance, the second fundamental form at the fixed point $x_0$ appears at third order, and its derivatives appear at fourth and higher orders. As a consequence, information about the second fundamental form at a point can be retrieved in a suitably scaled limit of coarse curvature. Please see the discussion below for an example.

\subsubsection*{Discussion} 

We illustrate this work using the following prototypical example. Let $\gamma:(-\delta_0,\delta_0) \rightarrow \R^2$ be a smooth, unit speed planar curve, and $\mathbf{n}: (-\delta_0,\delta_0) \rightarrow \R^2$ a unit normal vector field along $\gamma$,  unique up to sign. Denote by $R(\alpha):= \f{1}{\|\ddot{\gamma}(\alpha)\|}$ the radius of the osculating circle at the point $\gamma(\alpha)$. To detect the extrinsic curvature at $x_0:=\gamma(0)$, captured here by $R(0)$, we define test probability measures centered at nearby points $y:=\gamma(\delta)$ indexed by $\delta >0$.

Denote $\mu$ the Lebesgue measure on $\R^2$, $M := \gamma((-\delta_0,\delta_0))$ as the image of the curve, and $M_{\sigma_0}$ as a small enough tubular neighbourhood of $M$ so that the orthogonal projection $\pi: M_{\sigma_0} \rightarrow M$ is well-defined. 
Denote
$$
B_{\sigma, \varepsilon}(y) := \{z \in \R^2: \|z-\pi(z)\| < \sigma, d_\gamma(y, \pi(z)) < \varepsilon\}
$$
where $d_\gamma$ is the distance along $\gamma$. Define for $\sigma, \varepsilon >0$ with $\sigma \vee \varepsilon \leq \f{\delta}{4}$, the Borel measure on $\R^2$,
$$
 \mu_{y}^{\sigma, \epsilon}(A) := \f{\mu(A \cap B_{\sigma, \varepsilon}(y))}{\mu(B_{\sigma, \varepsilon}(y))}.
$$
We compare these in 1-Wasserstein distance to the initial measure, i.e. when $\delta=0$ and is denoted $\mu_{x_0}^{\sigma, \varepsilon}$. The Wasserstein distance has the form:
$$
W_1(\mu_{x_0}^{\sigma, \epsilon}, \mu_y^{\sigma, \epsilon}) = \|x_0-y\| \left(1- \f{\varepsilon^2}{6R(0)^2} +\f{\sigma^2}{3R(0)^2}\right) + O(\delta^4).
$$
Rearranging this expansion yields
\begin{equation}
\label{planar-curve-expansion}
1-\f {W_1(\mu_{x_0}^{\sigma, \varepsilon}, \mu_y^{\sigma, \varepsilon})} {\|x_0-y\|} = \f{1}{R(0)^2} \left(\f{\varepsilon^2}{6}-\f{\sigma^2}{3}\right) + O(\delta^3).
\end{equation}

From this point, depending on the application, we may consider three different regimes for the parameters $\varepsilon$ and $\sigma$ as $y$ converges to $x_0$. We recall the asymptotic notation $\sigma = \Theta(\delta)$ means there exist $c,C,\delta_0 >0$ such that for all $\delta < \delta_0$,
$$
c \delta < \sigma(\delta) < C \delta,
$$
and $\sigma = o(\delta)$ means $\lim_{\delta \rightarrow 0} \f{\sigma(\delta)}{\delta} = 0$.
\begin{enumerate}
    \item $\lim_{\delta \to 0} \f{\varepsilon(\delta)}{\sigma(\delta)} = C \neq \sqrt{2}$ for some known constant $C >0$, i.e. the decay of both $\sigma$ and $\varepsilon$ is controlled. In this case,
    $$
    \frac{1}{R(0)^2} = \lim_{\delta \rightarrow 0} -\f{6}{(C^2-2)\sigma^2}
\left( 1- \f {W_1(\mu_{x_0}^{\sigma, \varepsilon}, 
\mu_y^{\sigma, \varepsilon})} {\|x_0-y\|} \right),
    $$
    \item $\sigma = \Theta(\delta)$ and $\varepsilon = o(\delta)$, i.e. the decay of $\sigma$ is controlled, while the parameter of support size $\varepsilon$ vanishes fast. In this case,
    $$
    \frac{1}{R(0)^2} = \lim_{\substack{\varepsilon = o(\sigma), \sigma = \Theta(\delta), \\ \delta \rightarrow 0}} -\f 3{\sigma^2}
\left( 1- \f {W_1(\mu_{x_0}^{\sigma, \varepsilon}, 
\mu_y^{\sigma, \varepsilon})} {\|x_0-y\|} \right),
    $$
    \item $\varepsilon = \Theta(\delta)$ and $\sigma = o(\delta)$, i.e. the decay of $\varepsilon$ is controlled, while the size of the tubular neighbourhood $\sigma$ vanishes fast. In this case,
    $$
    \frac{1}{R(0)^2}= \lim_{\substack{\sigma = o(\varepsilon), \varepsilon = \Theta(\delta),\\ \delta \rightarrow 0}} \f{6}{\varepsilon^2}
    \left( 1- \f {W_1(\mu_{x_0}^{\sigma, \varepsilon}, \mu_y^{\sigma, \varepsilon})} {\|x_0-y\|} \right).
    $$
\end{enumerate}
The requirements $\sigma = \Theta(\delta)$ and $\varepsilon = \Theta(\delta)$ in the respective cases are in place to ensure the remainder term $O(\delta^3)$ in \eqref{planar-curve-expansion} does not explode upon division by $\sigma^2$ (resp. $\varepsilon^2$) in the limit as $\delta \rightarrow 0$.

In light of the above discussion, we may define the coarse extrinsic curvature between $x_0$ and $y$ at scales $\varepsilon, \sigma >0$ as:
\begin{equation}
\label{coarse-extrinsic-curvature-def}
\kappa_{\sigma, \varepsilon}(x_0,y)  :=1-\f {W_1(\mu_{x_0}^{\sigma, \varepsilon}, \mu_y^{\sigma, \varepsilon})} {\|x_0-y\|}.
\end{equation}
This quantity can be estimated from point cloud data and used for geometric inference.

In summary, this work focuses on Riemannian submanifolds embedded isometrically in Euclidean spaces with the aim of producing a reasonable measurement for the bending energy. This bending energy can also be estimated from point clouds obtained from sampling. One of the novel ingredients is the construction of a test function for using the Kantorovich-Rubinstein duality to obtain a lower bound for the Wasserstein distance in this setting.

The outline of this work is as follows. In \cref{preliminaries}, we establish geometric preliminaries pertaining to the volumes of tubular neighbourhoods and present approximate transport maps as a novel tool for approximating the 1-Wasserstein distance. In \cref{section:computed-examples}, we give description of coarse extrinsic curvature for a planar curve, space curve and a 2-surface embedded in $\R^3$. The coarse extrinsic curvature of a general submanifold of arbitrary codimension is studied in \cref{section:riemannian-submanifolds}. We present several immediate corollaries to our results with practical applications in \cref{section:applications}. Although the cases of curves and surfaces in \cref{section:computed-examples} are just instances of the general result in \cref{section:riemannian-submanifolds}, they  provide value in understanding this general case. Sections 3 and 4 can be read separately after reading Section 2, which contains all preliminaries.

\section{Preliminaries}
\label{preliminaries}
We prove a formula for volume growth of tubular neighbourhoods of submanifolds, leading to a disintegration of the ambient volume measure adapted to the submanifold. This formula is subsequently utilized  to derive explicit formulas for such disintegration in Fermi coordinates, considering cases such as a planar curve, space curve, and a surface in \cref{section:computed-examples}, and general Riemannian submanifolds in \cref{section:riemannian-submanifolds}.

Following the geometric preliminaries, we introduce the notion of an approximate transport map, enabling the computation of Wasserstein distances up to a sufficiently high degree of error.  Subsequently, we define the test measures to be transported and their representation in Fermi coordinates. Finally, we propose a transport map to evaluate the Wasserstein distance of these test measures.

\subsection{Ambient volume disintegration}
We begin with a simple lemma on evolution of probability densities. We denote $\mathcal{P}(\mathcal{X})$ as the space of probability measures  on a measurable space $(\mathcal{X}, \mathcal{M})$. The notation $\mu \ll \nu$ denotes the fact that the measure $\mu$ is absolutely continuous with respect to the measure~$\nu$.

\begin{lemma}
\label{density-dynamics-lemma}
    Consider  $(\mu_t)_{t \geq 0} \subset \mathcal{P}(\mathcal{X})$ such that $\mu_t \ll \mu_s$ for all $s\leq t$. Let $h_t: \mathcal{X} \rightarrow \mathbb{R}, t \ge 0,$ be a family of functions with $t \mapsto h_t(x)$  locally integrable, and such that $\left.\frac{d}{ds}\right|_{s=0}\frac{d\mu_{t+s}}{d\mu_t}(x) = h_t(x)$ for every $t \geq 0$.  Then
    $$
    \frac{d\mu_t}{d\mu_0}(x) = e^{\int_0^t h_s(x)ds}.
    $$
\end{lemma}

\begin{proof}
    The change of density at any $t \geq 0$ satisfies
    $$
    \left.\frac{d}{ds}\right|_{s=0}\frac{d\mu_{t+s}}{d\mu_t}(x) = \left.\frac{d}{ds}\right|_{s=0}\frac{\frac{d\mu_{t+s}}{d\mu_0}(x)}{\frac{d\mu_t}{d\mu_0}(x)}=h_t(x)
    $$
    implying
    $$
    \left.\frac{d}{ds}\right|_{s=0}\frac{d\mu_{t+s}}{d\mu_0}(x) = h_t(x) \frac{d\mu_t}{d\mu_0}(x)
    $$
    which has the unique solution $\frac{d\mu_t}{d\mu_0}(x) = e^{\int_0^t h_s(x)ds}$ by standard ODE theory.
\end{proof}

\begin{notation}
\label{notation-submanifold}
Throughout this article,  $M$ is a compact Riemannian manifold of dimension $m$,  isometrically immersed in a Riemannian manifold $N$ of dimension $n $. Set $k:=n-m$. Let $\sigma_0>0$ be a fixed number smaller than half the reach of $M$ in $N$. The reach is defined as the maximal number $r$ such that each point within a distance $r$ from $M$ has a unique orthogonal projection to $M$,  $\pi: M_{\sigma_0} \rightarrow M$. The projection map is well defined within the `reach'.
 
Let $U \subset M$ be a sufficiently small open neighbourhood such that there exists an orthonormal frame of unit normal vector fields $(\mathbf{n}_1, \ldots, \mathbf{n}_k)$ on $U$ and a one-parameter family of vector fields $\{(e_i(s))_{i=1}^{m}: s\in (-\sigma_0, \sigma_0)\}$ such that $(e_i(s))_{i=1}^m$ is an orthonormal frame on $\psi_s(U)$ for every $s \in (-\sigma_0,\sigma_0)$, and $s \mapsto e_i(s)$ is smooth for every $i=1,\ldots,m$. The latter can be constructed by taking the pushforward of an arbitrary initial orthonormal frame by $\psi_s$, denoted by $(D_{e_i(0)} \psi_s)_{i=1}^m$, and applying the Gram-Schmidt orthonormalization procedure.
\end{notation}

\begin{definition}
Let $\mathbf{n}\in \Gamma(TU^\perp)$ be a unit normal vector field, and define the normal flow $\psi: M \times (-2\sigma_0, 2\sigma_0)  \rightarrow N$ by
$$
\psi_t(x) := \exp_{N,x} (t\mathbf{n}(x))
$$
where $\exp_{N,x}: T_xN \rightarrow N$ denotes the exponential mapping on $N$. Denote by $\newparallel_t^N$ the parallel transport with respect to the Levi-Civita connection $\nabla^N$ along $t \mapsto \psi_t(x)$ for a fixed $x\in M$, and note that $\f{\partial}{\partial t} \psi_t(x) = \; \newparallel_t^N \mathbf{n}(x)$.

For every $t \in (-2\sigma_0, 2\sigma_0)$, $\psi_t$ is a diffeomorphism onto its image, and $t\mapsto \psi_t(x)$ is smooth with non-vanishing derivative for every $x\in M$. Equip every $\psi_t(M)$ with the Riemannian metric inherited from the ambient space. The mean curvature of the leaf $\psi_s(U)$ is then given by
\begin{equation}
\label{mean-curvature}
H(\psi_s(x)) = \sum_{i=1}^m \nabla^N_{e_i} e_i(\psi_s(x)) - \nabla^M_{e_i} e_i(\psi_s(x)).
\end{equation}
In particular, for any unit normal vector field $\mathbf{n}$ on $U$, 
\begin{equation}
\<H(\psi_s(x)), \newparallel_s^N \mathbf{n}(x)\> = \sum_{i=1}^m \<\newparallel^N_s \mathbf{n}(x), \nabla^N_{e_i} e_i(\psi_s(x))\>.
\end{equation}
\end{definition}

The following lemma shows $\newparallel_t^N \mathbf{n}$ stays normal to the leaves $\psi_t(U)$ as $t$ changes.

\begin{lemma}
\label{normal-to-leaf}
    The vector field $\newparallel_t^N \mathbf{n}(x)$ is normal to $\psi_t(U)$ for every $t \in (-\sigma_0, \sigma_0)$, i.e. $\<D_{e_i} \psi_t, \newparallel^N_t \mathbf{n}\> = 0$ for any local tangent frame $(e_i)_{i=1}^m$ on $M$.
\end{lemma}
\begin{proof}
    For every $t \in (-\sigma_0, \sigma_0)$, $\psi_t$ being a diffeomorphism implies that if $(e_i)_{i=1}^m$ is a frame on $U$, then $(D_{e_i} \psi_t)_{i=1}^m$ is a frame on $\psi_t(U)$, not necessarily orthonormal. Then
    $$
    \begin{aligned}
        \f{d}{dt} \< D_{e_i} \psi_t, \newparallel_t^N \mathbf{n} \> &= \<\f{D^N}{\partial t} D_{e_i} \psi_t, \newparallel_t^N \mathbf{n}\> \\
        &= \< \nabla^N_{e_i} \f{\partial}{\partial t} \psi_t, \newparallel_t^N \mathbf{n} \> \\
        &= \< \nabla^N_{e_i} \newparallel_t^N \mathbf{n}, \newparallel_t^N \mathbf{n}\> \\
        &=\f{1}{2} D_{e_i} \< \newparallel_t^N \mathbf{n}, \newparallel_t^N \mathbf{n} \> \\
        &= \frac{1}{2}D_{e_i} 1 = 0.
    \end{aligned}
    $$
    where on the second line we used that $\f{\partial}{\partial t} \psi_t(x) = \; \newparallel_t^N \mathbf{n}(x)$.
    The initial condition $\psi_0 = \textrm{id}$ gives $\< D_{e_i} \psi_0, \mathbf{n} \> = \< e_i, \mathbf{n} \> = 0$, so we may conclude that $\newparallel_t^N \mathbf{n}$ is normal to all tangent directions on $\psi_t(U)$ for all $t$.
\end{proof}

The action of push-forwards of volume forms on any orthonormal basis of tangent vectors is characterized by the determinant of the mapping which we make precise below. Let $M_1, M_2$ be Riemannian manifolds of the same dimension $m$, $\psi: M_1 \rightarrow M_2$ a diffeomorphism, $(e_i)_{i=1}^m$ an orthonormal frame on an open set $U_1 \subset M_1$ and $(\tilde{e}_i)_{i=1}^m$ an orthonormal frame on an open set $U_2 \subset M_2$, and $(e^i)_{i=1}^m, (\tilde{e}^i)_{i=1}^m$ the corresponding coframes characterized by $e^i(e_j) = \delta^i_j, \tilde{e}^i(\tilde{e}_j) = \delta^i_j$. Below, by the determinant of $D\psi^{-1} (x) : T_x M_2\to T_{\Psi^{-1}(x)}M_1$ we mean that of the matrix representing the map in these bases:
$$\det D\psi^{-1}=\sum_{\sigma \in S_m} \prod_{i=1}^m \sgn(\sigma)  e^i(\psi^{-1}_* \tilde{e}_{\sigma(i)}) .$$
By the rules of differential forms acting on tangent vectors, $\forall x \in U_2$,
$$
\begin{aligned}
\psi_* (e^1 \wedge \ldots \wedge e^m)(x)(\tilde{e}_1(x),\ldots, \tilde{e}_m(x))
&= (e^1 (x)\wedge \ldots \wedge e^m(x))(\psi^{-1}_* \tilde{e}_1(x),\ldots, \psi^{-1}_* \tilde{e}_m(x)) \\
&= \sum_{\sigma \in S_m} \prod_{i=1}^m \sgn(\sigma)  e^i(x)(\psi^{-1}_* \tilde{e}_{\sigma(i)(x)})  \\
&= \det D\psi^{-1}.
\end{aligned}
$$
Since linear maps are determined by their values on basis vectors, we may deduce
\begin{equation}\label{change-of-variable-forms}
\psi_* (e^1 \wedge \ldots \wedge e^m)(x) = \det D\psi^{-1}(x) \tilde{e}^1 \wedge \ldots \wedge \tilde{e}^m (x).
\end{equation}
With the above notation we return to the exponential map $\psi_t(x) := \exp_{N,x} (t\mathbf{n}(x))$.

\begin{proposition}[Change of volume]
\label{volume-growth}
For every $t \in (-\sigma_0, \sigma_0)$,
\begin{equation}
\label{flow-determinant}
\det D \psi_t(x) = \exp\left(-\int_0^t \<H(\psi_s(x)), \newparallel^N_s \mathbf{n}(x)\>  ds \right),
\end{equation}
and hence the volume of the image of any Borel measurable $A \subset U$ can be expressed as
\begin{equation}
\label{pushforward-volume}
\textrm{vol}_{\psi_t(M)}(\psi_t(A)) = \int_A \exp\left(-\int_0^t \<H(\psi_s(x)), \newparallel^N_s \mathbf{n}(x)\> ds \right) d\vol_M(x)
\end{equation}
where $\vol_{\psi_t(M)}$ is the Riemannian volume on $\psi_t(M)$.
\end{proposition}

\begin{proof}
First, we extend the map $\psi: (-\sigma_0,\sigma_0) \times U \rightarrow N$ to $\tilde{\psi}: (-\sigma_0,\sigma_0) \times U_{\sigma_0} \rightarrow N$ by the flow condition
$$
\forall s,t \in (-\sigma_0,\sigma_0): \tilde{\psi_s}(\psi_t(x)) := \psi_{t+s}(x)
$$
for all $s$ and $t$ in $(-\sigma_0,\sigma_0)$.
This determines $\tilde{\psi}$ uniquely because $\{\psi_t(U)\}_{t \in (-\sigma_0,\sigma_0)}$ is a foliation of the tubular neighbourhood $U_{\sigma_0}$. Then on every leaf $\psi_t(U)$ of the foliation, we have $\tilde{\psi}_0 = \textrm{id}$. If $(e^i)_{i=1}^m$ and $(\tilde{e}^i)_{i=1}^m$ are orthonormal coframes on $\psi_{t+s}(U)$ and $\psi_t(U)$ respectively, the change of variable formula for volume forms \eqref{change-of-variable-forms} states that
$$
(\tilde{\psi}_{-s})_* e^1 \wedge \ldots \wedge e^n = (\det D \tilde{\psi}_s) \tilde{e}^1 \wedge \ldots \wedge \tilde{e}^n.
$$
Then for every $A \subset U$ Borel measurable and $s \in (-\sigma_0,\sigma_0)$,
$$
\begin{aligned}
\vol_{\psi_{t+s}(M)}(\psi_{t+s}(A)) &= \vol_{\psi_{t+s}(M)}(\tilde{\psi}_s(\psi_{t}(A))) \\
&=  ((\tilde{\psi}_{-s})_* \vol_{\psi_{t+s}(M)}) (\psi_t(A)) \\
&= \int_{\psi_t(A)} \det D \tilde{\psi}_s(x) d \vol_{\psi_t(M)}(x)
\end{aligned}
$$
using respectively the flow property, definition of the push-forward of a measure, and the change of variable formula with $\vol_{\psi_{t+s}(M)}= e^1 \wedge \ldots \wedge e^m$ and $\vol_{\psi_{t}(M)} = \tilde{e}^1 \wedge \ldots \wedge \tilde{e}^m$.

Denoting by $\f{D}{\partial s}$ the covariant derivative along $s \mapsto \tilde{\psi}_s$, the Jacobi formula for the derivative of determinants gives
$$
\begin{aligned}
\left.\ps\right|_{s=0}& \textrm{vol}(\psi_{t+s}(A)) = \int_{\psi_t(A)} \left.\ps\right|_{s=0} \det D \tilde{\psi}_s(x) d \vol_{\psi_t(M)}(x)\\
&= \int_{\psi_t(A)} \textrm{trace}\left(\left(\left.\ps\right|_{s=0} \<D_{e_i(t)} \tilde{\psi}_s(x),e_j(t+s)(x) \> \right)_{i,j=1,\ldots,m} \right) d \vol_{\psi_t(M)}(x)\\
&= \sum_{i=1}^m \int_{\psi_t(A)} \< \left.\Dps\right|_{s=0} D_{e_i(t)} \tilde{\psi}_s (x), e_i(t)(x)\> d \vol_{\psi_t(M)}(x)\\
&= \sum_{i=1}^m \int_{\psi_t(A)} \< \nabla^N_{e_i(t)} \left.\ps\right|_{s=0} \tilde{\psi}_s (x), e_i(t)(x)\> d \vol_{\psi_t(M)}(x)\\
&=-\sum_{i=1}^m \int_{\psi_t(A)} \< \left.\ps\right|_{s=0} \tilde{\psi}_{s}(x), \nabla^N_{e_i(t)} e_i(t)(x) \> d\textrm{vol}_{\psi_t(M)}(x) \\
&=-\sum_{i=1}^m \int_{A} \< \left.\ps\right|_{s=0} \psi_{t+s}(x), \nabla^N_{e_i(t)} e_i(t)(\psi_t(x)) \> d(\psi_t^{-1})_*\textrm{vol}_{\psi_t(M)}(x).
\end{aligned}
$$
 From the second to third line, the other term coming from the product rule applied on the bracket does not contribute to the trace, because for $i=j$,
$$
\<D_{e_i(t)} \tilde{\psi}_0, \f{D}{dt} e_i(t) \>=  \<e_i(t), \f{D}{dt} e_i(t) \> = \f{1}{2} \pt \<e_i(t), e_i(t)\> = 0
$$
using that $\tilde{\psi}_0(x)=x$ so $D \tilde{\psi}_0 = \textrm{id}$.
From the fourth to fifth line, we used normality of the flow $\<\left.\ps\right|_{s=0} \tilde{\psi}_s(x), e_i(t)(x)\> = 0$,
and on the last line applied $\ps \tilde{\psi}_s(\psi_t(x)) = \ps \psi_{t+s}(x)$ from definition of the extension $\tilde{\psi}$, before pulling the integral from $\psi_t(A)$ back to $A$.

Hence the evolved volume measure pulled back to $U$ satisfies the dynamics
\begin{equation}
\label{density-dynamics}
\begin{aligned}
\left.\ps\right|_{s=0} \f{d(\psi_{t+s}^{-1})_* \vol_{\psi_{t+s}(M)}}{d (\psi_t^{-1})_*\vol_{\psi_t(M)}}(x) &= - \sum_{i=1}^m \<\pt \psi_t(x), \nabla^N_{e_i(t)} e_i(t)(\psi_t(x))\> \\
&= -\< H(\psi_t(x)), \newparallel^N_t \mathbf{n}(x)\>,
\end{aligned}
\end{equation}
which together with the initial condition $d(\psi_0)^{-1}_*\vol_{\psi_0(M)} = d\vol_M$ implies
$$
\det D \psi_t(x) = \frac{d(\psi_{t}^{-1})_* \vol_{\psi_{t}(M)}}{d\textrm{vol}_M}(x) = \exp\left(-\int_0^t \<H(\psi_s(x)), \newparallel^N_s \mathbf{n}(x)\> ds\right)
$$
by \cref{density-dynamics-lemma}, setting $h_t$ to be the right-hand side of \eqref{density-dynamics}. Equation \cref{pushforward-volume} is then simply the change of variable formula for the map $\psi_t$.

\end{proof}

\begin{remark}
The formula of \cref{volume-growth} can be extended from the neighbourhood $U$ to all of $M$ by a partition of unity argument, nonetheless the local formulation is sufficient for our purpose.
\end{remark}

We proceed to derive a disintegration of the ambient volume measure adapted to a submanifold of arbitrary codimension, at the cost of specialising to the case $N = \mathbb{R}^n$. Note that the covariant derivative $\nabla^{\R^n}$ then becomes the plain derivative denoted by $D$.
\begin{notation}
\label{psi-notation}
Let $(\mathbf{n}_j)_{j=1}^k$ be a local orthonormal frame for $TM^\perp$ on $U$ and denote $B_{\sigma_0}^k \subset \mathbb{R}^k$ the centered Euclidean ball of radius $\sigma_0$. Define the map
\begin{equation}
\label{psi-map}
\begin{aligned}
\psi: U &\times B_{\sigma_0}^k  \rightarrow \pi^{-1}(U) \subset \mathbb{R}^n\\
\psi(x, \beta) &= x+ \sum_{j=1}^k \beta_j \mathbf{n}_j(x).
\end{aligned}
\end{equation}
which gives the $k$-dimensional foliation $\{\psi(U, \beta) : \beta \in B_{\sigma_0}^k\}$ of $\pi^{-1}(U)$ with leaves of dimension $m$.
Extend $(\mathbf{n}_j)_{j=1}^k$ and $(e_i)_{i=1}^m$ smoothly to $\pi^{-1}(U)$ so that the restrictions to the submanifold $\psi(U,\beta)$ are an orthonormal frame in the tangent space and the normal space, respectively, for every $\beta \in \tilde{B}_{\sigma_0}^k$.

Denote $\mathbf{n}(x,\beta) = \sum_{j=1}^k \f{\beta_j \mathbf{n}_j(x)}{\|\beta\|_2}$ which was shown in \cref{normal-to-leaf} to be normal to each leaf $\psi(U,\beta)$. Then the mean curvature of $\psi(U,\beta)$ in the direction $\mathbf{n}(x,\beta)$ is
$$
\<H(\psi(x, \beta)), \mathbf{n}(x,\beta)\> = \<\mathbf{n}(x,\beta), \sum_{i=1}^m \nabla^{\R^{m+k}}_{e_i} e_i (\psi(x, \beta)) \>.
$$
Denote also the components of mean curvature in each of the directions of the normal frame,
$$
H^j(\psi(x,\beta)) := \<\mathbf{n}_j(x), H(\psi(x,\beta)) \> =  \<\mathbf{n}_j(x), \sum_{i=1}^m \nabla^{\R^{m+k}}_{e_i} e_i(\psi(x,\beta)) \>
$$
so that
$$
\|\beta\| \< H(\psi(x, \beta)), \mathbf{n}(x,\beta)\> = \sum_{j=1}^k \beta_j H^j(\psi(x,\beta)).
$$
\end{notation}

\begin{remark}
The collection of submanifolds $\{\psi(U,\beta): \beta \in \tilde{B}_{\sigma_0}^k\}$ is indeed a foliation of $\pi^{-1}(U) \subset \R^{m+k}$ (see e.g. the definition of foliation in \cite{lee-smooth-manifolds}). The leaves are disjoint submanifolds of dimension $m$. Defining
$$
F: (p_1,\ldots,p_m, \beta_1,\ldots, \beta_k) \mapsto \xi(p_1,\ldots,p_m) + \sum_{j=1}^k \beta_j \mathbf{n}_j(\xi(p_1,\ldots,p_m))
$$
where $\xi:O \subset \R^m \rightarrow U$ is an arbitrary chart on $U$, we have by definition that
$$
\psi(U,\beta) = F(\{\beta\}),
$$
so each leaf is a level set of $F$ and thus $F$ is a flat chart for the foliation.
\end{remark}

\begin{proposition}[Disintegration]
\label{ambient-volume-disintegration}
The ambient volume measure on $\pi^{-1}(U) \subset \mathbb{R}^n$ disintegrates with respect to the submanifold and the normal frame $(\mathbf{n}_j)_{j=1}^k$ as
\begin{equation}
\label{skew-decomposition-general}
\vol_{\R^n}(A) = \int_U \vol_M(dx) \int_{\tilde{B}_{\sigma_0}^k} d\beta_1 \ldots d\beta_k \mathbbm{1}_{A}(\psi(x, \beta)) e^{-\int_0^{1} \sum_{j=1}^k \beta_j H^j(\psi(x,s\beta)) ds}
\end{equation}
for any Borel measurable set $A \in \mathcal{B}(\pi^{-1}(U))$ in the tubular neighbourhood.
\end{proposition}
\begin{proof}
We apply the change of coordinates by the map defined by \eqref{psi-map}, which at every $(x,\beta)$ has block-triangular derivative with respect to the orthonormal bases 
$$
(e_1(x), \ldots, e_m(x), \partial_{\beta_1}(x), \ldots, \partial_{\beta_k}(x))
$$
and
$$
(e_1(\psi(x,\beta)), \ldots, e_m(\psi(x,\beta)), \mathbf{n}_1(x), \ldots, \mathbf{n}_k(x))
$$
in the domain and codomain respectively, since
$$
\< \partial_{\beta_i} \psi(x,\beta), e_j(x,\beta) \> = \< \mathbf{n}_i(x), e_j(\psi(x,\beta)) \> = 0.
$$
Hence the determinant can be computed as
$$
\begin{aligned}
\det D\psi &= \det \left(\< D_{e_i} \psi, e_j \>\right) \det ( \langle \partial_{\beta_i} \psi, \mathbf{n}_j \rangle ) \\
&= \det \left(\< D_{e_i} \psi, e_j \>\right) \det(\delta^i_j) \\
&= \det \left(\< D_{e_i} \psi, e_j \>\right)
\end{aligned}
$$
for which we have the right-hand side of \eqref{flow-determinant}.\\
Denoting $(e^i)_{i=1}^m$, $(\mathbf{n}^i)_{i=1}^k$ the coframes characterized by $e^i(e_j) = \delta^i_j$ and $\mathbf{n}^i(\mathbf{n}_j) = \delta^i_j$,
    $$
    \begin{aligned}
        & \vol_{\R^n}(A) = \int_{\pi^{-1}(U)} \mathbbm{1}_A(z) e^1 \wedge \ldots \wedge e^m \wedge \mathbf{n}^1 \wedge \ldots \wedge\mathbf{n}^k (z) \\
        &= \int_U e^1 \wedge \ldots \wedge e^m(x)  \int_{B_{\sigma_0}^k} d\beta_1 \ldots d\beta_k \mathbbm{1}_A(\psi(x,\beta)) |\det D\psi(x, \beta)|  \\
        &= \int_U \vol_M(dx) \int_{B_{\sigma_0}^k} d\beta_1 \ldots d\beta_k  \mathbbm{1}_A(\psi(x,\beta)) \exp\left(-\int_0^{\|\beta\|_2} \< H(\psi(x,\f{s\beta}{\|\beta\|}), \mathbf{n}(x,\f{s\beta}{\|\beta\|})\> ds \right)  
    \end{aligned}
    $$
    on the second line using the change of variable formula and on the third line plugging in the determinant expression \eqref{flow-determinant} with $\mathbf{n}(x, \beta) = \sum_{j=1}^k \f{\beta_j \mathbf{n}_j(x)}{\|\beta\|}$ and $t = \|\beta\|_2$. The final expression is obtained by the substitution $s' = \f{s}{\|\beta\|}$ so that
    $$
    \begin{aligned}
    \int_0^{\|\beta\|_2} \< H(\psi(x,\f{s\beta}{\|\beta\|}), \mathbf{n}(x,\f{s\beta}{\|\beta\|})\> ds &= \int_0^1  \|\beta\| \< H(\psi(x,s'\beta)), \mathbf{n}(x, s'\beta)\> ds'\\
    &= \int_0^1 \sum_{j=1}^k \beta_j H^j(\psi(x,s'\beta))ds'. 
    \end{aligned}
    $$
\end{proof}

\begin{corollary}[Codimension 1]
\label{volume-disintegration}
If $M$ has codimension 1 then the ambient volume measure on $\pi^{-1}(U)$ can be written in terms of the disintegration
\begin{equation}
\label{skew-decomposition-hypersurface}
\vol_{\R^n}(A) = \int_{U} \int_{-\sigma}^{\sigma} \mathbbm{1}_A(\psi(x,\beta)) e^{-\int_0^\beta \< H(\psi(x, \beta'), \mathbf{n}(x,\beta')\> d\beta'} d\beta\; \vol_M(dx), \; \forall A \in \mathcal{B}(\pi^{-1}(U)),
\end{equation}
where $\vol_M(dx)$ is the volume measure of the submanifold $M$ and $H(\psi(\cdot, \beta))$ is the mean curvature on the Riemannian submanifold $\psi(U,\beta)$.
\end{corollary}

\subsection{Approximate transport maps}
\label{approximate-transport-maps}

In the sequel, we work with transport maps which are only optimal up to sufficiently high degree for asymptotically small diameter of support of the test measures. We present a result which justifies the use of such transport maps.

Let $\mathcal{X}$ be a Polish space, $\mathcal{P}(\mathcal{X})$ the set of probability measures, define
$$
\mathcal{P}_1(\mathcal{X}) := \{\mu \in \mathcal{P}(\mathcal{X}) : \exists \; x_0 \in \mathcal{X} \textrm{ such that }\int_\mathcal{X} d(x_0,x) \mu(dx) <\infty\}
$$
and consider two families of probability measures $(\mu_1^{\delta})_{\delta \geq 0}$, $(\mu_2^{\delta})_{\delta \geq 0} \subset \mathcal{P}_1(\mathcal{X})$. 

\begin{lemma}[$W_1$ distance approximation]
\label{w1-distance-approximation}
If $\textrm{diam supp} \; \mu_2^\delta = O(\delta^\ell)$
and $\mu_1^{\delta} \ll \mu_2^\delta $ for every $\delta \geq 0$ with the density satisfying
$\sup_{x \in \textrm{supp }\mu_2} \frac{d\mu_1^\delta}{d\mu_2^\delta}(x) = 1+O(\delta^k)$, then
$$
\sup_{\mu \in \mathcal{P}_1(\mathcal{X})} |W_1(\mu, \mu_1^\delta) - W_1(\mu, \mu_2^\delta)| = O(\delta^{k+\ell}).
$$
\end{lemma}
\begin{proof}
By reverse triangle inequality and Kantorovich-Rubinstein duality, $\forall \mu \in \mathcal{P}_1(\mathcal{X})$:
    $$
    \begin{aligned}
    |W_1(\mu, \mu_1^\delta) - & W_1(\mu, \mu_2^\delta)| \leq W_1(\mu_1^\delta, \mu_2^\delta) \\
    &= \sup_{f \in \textrm{Lip}_1(\mathcal{X})} \int_\mathcal{X} f(x) (d\mu_1^\delta(x) - d\mu_2^\delta(x)) \\
    &= \sup_{f \in \textrm{Lip}_1(\mathcal{X})} \int_\mathcal{X} f(x) \left( \f{d\mu_1^\delta}{d\mu_2^\delta}(x)d\mu_2^\delta(x) -d\mu_2^\delta(x) \right) \\
    &= \sup_{f \in \textrm{Lip}_1(\mathcal{X})} \int_\mathcal{X} f(x) \left(\f{d\mu_1^\delta}{d\mu_2^\delta}(x)-1\right) d\mu_2^\delta(x) \\
    &=\sup_{f \in \textrm{Lip}_1(\mathcal{X})} \int_\mathcal{X} (f(x)-f(x_0)) O(\delta^k) d\mu_2^\delta(x) = O(\delta^{k + \ell}),
    \end{aligned}
    $$
    where $x_0 \in \textrm{supp } \mu_2^\delta$ is arbitrary.
    On the last line, we introduced the term
    $$
    \int_\mathcal{X} f(x_0) \left(\f{d\mu_1^\delta}{d\mu_2^\delta}(x)-1\right) d\mu_2(x) = 0,
    $$
    because $f(x_0)$ is constant and the density integrates to 1, and then used the 1-Lipschitz property of $f$ together with the $O(\delta^\ell)$ bound on the diameter of the support of $\mu_2$.
\end{proof}

Let $(\mu^\delta)_{\delta \geq 0} \subset \mathcal{P}_1(\mathcal{X})$ be another family of probability measures.

\begin{definition}[Approximate transport map]
\label{approximate-transport}
    A measurable map $T^\delta : \mathcal{X} \rightarrow \mathcal{X}$ is said to be an approximate transport from $\mu^\delta$ to $\mu_2^\delta$ with degree $k$ if $T^\delta_* \mu^\delta \ll \mu_2$ and the density satisfies
    $$
    \sup_{x \in \textrm{supp } \mu_2} \frac{d(T_*^\delta \mu^\delta)}{d\mu_2^\delta}(x) = 1+ O(\delta^k).
    $$
\end{definition}

\begin{corollary}
\label{asymptotically-optimal-transport}
    If $T^\delta : \mathcal{X} \rightarrow \mathcal{X}$ is an approximate transport map from $\mu^\delta$ to $\mu_2^\delta$ with degree $k$ and $\textrm{diam supp } \mu_2^\delta = O(\delta^\ell)$ then
    $$
    W_1(\mu^\delta, \mu_2^\delta) = W_1(\mu^\delta, T_*^\delta \mu^\delta) + O(\delta^{k+\ell}).
    $$
\end{corollary}

\begin{proof}
    Set $\mu_1^\delta := T_*^\delta \mu^\delta$ and apply the previous lemma.
\end{proof}

\subsection{Test measures in Fermi coordinates}
\label{section-fermi-coordinates}
Let $(M,g)$ be a Riemannian submanifold of codimension $k$ in $\R^{m+k}$ and $U \subset M$ an open neighbourhood of a point $x_0 \in M$ as in \cref{notation-submanifold}. 
The Fermi coordinates are a suitable tool for explicit computations and will be used throughout the rest of this work. The following is a modification of classical Fermi coordinates to the submanifold setting.

\begin{definition}[Fermi coordinates]
\label{definition-fermi-coordinates}
Let $\gamma: (-\delta_0,\delta_0) \rightarrow M$ be a unit speed geodesic with $\delta_0 >0$ small enough for $\gamma$ to be contained in $U$, $\varepsilon_0$ the uniform injectivity radius in $M$ along $\gamma$ and $\sigma_0$ smaller than half the reach of $U$ in $\R^{m+k}$. Let $(e_i)_{i=1}^m$ be an orthonormal frame for the fibres of $TM$ along $\gamma$ such that $e_1(\alpha_1) = \dot{\gamma}(\alpha_1)$ and $\nabla^M_{\dot{\gamma}} e_i (\alpha_1) = 0$ for $i=1,\ldots,m$ and every $\alpha_1 \in (-\delta_0,\delta_0)$. Also let $(\mathbf{n}_i)_{i=1}^k$ be a local orthonormal frame for fibres of the normal bundle $TM^\perp$ along $\gamma$. 

Denote by $\tilde{B}_\varepsilon^{m-1}$ the centered ball of radius $\varepsilon>0$ in $\mathbb{R}^{m-1}$ and by $\tilde{B}_\sigma^k$ the centered ball of radius $\sigma > 0$ in $\mathbb{R}^k$.
Denote $\alpha=(\alpha_1,\ldots,\alpha_m), \beta=(\beta_1,\ldots,\beta_k)$ and define 
$$
\begin{aligned}
&\phi: (-\delta_0,\delta_0) \times \tilde{B}_{\varepsilon_0}^{m-1} \times \tilde{B}_{\sigma_0}^k \rightarrow \R^{m+k}, \\
\phi(\alpha,\beta) &:=\exp_{M, \gamma(\alpha_1)} \left(\sum_{i=2}^m \alpha_i e_i(\alpha_1)\right) + \sum_{j=1}^k \beta_j \mathbf{n}_j(\alpha),
\end{aligned}
$$
which is a diffeomorphism provided that $\delta_0,\varepsilon_0,\sigma_0 > 0$ are sufficiently small. This is referred to as the Fermi chart along $\gamma$ adapted to the submanifold $M$. See \cref{fermi-coordinates-surface-fig} for an illustration on a 2-surface in $\R^3$.

The Riemannian metric is expressed in the Fermi coordinates as
\begin{equation}
\label{metric-in-fermi}
g_{ij}(\alpha) = \langle \partial_{\alpha_i} \phi(\alpha,0), \partial_{\alpha_j} \phi(\alpha,0)
\rangle.
\end{equation}
\end{definition}

\begin{remark}
The advantage of $\phi$ over a generic $\psi$ as given in \cref{psi-notation} is that $\phi$ is adapted to the geodesic $\gamma$ in a way that simplifies computations of distances relevant to our optimal transport problem. The chart $\phi$  yields again a foliation $\{\phi(U, \beta) : \beta \in \tilde{B}_{\sigma_0}^k\}$ of $M_{\sigma_0}$.
\end{remark}

\begin{definition}[Test measures]
\label{def-test-measures}
Denote the cylinder-like segment in $\R^n$ of height $\sigma$ and radius $\varepsilon$ centered at $x \in M$ as
$$
B_{\sigma,\varepsilon}(x) := \{z \in M_\sigma : d_M(\pi(z), x) < \varepsilon\}
$$
and let $\mu$ be the Lebesgue measure on $\mathbb{R}^n$. For any $x \in M$ define the family of test probability measures
$$
\forall A \in \mathcal{B}(\mathbb{R}^n): \mu_x^{\sigma, \varepsilon}(A) = \frac{\mu(A \cap B_{\sigma,\varepsilon}(x))}{\mu(B_{\sigma,\varepsilon}(x))}
$$
indexed by $\varepsilon, \sigma >0$.
\end{definition}

Denote $\hat{\alpha} = (\alpha_2,\ldots,\alpha_m)$ so that $\alpha = (\alpha_1, \hat{\alpha})$. The main purpose of the expansion in the following lemma is twofold. First, we use it to design the third order corrections in the approximate transport map of \cref{proposed-transport-map} so that density matching occurs in \cref{proposed-approximate}. Second, the first order term of the expansion interacts with first order term of pointwise distance when integrating to get the Wasserstein upper bound in the proofs of Sections 3 and 4.

\begin{lemma}[Test measures in Fermi coordinates]
\label{test-measures-in-fermi-coordinates}
For any $y = \gamma(\delta)$, the expansion of the density of test measures in Fermi coordinates is
\begin{equation}
\label{test-measures-fermi}
\begin{aligned}
&(\phi_*^{-1} \mu_y^{\sigma, \varepsilon})(d\alpha, d\beta) \\
&= \f{1}{Z} \mathbbm{1}_{\tilde{B}_{\sigma,\varepsilon}}(\delta+\alpha_1, \hat{\alpha},\beta) \bigg( 1 - \sum_{i=1}^k \beta_i H^i(\phi(\mathbf{0})) - \sum_{i=1}^k \sum_{j=1}^m \alpha_j \beta_i \partial_{\alpha_j}(H^i \circ \phi)(\mathbf{0}) \\
& \hspace{1.5cm} - \sum_{i,j=1}^k \beta_i \beta_j \partial_{\beta_j} (H^i \circ \phi)(\mathbf{0}) + \f{1}{2} \sum_{i,j=1}^k \beta_i \beta_j H^i(\phi(\mathbf{0})) H^j(\phi(\mathbf{0})) \\
& \hspace{1.5cm} + \f{1}{4} \sum_{q,\ell=2}^m \sum_{i=1}^m\alpha_q \alpha_\ell \partial_{\alpha_q}\partial_{\alpha_\ell} g_{ii}(\mathbf{0}) + O(\delta^3) \bigg) d\alpha d\beta
\end{aligned}
\end{equation}
where $Z$ is the probability normalization constant and $g=(g_{ij})$ is the Riemannian metric of $M$ in Fermi coordinates given by \eqref{metric-in-fermi}.
\end{lemma}

\begin{proof}
First, note the pull-back of the test measure $\mu_y^{\sigma, \varepsilon}(d\alpha, d\beta)$ to Fermi coordinates is
$$
\begin{aligned}
&(\phi_*^{-1} \mu_y^{\sigma, \varepsilon})(d\alpha, d\beta) \\
&= \frac{1}{Z} \mathbbm{1}_{\tilde{B}_{\sigma,\varepsilon}}(\alpha_1+\delta,\hat{\alpha},\beta) \exp\left( -\int_0^1 \|\beta\| \< H(\phi(\alpha, s\beta)), \mathbf{n}(\phi(\alpha, s\beta)) \> ds\right)  (\phi_*^{-1} \vol_M)(d\alpha)  d\beta .
\end{aligned}
$$
The Riemannian metric in Fermi coordinates expands as
\begin{equation}
\label{factor-1}
\begin{aligned}
g_{ij}(\alpha) &= g_{ij}(\mathbf{0}) + \sum_{\ell=1}^m \alpha_\ell \partial_{\alpha_\ell}g_{ij}(\mathbf{0}) + \f{1}{2} \sum_{q,\ell=1}^m \alpha_q \alpha_\ell \partial_{\alpha_q}\partial_{\alpha_\ell} g_{ij}(\mathbf{0}) + O(\varepsilon^3) \\
&= \delta_{ij} + \f{1}{2} \sum_{q,\ell=2}^m \alpha_q \alpha_\ell \partial_{\alpha_q}\partial_{\alpha_\ell} g_{ij}(\mathbf{0}) + O(\varepsilon^3).
\end{aligned}
\end{equation}
Indeed, $\partial_{\alpha_i} g_{j\ell}(\alpha_1,\mathbf{0}) = 0$ for all $\alpha_1 \in (-\delta_0,\delta_0)$, and hence also $\partial_{\alpha_1} \partial_{\alpha_i} g_{j\ell}(\alpha_1,\mathbf{0}) = 0$. We show this by cases:
\begin{itemize}
    \item $\forall i,j,\ell = 2,\ldots, m: \partial_i g_{j\ell}(\alpha_1, \mathbf{0}) = 0$ because for every fixed $\alpha_1 \in (-\delta_0,\delta_0)$, $\phi(\alpha_1, \cdot)$ are normal coordinates within the injectivity radius of $\exp_{\gamma(\alpha_1)} (\{\dot{\gamma}(\alpha_1)\}^\perp) \subset M$ at $\gamma(\alpha_1)$ (see e.g. \cite[Sec. 1.4]{MR3726907} for a proof).
    \item $\forall i,j=1,\ldots,m: \partial_1 g_{ij}(\alpha_1, \mathbf{0}) = 0$ for any $\alpha_1 \in (-\delta_0,\delta_0)$ as the orthonormal frame along $\gamma$ used to define the Fermi chart is parallel translated along $\gamma$.
    
    \item $\forall i=2,\ldots,m$ and $j=1,\ldots,m$ and any $\alpha_1 \in (-\delta_0,\delta_0)$,
    $$
    \begin{aligned}
    \partial_i g_{1j}(\alpha_1, \mathbf{0}) &= \<\partial_{\alpha_1} \partial_{\alpha_i} \phi(\alpha_1, \mathbf{0}), \partial_{\alpha_j} \phi(\alpha_1, \mathbf{0})\> + \<\partial_{\alpha_1} \phi(\alpha_1,\mathbf{0}), \nabla^M_{\partial_{\alpha_i}\phi} \partial_{\alpha_j} \phi(\alpha_1,\mathbf{0})\> \\
    &= \<\nabla^{\R^n}_{\partial_{\alpha_1} \phi} \partial_{\alpha_i} \phi(\alpha_1,\mathbf{0}) - \nabla^M_{\partial_{\alpha_1} \phi} \partial_{\alpha_i} \phi(\alpha_1,\mathbf{0}), \partial_{\alpha_j} \phi(\alpha_1,\mathbf{0})\> = 0
    \end{aligned}
    $$
    using that $\nabla^{\R^n}_{\partial_{\alpha_1} \phi} \partial_{\alpha_i} \phi(\alpha_1,\mathbf{0}) - \nabla^M_{\partial_{\alpha_1} \phi} \partial_{\alpha_i} \phi(\alpha_1,\mathbf{0}) \perp M$ and $\nabla^M_{\partial_{\alpha_i}\phi} \partial_{\alpha_j} \phi(\alpha_1,\mathbf{0}) = 0$. The latter vanishes for $j\neq 1$ again by normality of the chart on $\exp_{\gamma(\alpha_1)} (\{\dot{\gamma}(\alpha_1)\}^\perp)$, and for $j = 1$ because $\partial_{\alpha_i} \phi$ is given by parallel translation along $\gamma$.
\end{itemize}
The Riemannian volume expanded in the Fermi coordinates then simplifies to
$$
\begin{aligned}
\phi_*^{-1} \vol_M(d\alpha) &= \sqrt{\det g_{ij}(\alpha)} d\alpha \\
&=\det \left( \delta_{ij} + \f{1}{2} \sum_{q,\ell=2}^m \alpha_q \alpha_\ell \partial_{\alpha_q}\partial_{\alpha_\ell} g_{ij}(\mathbf{0}) + O(\varepsilon^3)\right)^{\frac{1}{2}} d\alpha \\
&=  \left( 1 + \f{1}{4} \sum_{q,\ell=2}^m \sum_{i=1}^m \alpha_q \alpha_\ell \partial_{\alpha_q}\partial_{\alpha_\ell} g_{ii}(\mathbf{0})+ O(\varepsilon^3) \right) d\alpha.
\end{aligned}
$$
Moreover, expand the exponent in the normal part of the disintegration as
$$
\begin{aligned}
-\int_0^1 \|\beta\| H_{\mathbf{n}}(\phi(\alpha,s\beta))ds &= -\int_0^1 \sum_{i=1}^k \beta_i H^i(\phi(\alpha,s\beta)) ds\\
&= - \sum_{i=1}^k \beta_i H^i(\phi(\mathbf{0})) - \sum_{i,j} \alpha_j \beta_i \partial_{\alpha_j}(H^i \circ \phi)(\mathbf{0}) \\
&\qquad - \sum_{i,j} \beta_i \beta_j \partial_{\beta_j} (H^i \circ \phi)(\mathbf{0}) + O(\delta^3)
\end{aligned}
$$
and apply the approximation up to second order $e^x = 1+x+\f{x^2}{2} + O(x^3)$ to obtain
\begin{equation}
\label{factor-2}
\begin{aligned}
&\exp\left(-\int_0^1 \|\beta\| H_{\mathbf{n}}(\phi(\alpha,s\beta))ds \right) \\
&= 1 - \sum_{i=1}^k \beta_i H^i(\phi(\mathbf{0})) - \sum_{i,j} \alpha_j \beta_i \partial_{\alpha_j}(H^i\circ \phi)(\mathbf{0}) \\
&\qquad - \sum_{i,j} \beta_i \beta_j \partial_{\beta_j} (H^i \circ \phi)(\mathbf{0}) +  \f{1}{2} \sum_{i,j} \beta_i \beta_j H^i(\phi(\mathbf{0})) H^j(\phi(\mathbf{0})) + O(\delta^3).
\end{aligned}
\end{equation}
We conclude the result by taking the product of the two factors \eqref{factor-1} and \eqref{factor-2}, merging third order terms in $\alpha,\beta$ into $O(\delta^3)$ by the assumption $\varepsilon \vee \sigma \leq \f{\delta}{4}$.

The probability normalization constant can be deduced by integration of \eqref{factor-2} with respect to $\mu_y^{\sigma,\varepsilon}$ as
$$
Z = 1 + \f{1}{2(k+2)} \sigma^2 \sum_{i=1}^k H^i(\phi(\mathbf{0}))^2 + O(\delta^3),
$$
and so
$$
\f{1}{Z} = 1 - \f{1}{2(k+2)}\sigma^2 \sum_{i=1}^k H^i(\phi(\mathbf{0}))^2 + O(\delta^3).
$$
\end{proof}

\subsection{Proposed transport map}
\label{section-proposed-transport-map}
As mentioned in the introduction, when considering an embedded manifold, it is crucial to include the mean curvature in the transport map. We will show that the transport map proposed below is an approximate transport map with degree 3.
We then present a criterion for optimality of the proposed map in \cref{test-function-gradient}.

\begin{definition}
\label{proposed-transport-map}
Define $T : B_{\sigma,\varepsilon}(x_0) \rightarrow B_{\sigma, \varepsilon}(y)$ in Fermi coordinates as
$$
\begin{aligned}
T(\phi(\alpha,\beta)) := \phi\bigg(&\delta-\alpha_1, \alpha_2, \ldots, \alpha_m, \\
& \beta_1-\f{1}{2}(\sigma^2 -\beta_1^2)(\delta-2\alpha_1) \partial_{\alpha_1} (H^1 \circ \phi)(\mathbf{0}),\\
&\ldots,\\
&\beta_k-\f{1}{2}(\sigma^2 -\beta_k^2)(\delta-2\alpha_1) \partial_{\alpha_1} (H^k \circ \phi)(\mathbf{0})\bigg).
\end{aligned}
$$
Denote by $\hat{\alpha} = (\alpha_2,\ldots,\alpha_m)$ and similarly for $\hat{\alpha}'$, and denote the input vector on the right in the above definition as $(\alpha',\beta')$. Note that 
$$
\alpha'_1 = \delta-\alpha_1, \quad \hat{\alpha}'=\hat{\alpha}, \quad \beta' = \beta + O(\delta^3).
$$
\end{definition}

\begin{remark}
Observe that $T$ is a local diffeomorphism and 
\begin{equation}
|\det D(\phi^{-1} \circ T \circ \phi)(\alpha,\beta))| = 1 - \sum_{i=1}^k \beta_i(\delta-2\alpha_1) \partial_{\alpha_1} (H^i\circ\phi)(\mathbf{0}) + O(\delta^3)
\end{equation}
and deduce
\begin{equation}
\label{determinant-coordinates}
\begin{aligned}
|\det D(\phi^{-1} \circ T \circ \phi)^{-1}(\alpha',\beta'))| &= |\det D(\phi \circ T \circ \phi)(\alpha,\beta))|^{-1}\\
&= 1 + \sum_{i=1}^k \beta_i(\delta-2\alpha_1) \partial_{\alpha_1} (H^i\circ\phi)(\mathbf{0}) + O(\delta^3).
\end{aligned}
\end{equation}

\end{remark}

\begin{remark}
    The third order terms in the definition of $T$ are adjustments to cancel out second order terms in the proof of \cref{proposed-approximate} below,  obtaining an approximate transport of degree 3 as a result. In fact, the form of $T$ is tailored precisely for this to occur. It turns out these third order adjustment terms do not influence the Wasserstein distance computation up to order 4.
\end{remark}

We need two general lemmas to show that $T$ is an approximate transport of degree 3.

\begin{lemma}[Density under pushforward]
\label{density-change-of-variable}
    Let $\mathcal{X},\mathcal{Y}$ be measurable spaces, $\phi: \mathcal{X} \rightarrow \mathcal{Y}$ a measurable bijection with measurable inverse, and $\mu, \nu$ two measures on $\mathcal{X}$ with $\mu \ll \nu$. Then the push-forward measures are also absolutely continuous with density
    \begin{equation}
    \label{density-change-of-variable-formula}
    \frac{d(\psi_*\mu)}{d(\psi_*\nu)}(x) = \frac{d\mu}{d\nu}(\psi^{-1}(x)).
    \end{equation}
\end{lemma}
\begin{proof}
For all measurable sets $A \subset \mathcal{X}$,
$$
\begin{aligned}
\psi_*\mu(A)&=\mu(\psi^{-1}(A)) \\
&=\int_{\psi^{-1}(A)} \frac{d\mu}{d\nu}(x)d\nu(x)\\
&= \int_{\psi^{-1}(A)} \frac{d\mu}{d\nu}(\psi^{-1}\circ \psi(x))d\nu(x)\\
&= \int_A \frac{d\mu}{d\nu}(\psi^{-1}(x))d(\psi_*\nu)(x).
\end{aligned}
$$
\end{proof}

Noting that the representations \eqref{skew-decomposition-general} and \eqref{skew-decomposition-hypersurface} are decompositions into skew-products of two measures, the following will be used for density comparisons.

Let $\mathcal{X},\mathcal{Y}$ be measurable spaces. Given measures 
$\{\mu_1^y: y\in \mathcal{Y}\}$ on $\mathcal{X}$ and a measure $\mu_2$ on $\mathcal{Y}$, the skew-product is defined as follows:
For all bounded measurable real valued functions $f$ on $\mathcal{X}\otimes \mathcal{Y}$,
    $$
    \mu_1 \otimes \mu_2(f):= \int_\mathcal{Y} \int_\mathcal{X} f(x,y) d\mu_1^y(x) d\mu_2(y).
    $$

\begin{lemma}[Skew-product density factorization]
\label{skew-product-density}
    Consider two families of measures $(\mu_1^y)_{y \in \mathcal{Y}}$ and $(\nu_1^y)_{y \in \mathcal{Y}}$ on $\mathcal{X}$ such that $\mu_1^y \ll \nu_1^y$ for every $y \in \mathcal{Y}$ and the map $(x,y) \mapsto \frac{d\mu_1^y}{d\nu_1^y}(x)$ is measurable. Furthermore, let $\mu_2$ and $\nu_2$ be measures on $\mathcal{Y}$ with $\mu_2 \ll \nu_2$.
Consider the skew products of $(\mu_1^y)_{y\in \mathcal{Y}}$ with $\mu_2$ and that of $(\nu_1^y)_{y \in \mathcal{Y}}$ with $\nu_2$.
Then $
\mu_1 \otimes \mu_2 \ll \nu_1 \otimes \nu_2 $ and
$$
\frac{d(\mu_1\otimes \mu_2)}{d(\nu_1 \otimes \nu_2)}(x,y) = \frac{d\mu_1^y}{d\nu_1^y}(x) \frac{d\mu_1}{d\nu_2}(y).
$$
\end{lemma}
\begin{proof}
    Plugging in the densities, $\forall f \in \mathcal{M}_\mathcal{X} \otimes \mathcal{M}_\mathcal{Y}$ bounded:
    $$
    \begin{aligned}
    \mu_1 \otimes \mu_2(f) &= \int_\mathcal{Y} \int_\mathcal{X} f(x,y) d\mu_1^y(x) d\mu_2(y) \\
    &= \int_\mathcal{Y} \int_\mathcal{X} f(x,y) \frac{d\mu_1^y}{d\nu_1^y}(x)\frac{d\mu_2}{d\nu_2}(y) d\nu_1^y(x) d\nu_2(y) \\
    &= \int_\mathcal{Y} \int_\mathcal{X} f(x,y) \frac{d\mu_1^y}{d\nu_1^y}(x)\frac{d\mu_2}{d\nu_2}(y) d(\nu_1 \otimes \nu_2)(x,y) \\
    &= \int_\mathcal{Y} \int_\mathcal{X} f(x,y)\frac{d(\mu_1\otimes \mu_2)}{d(\nu_1 \otimes \nu_2)}(x,y)  d(\nu_1 \otimes \nu_2)(x,y).
\end{aligned}
$$
\end{proof}

We verify that the density of $T_*\mu_{x_0}^{\sigma,\varepsilon}$ via $T$ matches that of $\mu_y^{\sigma,\varepsilon}$ up to $O(\delta^3)$.
\begin{proposition}
\label{proposed-approximate}
    The proposed map is an approximate transport map of degree 3, i.e.
    $$
    \f{d(T_* \mu_{x_0}^{\sigma,\varepsilon})}{d\mu_y^{\sigma,\varepsilon}}(\phi(\alpha,\beta)) = 1 + O(\delta^3).
    $$
\end{proposition}

\begin{proof}
    First, combining the elementary change of variable formula with the Fermi coordinate representation of $\mu_{x_0}^{\sigma,\varepsilon}$, with notation of \cref{proposed-transport-map} we have
    \begin{equation}
    \label{measure-1}
    \begin{aligned}
    &(\phi_*^{-1} T_* \mu_{x_0}^{\sigma,\varepsilon})(d\alpha,d\beta) \\
    &= (\phi^{-1} \circ T  \circ \phi)_* (\phi_*^{-1}\mu_{x_0}^{\sigma,\varepsilon})(d\alpha,d\beta) \\
    &= \f{1}{Z} \mathbbm{1}_{\tilde{B}_{\sigma,\varepsilon}}(\alpha'_1-\delta, \hat{\alpha}', \beta')
    \bigg( 1 - \sum_{i=1}^k \beta'_i H^i(\phi(\mathbf{0})) - \sum_{i=1}^k \sum_{j=1}^m \alpha'_j \beta'_i \partial_{\alpha_j}(H_{\mathbf{n}_i} \circ \phi)(\mathbf{0}) \\
    & \hspace{0.5cm} - \sum_{i,j=1}^k \beta'_i \beta'_j \partial_{\beta_j} (H^i \circ \phi)(\mathbf{0}) + \f{1}{2} \sum_{i,j=1}^k \beta'_i \beta'_j H^i(\phi(\mathbf{0})) H^j(\phi(\mathbf{0})) \\
    &\hspace{0.5cm} + \f{1}{4} \sum_{q,\ell=2}^m \sum_{i=1}^m \alpha_q \alpha_\ell \partial_{\alpha_q}\partial_{\alpha_\ell} g_{ii}(\mathbf{0}) + \sum_{i=1}^k \beta_i(\delta-2\alpha_1) \partial_{\alpha_1} (H^i\circ\phi)(\mathbf{0}) + O(\delta^3) \bigg) d\alpha d\beta
    \end{aligned}
    \end{equation}
    using the expansions \eqref{determinant-coordinates} for the determinant of $\phi^{-1} \circ T \circ \phi$ and \eqref{test-measures-fermi} for the coordinate representation of $\phi^{-1}_* \mu_{x_0}^{\sigma,\varepsilon}$.
  
    We use \cref{density-change-of-variable} to push the density into Fermi coordinates, and then \cref{skew-product-density} allows us to take the ratio of the densities of \eqref{measure-1} and \eqref{test-measures-fermi}, obtaining
    $$
    \begin{aligned}
    \f{d(T_* \mu_{x_0}^{\sigma,\varepsilon})}{d\mu_y^{\sigma,\varepsilon}} (\phi(\alpha,\beta)) &= \f{d(\phi_*^{-1} T_* \mu_{x_0}^{\sigma,\varepsilon})}{d(\phi_*^{-1} \mu_y^{\sigma,\varepsilon})} (\alpha,\beta) \\
    &= \mathbbm{1}_{\tilde{B}_{\sigma,\varepsilon}(\mathbf{0})}(\alpha_1-\delta, \hat{\alpha},\beta) \left( 1 + O(\delta^3)\right)
    \end{aligned}
    $$
    because the second order terms cancel out. 
    Here we also used that $T$ is a diffeomorphism from $\tilde{B}_{\sigma, \varepsilon}(\mathbf{0}_{m+k})$ to $\tilde{B}_{\sigma, \varepsilon}(\delta, \mathbf{0}_{m+k-1})$, hence
    $$
    \mathbbm{1}_{\tilde{B}_{\sigma,\varepsilon}(\delta,\mathbf{0})}(\alpha',\beta') = \mathbbm{1}_{\tilde{B}_{\sigma,\varepsilon}(\mathbf{0})}(\alpha_1-\delta, \hat{\alpha},\beta).
    $$
\end{proof}

\begin{remark}
Building upon the preceding proposition and leveraging \cref{asymptotically-optimal-transport}, we readily deduce that the proposed transport map satisfies:
$$
    W_1(\mu_{x_0}^{\sigma,\varepsilon}, \mu_y^{\sigma,\varepsilon}) 
    = W_1(\mu_{x_0}^{\sigma,\varepsilon}, T_* \mu_{x_0}^{\sigma,\varepsilon}) + O(\delta^4),
$$  
taking also into account that $\textrm{supp } T_*\mu_{x_0}^{\sigma,\varepsilon} = \textrm{supp } \mu_y^{\sigma,\varepsilon}$ leading to $\textrm{diam } \textrm{supp } T_*\mu_{x_0}^{\sigma,\varepsilon} = O(\delta)$ when $\sigma \vee \varepsilon \leq \f{\delta}{4}$.
Thus, when computing the coarse curvature, we may use $W_1(\mu_{x_0}^{\sigma,\varepsilon}, T_* \mu_{x_0}^{\sigma,\varepsilon})$. This is justified as terms involving the second fundamental form at the point $x_0$ emerge only at the third order in the expansion of $W_1(\mu_{x_0}^{\sigma,\varepsilon}, \mu_y^{\sigma,\varepsilon})$, making precision up to $O(\delta^4)$ sufficient.

\end{remark}

The following will allow us to deduce a Wasserstein lower bound from an upper bound provided by an approximate transport map of degree $3$, and merging these into a both-sided estimate up to $O(\delta^4)$. 

\begin{lemma}
    \label{test-function-gradient}
    If $f: B_{2\delta}(x_0) \subset \R^{m+k} \rightarrow \R$ is smooth and takes the form
    \begin{equation}
    \label{test-function-assumption}
    f(Tz) - f(z) = \|Tz - z\| + O(\delta^4) = O(\delta)
    \end{equation}
    and the magnitude of its gradient satisfies
    $$
    \sup_{z \in B_{2\delta}(x_0)} \|\nabla f(z)\| = 1 + O(\delta^3)
    $$
    then for all $\sigma,\varepsilon,\delta$ sufficiently small with $\sigma \vee \varepsilon \leq \f{\delta}{4}$,
    $$
    W_1(\mu_{x_0}^{\sigma,\varepsilon}, T_*\mu_{x_0}^{\sigma,\varepsilon}) = \int \|Tz-z\| d\mu_{x_0}^{\sigma,\varepsilon}(z) + O(\delta^4) = \int (f(Tz) - f(z)) d\mu_{x_0}^{\sigma,\varepsilon}(z) + O(\delta^4).
    $$
\end{lemma}

\begin{proof}
    Using the expansion $\f{1}{1+a} = 1 - a + O(a^2)$, we deduce that
    $$
    \left(\sup_{z' \in B_{2\delta}(x_0)} \|\nabla f(z')\|\right)^{-1} f(z) = (1 + O(\delta^3))f(z).
    $$
    By the mean value theorem, a differentiable function divided by the supremum of its gradient is 1-Lipschitz. Then by Kantorovich-Rubinstein duality
$$
\begin{aligned}
W_1(\mu_{x_0}^{\sigma,\varepsilon}, T_*\mu_{x_0}^{\sigma,\varepsilon}) &\geq \int_{B_{\sigma,\varepsilon}(x_0)} \left(\sup_{z' \in B_{2\delta}(x_0)} \|\nabla f(z')\|\right)^{-1} (f(Tz) -f(z)) \mu_{x_0}^{\sigma,\varepsilon}(dz) \\
&= \int_{B_{\sigma,\varepsilon}(x_0)} (1 + O(\delta^3))(f(Tz) -f(z)) d\mu_{x_0}^{\sigma,\varepsilon}(\alpha,\beta) \\
&= \int_{B_{\sigma,\varepsilon}(x_0)} (f(Tz) -f(z))\mu_{x_0}^{\sigma,\varepsilon}(dz) + O(\delta^4)\\
&= \int_{B_{\sigma,\varepsilon}(x_0)} \|Tz-z\| \mu_{x_0}^{\sigma,\varepsilon}(dz) + O(\delta^4) \\
&\geq W_1(\mu_{x_0}^{\sigma,\varepsilon}, T_*\mu_{x_0}^{\sigma,\varepsilon}) + O(\delta^4),
\end{aligned}
$$
using the assumption \eqref{test-function-assumption} on the third and fourth line.
\end{proof}

Finally, in order to integrate over the correct range of Fermi coordinates to cover precisely $B_{\sigma,\varepsilon}(x_0)$ as the support of $\mu_{x_0}^{\sigma,\varepsilon}$, we need to find the range parameter $\varepsilon(\alpha)$ such that if
$$
\mathbf{B}_{\sigma,\varepsilon}(\mathbf{0}) := \{(\alpha, \beta): |\alpha_1| \leq \varepsilon, \sum_{j=2}^m \alpha_j^2 \leq \varepsilon(\alpha)^2, \sum_{i=1}^k \beta_i^2 \leq \sigma^2 \} \subset \R^{m+k}
$$
then
$$
\phi(\mathbf{B}_{\sigma, \varepsilon}(\mathbf{0})) = B_\varepsilon(x_0).
$$
This is a necessary consideration, because in general non-flat spaces
$$
\phi(\tilde{B}_{\sigma,\varepsilon}(\mathbf{0})) \neq B_{\sigma,\varepsilon}(x_0).
$$

The following is a classical result of Toponogov, which is a generalization of Pythagoras theorem for Riemannian manifolds and gives a characterisation of sectional curvature. See e.g. \cite{Meyer2004ToponogovsTA} for a proof.
\begin{lemma}
For any point $x_0\in M$ and any $w_1,w_2 \in T_{x_0}M$ sufficiently small, the Riemannian distance between $\exp_{x_0} (w_1)$ and $\exp_{x_0} (w_2)$ has the expansion
$$
d(\exp_{x_0}(w_1), \exp_{x_0}(w_2)) = \|w_1-w_2\| - \f{1}{3} \< R(w_1,w_2)w_2,w_1 \> + O(\max(\|w_1\|, \|w_2\|)^5).
$$
\end{lemma}

As a consequence, we deduce that given a coordinate $\alpha_1 \in (-\varepsilon, \varepsilon)$, the range parameter $\varepsilon(\alpha)$ is characterized by the relation
$$
\varepsilon^2 = \alpha_1^2 + \varepsilon(\alpha)^2 + O(\max(\alpha_1^2, \varepsilon(\alpha)^2))
$$
where the coefficient in the remainder term only depends on a fixed neighbourhood of $x_0$. This implies $\varepsilon(\alpha) = O(\varepsilon)$ and 
$$
\varepsilon(\alpha) = (1+ O(\varepsilon^2)) \sqrt{\varepsilon^2 -\alpha_1^2} = \sqrt{\varepsilon^2 -\alpha_1^2} + O(\varepsilon^3).
$$
We shall label the remainder term $r(\alpha) = O(\varepsilon^3)$ for the purpose of the following proof. The next corollary will allow us to ignore the distinction between $\tilde{B}_{\sigma,\varepsilon}(\mathbf{0})$ and $\mathbf{B}_{\sigma,\varepsilon}(\mathbf{0})$ up to $O(\delta^4)$ whenever we integrate with respect to the test measure $\mu_{x_0}^{\sigma,\varepsilon}$ in Fermi coordinates.
\begin{corollary}
If $P: \R^{m+k} \rightarrow \R$ is a polynomial with no constant term and $\max(\sigma,\varepsilon) \leq \delta$ then
$$
\int_{\tilde{B}_{\sigma,\varepsilon}(\mathbf{0})} P(\alpha,\beta) d (\phi^{-1}_* \mu_{x_0}^{\sigma,\varepsilon})(\alpha,\beta)= \int_{\mathbf{B}_{\sigma,\varepsilon}(\mathbf{0})} P(\alpha,\beta) d (\phi^{-1}_* \mu_{x_0}^{\sigma,\varepsilon})(\alpha,\beta) + O(\delta^4).
$$
\end{corollary}
\begin{proof}
We split the domain of integral on the right so that one part matches the domain on the left and the integral of the other part is $O(\delta^4)$:
$$
\begin{aligned}
&\int_{\mathbf{B}_{\sigma,\varepsilon}(\mathbf{0})} P(\alpha,\beta) d (\phi^{-1}_* \mu_{x_0}^{\sigma,\varepsilon})(\alpha,\beta) \\ 
&= \int_{\substack{|\alpha_1|\leq \varepsilon, \|\hat{\alpha}\| \leq \varepsilon(\alpha),\\ \|\beta\| \leq \sigma}} P(\alpha,\beta) d (\phi^{-1}_* \mu_{x_0}^{\sigma,\varepsilon})(\alpha,\beta) \\
&= \int_{\substack{|\alpha_1|\leq \varepsilon, \|\hat{\alpha}\|\leq \sqrt{\varepsilon^2 - \alpha_1^2},\\ \|\beta\|\leq \sigma}} P(\alpha,\beta) d (\phi^{-1}_* \mu_{x_0}^{\sigma,\varepsilon})(\alpha,\beta) \\
& \qquad + O(\int_{|\alpha_1|\leq \varepsilon, -|r(\alpha)| \leq \|\hat{\alpha}\| - \sqrt{\varepsilon^2-\alpha_1^2} \leq |r(\alpha)|} P(\alpha,\beta)) d (\phi^{-1}_* \mu_{x_0}^{\sigma,\varepsilon})(\alpha,\beta) \\
&= \int_{\tilde{B}_{\sigma,\varepsilon}(\mathbf{0})} P(\alpha,\beta) d (\phi^{-1}_* \mu_{x_0}^{\sigma,\varepsilon})(\alpha,\beta) + O(\delta^4),
\end{aligned}
$$
on the last line we using that $P(\alpha,\beta)$ has no constant term and $r(\alpha) = O(\varepsilon^3) = O(\delta^3)$.
\end{proof}

\section{Curves and surfaces}
\label{section:computed-examples}
We establish explicit formulas for the coarse extrinsic curvature defined by \eqref{coarse-extrinsic-curvature-def} in four practically relevant cases: a circle, a planar curve, a space curve, and a surface.
We begin by presenting the common setup shared among all these cases.

\subsection{The circle example}
\label{section-circle-example}
Our motivating example is the circle $S^1_R$ with a fixed radius $R>0$, which avoids technicalities arising from varying radius in the osculating circle, an issue that will be addressed in \Cref{planar-curve} in the case of planar curves.

\begin{notation}
Denote the polar coordinates
\begin{equation}
\label{circle-phi}
\phi(\alpha,\beta) := \begin{pmatrix} (R-\beta) \cos (\alpha/R)\\  (R-\beta) \sin (\alpha/R)) \end{pmatrix},
\end{equation}
where $\alpha \in (-\pi R,\pi R)$ parametrizes arc-length distance from the point $(R,0)$ along the circle and $\beta \in (-\sigma, \sigma)$ parametrizes the direction normal to the circle.

Denote $x_0 := \phi(0,0) = (R,0)$ and for every $\delta >0$ denote $y := \begin{pmatrix}
R\cos(\delta/R)\\  R\sin(\delta/R))\end{pmatrix}$.
\end{notation}

\begin{lemma}
The test measures in polar coordinates take the form
$$
\begin{aligned}
(\phi_*^{-1} \mu_{y}^{\sigma, \varepsilon})(d\alpha,d\beta) &= \f{1}{4\sigma \varepsilon}\mathbbm{1}_{(\delta-\varepsilon, \delta+\varepsilon)\times (-\sigma,\sigma)}(\alpha, \beta)\left(1-\f{\beta}{R}\right) d\alpha d\beta.
\end{aligned}
$$
\end{lemma}
\begin{proof}
    At any $(\alpha, \beta)$, the radial coordinate is $R-\beta$, the radial length element is $d\beta$ and the angular element is $\f{d\alpha}{R}$, giving the volume element $(R-\beta) d\beta \frac{d\alpha}{R} = (1-\frac{\beta}{R})d\alpha d\beta$, with $\frac{1}{4\sigma \varepsilon}$ as the probability normalization factor for the support $(\delta-\varepsilon, \delta+\varepsilon)\times (-\sigma,\sigma)$.

    This is consistent with the formula of \cref{ambient-volume-disintegration}, as the mean curvature at $(\alpha,\beta)$ is $\f{1}{R-\beta}$, which gives the density
    $$
    e^{-\int_0^1 \f{\beta}{R-s\beta} ds} = e^{\log (R-\beta) - \log R} = 1-\f{\beta}{R}
    $$
    on $(\delta-\varepsilon, \delta+\varepsilon)\times (-\sigma,\sigma)$.
\end{proof}

The transport map of \cref{proposed-transport-map} boils down to
\begin{equation}
\label{transport-map-circle}
T(\phi(\alpha,\beta)) = \phi(\delta-\alpha, \beta) = \begin{pmatrix}
(R-\beta) \cos((\delta-\alpha)/R) \\ (R-\beta) \sin((\delta-\alpha)/R)\end{pmatrix},
\end{equation}
and note that $y = Tx_0 = T(\phi(0,0))$. See also \cref{planar-curve-fig} below.

\begin{remark}
In this case the transport map $T$ is precise in the sense that $T_* \mu_{x_0}^{\sigma,\varepsilon} = \mu_y^{\sigma, \varepsilon}$. Indeed, for any $f : \R^2 \rightarrow \R$ Borel measurable,
$$
\begin{aligned}
    \int f(z) d(T_* \mu_{x_0}^{\sigma,\varepsilon})(z) &= \int f(Tz) d\mu_{x_0}^{\sigma,\varepsilon}(z)\\
    &= \int f(T(\phi(\alpha,\beta))) d(\phi_*^{-1} \mu_{x_0}^{\sigma,\varepsilon})(d\alpha,d\beta) \\
    &= \f{1}{4\sigma \varepsilon} \int f(T(\phi(\alpha,\beta))) \mathbbm{1}_{(-\varepsilon, \varepsilon)\times (-\sigma,\sigma)}(\alpha, \beta) \left(1-\f{\beta}{R}\right)d\alpha d\beta \\
    &= \f{1}{4\sigma \varepsilon} \int f(\phi(\delta-\alpha,\beta)) \mathbbm{1}_{(-\varepsilon, \varepsilon)\times (-\sigma,\sigma)}(\alpha, \beta) \left(1-\f{\beta}{R}\right)d\alpha d\beta \\
    &= \f{1}{4\sigma \varepsilon} \int f(\phi(\alpha,\beta)) \mathbbm{1}_{(\delta-\varepsilon, \delta+\varepsilon)\times (-\sigma,\sigma)}(\alpha, \beta)\left(1-\f{\beta}{R}\right) d\alpha d\beta\\
    &= \int f(z) d\mu_y^{\sigma,\varepsilon}(z).
\end{aligned}
$$
\end{remark}

\begin{proposition}
For all $\delta, \varepsilon, \sigma >0$ sufficiently small with $\sigma \vee \varepsilon \leq \f{\delta}{2}$, it holds that
$$
\begin{aligned}
W_1(\mu_{x_0}^{\sigma,\varepsilon}, \mu_{y}^{\sigma,\varepsilon})&=2R^2 \sin\left(\f{\delta}{2R}\right)\f{1}{\varepsilon}\sin\left(\f{\varepsilon}{R}\right)\left(1+\f{\sigma^2}{3R^2}\right) \\
&= \|x_0-y\| \f{R}{\varepsilon}\sin\left(\f{\varepsilon}{R}\right) \left(1+\f{\sigma^2}{3R^2}\right).
\end{aligned}
$$
\end{proposition}
\begin{proof} 
For every point $z=\phi(\alpha,\beta)$, 
$$
\|Tz-z\| = 2 (R-\beta) \sin \left(\frac{\delta -2\alpha}{2R}\right),
$$
which is the Euclidean distance of two points on the circle at angle $\frac{\delta-2\alpha}{R}$ apart.
Integrating with respect to the test measure yields
$$
\begin{aligned}
W_1(\mu_{x_0}^{\sigma,\varepsilon}, \mu_y^{\sigma,\varepsilon}) &\leq \int \|Tz-z\| d\mu_{x_0}^{\sigma,\varepsilon}(z)\\
&= \frac{1}{4\sigma\varepsilon} \int_{-\sigma}^\sigma d\beta \int_{-\varepsilon}^\varepsilon d\alpha \left(1-\f{\beta}{R}\right) 2 (R-\beta) \sin \left(\frac{\delta -2\alpha}{2R}\right) \\
&= \frac{1}{4\sigma\varepsilon} \int_{-\sigma}^\sigma d\beta \left(1-\f{\beta}{R}\right) 2 (R-\beta) \\
& \qquad \times \int_{-\varepsilon}^\varepsilon d\alpha \left(\sin\left(\f{\delta}{2R}\right) \cos\left(\f{\alpha}{R}\right) - \sin\left(\f{\alpha}{R}\right)\cos\left(\f{\delta}{2R}\right)\right) \\
&= 2R^2 \sin\left(\f{\delta}{2R}\right)\f{1}{\varepsilon}\sin\left(\frac{\varepsilon}{R}\right)\left(1+\f{\sigma^2}{3R^2}\right).
\end{aligned}
$$
For the lower bound, we test against the 1-Lipschitz function
\begin{equation}
\label{lower-bound-test}
f(z) := \< z-x_0, \f{y-x_0}{\|y-x_0\|} \>.
\end{equation}
We have
$$
\begin{aligned}
y-x_0 = \phi(\delta,0)-\phi(0,0) &= \begin{pmatrix} R\cos(\delta/R) \\ R\sin(\delta/R) \end{pmatrix} -\begin{pmatrix} R \\ 0 \end{pmatrix} = \begin{pmatrix} R(\cos(\delta/R)-1)\\ R\sin (\delta/R)\end{pmatrix},
\end{aligned}
$$
and so $\|y-x_0\| = R\sqrt{2(1-\cos(\delta/R)} = 2R \sin(\delta/(2R))$, giving
\begin{equation}
\label{circle-projection-vector}
\f{y-x_0}{\|y-x_0\|} = \f{1}{2 \sin (\delta/(2R))} \begin{pmatrix} \cos(\delta/ R)-1 \\ \sin(\delta/R)\end{pmatrix}.
\end{equation}
Then we compute using \eqref{circle-phi}, \eqref{transport-map-circle} and \eqref{circle-projection-vector}:
$$
\begin{aligned}
    f(Tz)-f(z) &=  \< \phi(\delta-\alpha,\beta) - \phi(\alpha,\beta), \f{y-x_0}{\|y-x_0\|} \> \\
    &= \f{(R-\beta)}{2 \sin(\delta/(2R))} \begin{pmatrix}\cos((\delta-\alpha)/R) - \cos(\alpha/R) \\ \sin((\delta-\alpha)/R) - \sin(\alpha/R) \end{pmatrix} \cdot \begin{pmatrix} \cos(\delta/R)-1\\ \sin (\delta/R)\end{pmatrix} \\
    &= \f{R-\beta}{\sin(\delta/(2R))} (\cos(\alpha/R) - \cos((\delta-\alpha)/R) )\\
    &= 2(R-\beta) \sin((\delta-2\alpha)/(2R)) = \|Tz-z\|
\end{aligned}
$$
by trigonometric identities.
Therefore 
$$
\begin{aligned}
W_1(\mu_{x_0}^{\sigma,\varepsilon}, \mu_y^{\sigma,\varepsilon}) &\geq \int f(z)(d\mu^{\sigma,\varepsilon}_{y}(z)-d\mu^{\sigma,\varepsilon}_{x_0}(z)) \\
&= \int (f(Tz) -f(z)) d\mu^{\sigma,\varepsilon}_{x_0}(z) \\
&= \int \|Tz-z\| d\mu^{\sigma,\varepsilon}_{x_0}(z),
\end{aligned}
$$
which shows the lower bound agrees exactly with the upper bound. 
\end{proof}

\subsection{Planar curve}
\label{planar-curve}
Let $\gamma: (-\delta_0,\delta_0) \rightarrow \R^2$ be a smooth unit speed curve. As before, let $x_0:=\gamma(0)$, $y := \gamma(\delta)$ where $\delta\in (-\delta_0, \delta_0)$.

The normal vector field along $\gamma$ is given by $\mathbf{n}(\alpha) := \f{\ddot{\gamma}(\alpha)}{\|\ddot{\gamma}(\alpha) \|}$, the radius of the osculating circle is $R(\alpha) := \f{1}{\|\ddot{\gamma}(\alpha)\|}$ and we have the relationships
\begin{equation}
\label{derivatives-list}
\begin{aligned}
\ddot{\gamma}(\alpha) &= \frac{\mathbf{n}(\alpha)}{R(\alpha)}, \quad \dddot{\gamma}(\alpha)=-\f{1}{R(\alpha)^2} \dot{\gamma}(\alpha)- \f{\dot{R}(\alpha)}{R(\alpha)^2}\mathbf{n}(\alpha),\\
\dot{\mathbf{n}}(\alpha) &= -\f{\dot{\gamma}(\alpha)}{R(\alpha)}, \quad \ddot{\mathbf{n}}(\alpha) = \f{\dot{R}(\alpha)}{R(\alpha)^2}\dot{\gamma}(\alpha) - \f{1}{R(\alpha)^2} \mathbf{n}(\alpha).
\end{aligned}
\end{equation}
Let $\phi: (-\delta_0,\delta_0) \times (-\sigma_0,\sigma_0) \rightarrow \R^2$ be given as follows:
$$
\phi(\alpha,\beta) := \gamma(\alpha)+\beta \mathbf{n}(\alpha).
$$
This is the Fermi chart along $\gamma$. While we have the general Fermi coordinate representation in terms of the expansion in \cref{test-measures-in-fermi-coordinates}, in this case we arrive at a precise form:
\begin{lemma}
\label{test-measures-planar}
The test measures at $y = \gamma(\delta)$ are
\begin{equation}
\label{target-measure-in-coordinates-planar}
\begin{aligned}
(\phi_*^{-1} \mu_{y}^{\sigma, \varepsilon})(d\alpha,d\beta) &= \frac{1}{4\sigma\varepsilon}\mathbbm{1}_{(\delta-\varepsilon, \delta+\varepsilon)\times (-\sigma,\sigma)}(\alpha, \beta)\left(1-\f{\beta}{R(\alpha)} \right) d\alpha d\beta.
\end{aligned}
\end{equation}
\end{lemma}

\begin{proof}
To evaluate $H(\phi(\alpha,\beta))$ in applying \cref{ambient-volume-disintegration}, normalize the vector field tangent to the curve $\alpha \mapsto \phi(\alpha,\beta)$ and compute the second derivative in $\R^2$ as
$$
\begin{aligned}
\f{\partial_\alpha}{\|\partial_\alpha \phi(\alpha,\beta) \|}\left( \f{\partial_\alpha \phi(\alpha,\beta)}{\|\partial_\alpha \phi(\alpha,\beta)\|} \right) &= \f{\partial_\alpha^2 \phi(\alpha,\beta)}{\|\partial_\alpha \phi(\alpha,\beta)\|^2} - \f{\< \partial_\alpha^2 \phi(\alpha,\beta), \partial_\alpha \phi(\alpha,\beta)\>}{ \|\partial_\alpha \phi(\alpha,\beta)\|^3 } \partial_\alpha \phi(\alpha,\beta).
\end{aligned}
$$
The second term is tangential to the curve, so may be ignored for the computation of $H$. Moreover, 
$$
\begin{aligned}
\partial_\alpha^2 \phi(\alpha,\beta) &= \ddot{\gamma}(\alpha) + \beta \ddot{\mathbf{n}}(\alpha) = \beta \f{\dot{R}(\alpha)}{R(\alpha)^2} \dot{\gamma}(\alpha) + \f{1}{R(\alpha)} \left(1 -\f{\beta}{R(\alpha)} \right)\mathbf{n}(\alpha), \\
\|\partial_\alpha \phi(\alpha,\beta)\|^{-2} &= \|\dot{\gamma} + \beta \dot{\mathbf{n}}(\alpha)\|^{-2} = \left(1-\f{\beta}{R(\alpha)}\right)^{-2}.
\end{aligned}
$$
Note that $\mathbf{n}(\alpha)$ is normal to $\alpha \mapsto \phi(\alpha,\beta)$ for every $\beta$ since
$$
\<\mathbf{n}(\alpha), \partial_\alpha \phi(\alpha,\beta)\> =\<\mathbf{n}(\alpha), \dot{\gamma}(\alpha) + \beta \dot{\mathbf{n}}(\alpha)\> = 0,
$$
therefore the mean curvature is
$$
\begin{aligned}
H(\phi(\alpha,\beta)) &= \< \mathbf{n}(\alpha), \f{\partial_\alpha}{\|\partial_\alpha \phi(\alpha,\beta) \|}\left( \f{\partial_\alpha \phi(\alpha,\beta)}{\|\partial_\alpha \phi(\alpha,\beta)\|} \right)\> \\
&= \f{1}{\|\partial_\alpha \phi(\alpha,\beta)\|^2} \< \mathbf{n}(\alpha), \partial_\alpha^2 \phi(\alpha,\beta)\> \\
&= \left(1-\f{\beta}{R(\alpha)}\right)^{-2} \frac{1}{R(\alpha)} \left(1-\f{\beta}{R(\alpha)}\right) = \f{1}{R(\alpha)-\beta}.
\end{aligned} 
$$
Finally,
    \begin{equation}
    e^{-\int_0^\beta H(\phi(\alpha,\beta')) d\beta'} d\alpha d\beta= e^{-\int_0^\beta \f{1}{R(\alpha)-\beta'}d\beta'} d\alpha d\beta = \left(1-\f{\beta}{R(\alpha)}\right) d\alpha d\beta
    \end{equation}
and the Lebesgue measure of the support $\f{1}{4\sigma \varepsilon}$ is the normalization factor because the $\beta$ term vanishes when integrating over $\beta \in (-\sigma,\sigma)$.
\end{proof}

In this case the proposed transport map of \cref{proposed-transport-map} reduces to
$$
T(\phi(\alpha, \beta)) = \phi\left(\delta-\alpha, \beta- \f{1}{2}\f{\dot{R}(0)}{R(0)^2}(\sigma^2-\beta^2)(\delta-2\alpha)\right).
$$
As a consequence of \cref{asymptotically-optimal-transport}, 
\begin{lemma}
\label{good-approximation-map}
    For all $\delta, \varepsilon, \sigma >0$ sufficiently small with $\sigma \vee \varepsilon \leq \f{\delta}{4}$, it holds that
    $$
    W_1(\mu_{x_0}^{\sigma,\varepsilon}, \mu_y^{\sigma,\varepsilon}) = W_1(\mu^{\sigma,\varepsilon}_{x_0}, T_* \mu^{\sigma,\varepsilon}_{x_0}) + O(\delta^4).
    $$
\end{lemma}

\begin{figure}
\centering
\begin{tikzpicture}
    \path[fill=red!30, opacity=0.7] (0.9, 0.3) -- (1.1, 0.75) -- (1.3, 1.1) -- (0, 1.7) -- (-0.2,1.3) -- (-0.4, 0.8) -- cycle;

    \path[fill=red!30, opacity=0.7] (2, 2.1) -- (2.3, 2.35) -- (2.75, 2.55) -- (2.15, 3.85) -- (1.6, 3.65) -- (1.2, 3.3) -- cycle;
     
    \draw[thick] plot[smooth, tension=0.7] coordinates {(0,0) (1,2) (2,3) (3,3.5)};
    \node[fill, circle, scale=0.5, label=right:$x_0$] at (0.45, 1){};
    \node[fill, circle, scale=0.5, label=below:$y$] at (2, 3){};

    \draw[thick] plot[smooth, tension=1] coordinates {(0.9, 0.3) (1.1, 0.75) (1.3, 1.1)};
    \draw[thick] plot[smooth, tension=1] coordinates {(-0.4, 0.8) (-0.2, 1.3) (0, 1.7)};

    \draw[thick] plot[] coordinates {(0.9, 0.3) (-0.4, 0.8)};
    \draw[thick] plot[] coordinates {(1.3, 1.1) (0, 1.7)};

    \draw[thick] plot[smooth, tension=1] coordinates {(1.2,3.3) (1.6, 3.65) (2.15, 3.85)};
    \draw[thick] plot[smooth, tension=1] coordinates {(2,2.1) (2.3, 2.35) (2.75, 2.55)};

    \draw[thick] plot[] coordinates {(1.2, 3.3) (2, 2.1)};
    \draw[thick] plot[] coordinates {(2.15, 3.85) (2.75, 2.55)};

    \node[label=below:$\gamma$] at (0,0){};

    \node[label=below right:$\color{red}{\mu_{x_0}^{\sigma,\varepsilon}}$] at (0.9, 0.5){};

    \node[label=right:$\color{red}{\mu_{y}^{\sigma,\varepsilon}}$] at (2.75, 2.55){};

    \draw[thick, blue, ->] plot[] coordinates {(0, 1.7) (1.2, 3.3)};
    \draw[thick, blue, ->] plot[] coordinates {(-0.4, 0.8) (2.15, 3.85)};
    \draw[thick, blue, ->] plot[] coordinates {(0.9, 0.3) (2.75, 2.55)};
    \draw[thick, blue, ->] plot[] coordinates {(1.3, 1.1) (2, 2.1)};

    \node[label=left:$\color{blue}{T}$] at (0.8, 2.7){};

    \draw[dashed] (0.9, 0.15) -- (-0.5, 0.65);
    \draw[thick, <->] (0.2, 0.4) -- (-0.5, 0.65);

    \node[label=left:$\sigma$] at (0.2, 0.25){};

    \draw[dashed] (-0.6, 0.8) -- (-0.2, 1.8);
    \draw[thick, <->] (-0.4, 1.3) -- (-0.2, 1.8);

    \node[label=left:$\varepsilon$] at (-0.2, 1.6){};

\end{tikzpicture}
\caption{Planar curve case: test measures in red with some transport pairs of $T$ in blue.}
\label{planar-curve-fig}
\end{figure}
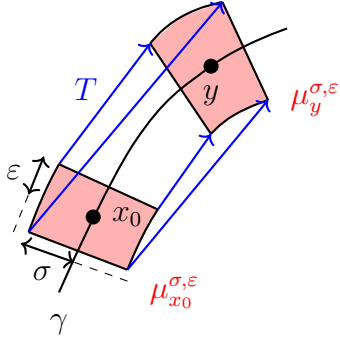

For notational ease we shall from here onwards denote $R:=R(0)$ and $\dot{R}:=\dot{R}(0)$.

\planarcoarsecurvature

\begin{proof}
\cref{good-approximation-map} allows computing $W_1(\mu^{\sigma,\varepsilon}_{x_0}, T_* \mu^{\sigma,\varepsilon}_{x_0})$ instead. Throughout the proof, terms of order $\delta^4$ and higher are absorbed into $O(\delta^4)$.
For the upper bound, we compute by expansion with respect to the orthonormal basis $(\dot{\gamma}(0), \mathbf{n}(0))$ at $x_0$,
\begin{equation}
\label{fermi-planar}
\begin{aligned}
    &\phi(\alpha, \beta) = \gamma(0) + \beta \mathbf{n}(0) + \alpha (\dot{\gamma}(0) + \beta \dot{\mathbf{n}}(0)) + \f{\alpha^2}{2} (\ddot{\gamma}(0) + \beta \ddot{\mathbf{n}}(0)) + \f{\alpha^3}{6} \dddot{\gamma}(0) + O(\delta^4) \\
    &= x_0 + \left( \alpha - \f{\alpha \beta}{R} - \f{\alpha^3}{6R^2} + \f{\beta \alpha^2 \dot{R}}{2R^2} \right) \dot{\gamma}(0) + \left(\beta + \f{\alpha^2}{2R} - \f{\alpha^3 \dot{R}}{6R^2} - \f{\beta \alpha^2}{2R^2} \right) \mathbf{n}(0) + O(\delta^4)
\end{aligned}
\end{equation}
having inserted for the derivatives at $0$ using the list \eqref{derivatives-list}.
Then the distance of the transport pairs up to order 4 is
\begin{equation}
\label{planar-case-exp}
\begin{aligned}
&\|T(\phi(\alpha, \beta))-\phi(\alpha,\beta)\| = \|\phi(\delta-\alpha, \beta) - \phi(\alpha,\beta)\|\\
&=(\delta-2\alpha)\bigg\|\left(1 - \f{\beta}{R}-\f{1}{6R^2} (\delta^2-\delta\alpha+\alpha^2) +\f{\beta \dot{R}}{2R^2}\delta + O(\delta^3) \right) \dot{\gamma}(0) \\
& \hspace{2.3cm} +  \left(\f{\delta}{2R} + O(\delta^2)\right) \mathbf{n}(0)\bigg\|
\end{aligned}
\end{equation}
having used the factorizations $(\delta-\alpha)^3-\alpha^3 = (\delta-2\alpha)(\delta^2-\delta\alpha+\alpha^2)$ and $(\delta-\alpha)^2 - \alpha^2 = \delta(\delta-\alpha)$. By orthonormality of $(\dot{\gamma}(0), \mathbf{n}(0))$, we compute this norm as
\begin{equation}
\label{planar-case-pairwise}
\begin{aligned}
&(\delta-2\alpha) \left[ \left( 1 - \f{\beta}{R}-\f{1}{6R^2} (\delta^2-\delta\alpha+\alpha^2) +\f{\beta \dot{R}}{2R^2}\delta + O(\delta^3) \right)^2 + \left(\f{\delta}{2R} + O(\delta^2) \right)^2 \right]^{\f{1}{2}}\\
&= (\delta-2\alpha) \left(1-\frac{2\beta}{R} + \f{\beta^2}{R^2} - \f{1}{3R^2} (\delta^2-\delta\alpha+\alpha^2) + \f{ \dot{R}}{R^2}\beta \delta +\f{1}{4R^2}\delta^2 \right)^{\f{1}{2}} + O(\delta^4) \\
&= (\delta-2\alpha) \left( 1 -\f{\beta}{R} - \f{\delta^2}{24R^2} + \f{\delta \alpha}{6R^2} - \f{\alpha^2}{6R^2}+ \f{ \dot{R}}{2R^2} \beta\delta \right) + O(\delta^4).
\end{aligned}
\end{equation}
by the expansion $\sqrt{1+x} = 1+\f{1}{2}x-\f{1}{8}x^2+O(x^3)$ for the square root on the last line.

Moreover, expanding the volume distortion factor as
$$
\left(1-\frac{\beta}{R(\alpha)}\right) = 1- \f{\beta}{R} + \f{\dot{R}}{R^2}\alpha \beta + O(\delta^3)
$$
and multiplying the expression for $\|T(\phi(\alpha, \beta))-\phi(\alpha,\beta)\|$ by this factor, we integrate and note that only terms of even order in both $\alpha$ and $\beta$ contribute, yielding
$$
\begin{aligned}
&W_1(\mu_{x_0}^{\sigma,\varepsilon}, \mu_y^{\sigma, \varepsilon}) \\
&\leq \f{1}{4\sigma \varepsilon} \int_{-\varepsilon}^\varepsilon d\alpha \int_{-\sigma}^\sigma d\beta \left(1- \f{\beta}{R}+\f{\dot{R}}{R^2}\alpha \beta+O(\delta^3)\right) \|T(\phi(\alpha, \beta))-\phi(\alpha,\beta)\| \\
&= \f{1}{4\sigma \varepsilon} \int_{-\varepsilon}^\varepsilon d\alpha \int_{-\sigma}^\sigma d\beta \left(1- \f{\beta}{R}+\f{\dot{R}}{R^2}\alpha \beta+O(\delta^3)\right)\\
&\qquad \times (\delta-2\alpha) \left( 1 -\f{\beta}{R} - \f{\delta^2}{24R^2} + \f{\delta \alpha}{6R^2} - \f{\alpha^2}{6R^2}+ \f{ \dot{R}}{2R^2} \beta\delta \right) + O(\delta^4)\\
&=\delta \left(1 - \f{\delta^2}{24R^2} -\f{\varepsilon^2}{6R^2} + \f{\sigma^2}{3R^2} \right) + O(\delta^4) \\
&= \|y-x_0\| \left( 1-\f{\varepsilon^2}{6R^2} + \f{\sigma^2}{3R^2} \right) + O(\delta^4).
\end{aligned}
$$
To obtain the factor $\|y-x_0\|$ on the last line, we applied that 
$$
\|y-x_0\| = \delta \left(1-\f{\delta^2}{24R^2}\right)+O(\delta^4)
$$
which can be deduced by plugging in for $\alpha=\beta=0$ in the previous computation of $T(\phi(\alpha,\beta))-\phi(\alpha,\beta)$.
The $\sigma^2$ coefficient came from integrating the $\beta^2$ term of the integrand, $\f{\sigma^2}{3R^2} = \frac{1}{2\sigma}\int_{-\sigma}^\sigma (-\f{\beta}{R}) \times  (-\f{\beta}{R}) d\beta$. The terms with odd power in $\alpha$ or $\beta$ such as $\delta\alpha\beta$ vanished as they are mean zero.

We proceed with showing the lower bound, using again the 1-Lipschitz test function
$$
f(z) := \< z-x_0, \f{y-x_0}{\|y-x_0\|} \>.
$$
Express the vector between the centres of the two test measures, recalling $\gamma(0)=x_0$,
$$
\begin{aligned}
\gamma(\delta)-\gamma(0) &= T(\phi(0,0)) - \phi(0,0) \\
&= \delta \left(1- \f{\delta^2}{6R^2}\right) \dot{\gamma}(0) - \delta \left(\f{\delta}{2R} + O(\delta^2) \right) \mathbf{n}(0)  + O(\delta^4).
\end{aligned}
$$
This vector has magnitude
$$
\|\gamma(\delta)-\gamma(0)\| = \delta \left( 1-\f{\delta^2}{24R^2} \right) +O(\delta^4),
$$
and so we deduce that
$$
\begin{aligned}
\f{y-x_0}{\|y-x_0\|} &= \f{\gamma(\delta)-\gamma(0)}{\|\gamma(\delta)-\gamma(0)\|} = \left(1-\f{\delta^2}{8R^2} + O(\delta^3)\right)\dot{\gamma}(0) + \left(\f{\delta}{2R} + O(\delta^2) \right)\mathbf{n}(0).
\end{aligned}
$$
Then we compute, using the expression \eqref{planar-case-exp} for $ T(\phi(\alpha,\beta)) - \phi(\alpha,\beta)$ obtained above,
$$
\begin{aligned}
    &f(Tz)-f(z)\\
    &=  \< T(\phi(\alpha,\beta)) - \phi(\alpha,\beta), \f{y-x_0}{\|y-x_0\|} \> \\
    &= (\delta-2\alpha)\left(1 - \f{\beta}{R}-\f{1}{6R^2} (\delta^2-\delta\alpha+\alpha^2) +\f{\beta \delta \dot{R}}{2R^2}\right) \left(1-\f{\delta^2}{8R^2}+O(\delta^3)\right) \\
    & \quad + (\delta-2\alpha) \left(\f{\delta}{2R}  + O(\delta^2)\right)\left(\f{\delta}{2R} +O(\delta^2)\right) +O(\delta^4)\\
    &= (\delta -2\alpha) \left(1 - \f{\beta}{R} -\f{\delta^2}{24R^2} +\f{\delta\alpha}{6R^2}  -\f{\alpha^2}{6R^2} +\f{\beta \delta \dot{R}}{2R^2}\right) + O(\delta^4).
\end{aligned}
$$
We see that this agrees with the pairwise transport distance \eqref{planar-case-pairwise} up to $O(\delta^4)$, hence \cref{test-function-gradient} applies and the upper and lower bounds agree up to an $O(\delta^4)$ term.
\end{proof}

\subsection{Space curve}
Let $\gamma : (-\delta_0,\delta_0) \rightarrow \R^3$ be a smooth, unit speed curve with velocity $\dot{\gamma}$. Define the unit normal and binormal vector fields along $\gamma$ as
$$
\mathbf{n}(\alpha) := \f{\ddot{\gamma}(\alpha)}{\|\ddot{\gamma}(\alpha) \|}, \quad \mathbf{b}(\alpha) := \f{\dot{\gamma}(\alpha) \times \mathbf{n}(\alpha)}{\|\dot{\gamma}(\alpha) \times \mathbf{n}(\alpha)\|}.
$$
This yields the so-called Frenet-Serret frame $(\dot{\gamma}(\alpha), \mathbf{n}(\alpha), \mathbf{b}(\alpha))$ of $\R^3$ along $\gamma$. 
Writing $R(\alpha):=\f{1}{\|\ddot{\gamma}(\alpha)\|}$ for the radius of the osculating circle and $\tau(\alpha):= \|\dot{\mathbf{b}}(\alpha)\|$ for the torsion, 
the Frenet-Serret formulas give relationships between the vector fields of the frame,
\begin{equation}
\label{frenet-serret-first-derivatives}
\begin{aligned}
\ddot{\gamma}(\alpha) &= \frac{\mathbf{n}(\alpha)}{R(\alpha)}, \\
\dot{\mathbf{n}}(\alpha) &= -\f{\dot{\gamma}(\alpha)}{R(\alpha)} + \tau(\alpha)\mathbf{b}(\alpha), \\
\dot{\mathbf{b}}(\alpha) &= -\tau(\alpha) \mathbf{n}(\alpha).
\end{aligned}
\end{equation}
From these, we deduce the higher order derivatives
\begin{equation}
\label{frenet-serret-second-derivatives}
\begin{aligned}
\dddot{\gamma}(\alpha) & = -\f{1}{R(\alpha)^2} \dot{\gamma}(\alpha) - \f{\dot{R}(\alpha)}{R(\alpha)^2}\mathbf{n}(\alpha) + \f{\tau(\alpha)}{R(\alpha)} \mathbf{b}(\alpha),\\
\ddot{\mathbf{n}}(\alpha) &= \f{\dot{R}(\alpha)}{R(\alpha)^2}\dot{\gamma}(\alpha) - \left(\tau(\alpha)^2+\f{1}{R(\alpha)^2}\right) \mathbf{n}(\alpha) + \dot{\tau}(\alpha)\mathbf{b}(\alpha),\\
\ddot{\mathbf{b}}(\alpha) &= \f{\tau(\alpha)}{R(\alpha)}\dot{\gamma}(\alpha) -\dot{\tau}(\alpha)\mathbf{n}(\alpha)-\tau(\alpha)^2 \mathbf{b}(\alpha).
\end{aligned}
\end{equation}
We will employ the Frenet-Serret frame for explicit computations of distances between points in the tubular neighborhood of a space curve. Additionally, we will employ it in formulating a sufficiently accurate approximate transport map between test measures,  represented through an expansion in Fermi coordinates.

\cref{definition-fermi-coordinates} for Fermi coordinates requires a choice of a local orthonormal frame of the normal bundle along $\gamma$. We choose $(\mathbf{n}_1,\mathbf{n}_2)$ as follows:
$$
\begin{aligned}
\mathbf{n}_1(\alpha) &:= \f{\mathbf{n}(\alpha) - \alpha\tau(\alpha) \mathbf{b}(\alpha)}{\sqrt{1+\alpha^2\tau(\alpha)^2}}  = (1+O(\alpha^2))\mathbf{n}(\alpha) - (\alpha\tau(\alpha)+O(\alpha^2)) \mathbf{b}(\alpha)),\\
\mathbf{n}_2(\alpha) &:= \f{\mathbf{b}(\alpha) + \alpha\tau(\alpha) \mathbf{n}(\alpha)}{\sqrt{1+\alpha^2\tau(\alpha)^2}} = (1+O(\alpha^2))\mathbf{b}(\alpha) + (\alpha\tau(\alpha)+O(\alpha^2)) \mathbf{n}(\alpha)
\end{aligned}
$$
where $\mathbf{n}, \mathbf{b}$ come from the Frenet-Serret frame.
\begin{definition}
    Define the Fermi coordinates $\phi: (-\delta_0,\delta_0)^3 \rightarrow \R^3$, adapted to $\gamma$, by the formula:
    $$
    \begin{aligned}
    \phi(\alpha, \beta_1, \beta_2) :&= \gamma(\alpha) +\beta_1 \mathbf{n}_1(\alpha) + \beta_2 \mathbf{n}_2(\alpha)\\
    &= \gamma(\alpha) + (\beta_1 + \alpha \beta_2 \tau(\alpha)+O(\delta^3)) \mathbf{n}(\alpha) + (\beta_2 - \alpha \beta_1 \tau(\alpha)+O(\delta^3)) \mathbf{b}(\alpha).
    \end{aligned}
    $$
    Denote $R=R(0), \dot{R}=\dot{R}(0), \tau=\tau(0), \dot{\tau}=\dot{\tau}(0)$.
\end{definition}

    Consider the family of curves
    $
    \{\alpha \mapsto \phi(\alpha,\beta_1,\beta_2) : (\beta_1,\beta_2) \in B_\sigma\}.
    $
    Denote the particular unit normal vector fields $$
    \tilde{\mathbf{n}}(\alpha, \beta_1,\beta_2) := \f{1}{\sqrt{\beta_1^2+\beta_2^2}} \left(\beta_1 \mathbf{n}_1(\alpha) + \beta_2 \mathbf{n}_2(\alpha)\right).
    $$
    The mean curvature of each curve in the direction $\tilde{\mathbf{n}}$ is expressed as
    $$
    \begin{aligned}
    \<H(\phi(\alpha, \beta_1,\beta_2)), \tilde{\mathbf{n}}(\alpha,\beta_1,\beta_2)\> &= \<\tilde{\mathbf{n}}(\alpha, \beta_1,\beta_2), \f{\partial_\alpha}{\|\partial_\alpha \phi(\alpha,\beta_1,\beta_2)\|} \left( \f{\partial_\alpha \phi(\alpha,\beta_1,\beta_2)}{\|\partial_\alpha \phi(\alpha,\beta_1,\beta_2)\|} \right)\>\\
    &= \<\tilde{\mathbf{n}}(\alpha, \beta_1,\beta_2), \f{\partial_\alpha^2 \phi(\alpha,\beta_1,\beta_2)}{\|\partial_\alpha \phi(\alpha,\beta_1,\beta_2)\|^2} \>.
    \end{aligned}
    $$
    where the second equality holds because $\tilde{\mathbf{n}}$ is normal to $\partial_\alpha \phi$ by \cref{normal-to-leaf}.

\begin{remark}
\label{space-curve-test-measure}

We perform computations in terms of the Frenet-Serret frame as it can be interpreted in terms of the radius of the osculating circle and torsion of the curve.
\begin{itemize}
\item [(1)]   As a special case of \cref{test-measures-in-fermi-coordinates}, using that the mean curvature components are
    $$
    \begin{aligned}
    H^1(\phi(\mathbf{0})) &= \<\mathbf{n}(0), \ddot{\gamma}(0)\> = \<\mathbf{n}(0), \f{\mathbf{n}(0)}{R(0)}\> = \f{1}{R(0)},\\ 
    H^2(\phi(\mathbf{0})) &= \<\mathbf{b}(0), \ddot{\gamma}(0)\> = \<\mathbf{b}(0), \f{\mathbf{n}(0)}{R(0)}\> = 0,
    \end{aligned}
    $$
    the test measures at every $y = \gamma(\delta)$ in Fermi coordinates along $\gamma$ are
    $$
    \begin{aligned}
    &(\phi^{-1}_* \mu_y^{\sigma,\varepsilon})(d\alpha,d\beta_1,d\beta_2) \\
    &= \f{\mathbbm{1}_{\tilde{B}_{\sigma,\varepsilon}}(\alpha,\beta_1,\beta_2) }{\int_{\tilde{B}_{\sigma,\varepsilon}} (1+ r(\alpha,\beta_1,\beta_2)) d (\phi^{-1}_* \mu_y^{\sigma,\varepsilon})(\alpha,\beta_1,\beta_2)}  \bigg(1- \f{\beta_1}{R(0)} + r(\alpha,\beta_1,\beta_2)\bigg) d\alpha d\beta_1 d\beta_2
    \end{aligned}
    $$
    where $r(\alpha,\beta_1,\beta_2) = O(\delta^2)$ is the remainder.

    \item [(2)] The proposed transport map of \cref{proposed-transport-map} reduces, in this case, to
    $$
    \begin{aligned}
    T(\phi(\alpha,\beta_1,\beta_2)) = \phi(\delta-\alpha, \beta_1 + O(\delta^3), \beta_2 + O(\delta^3)).
    \end{aligned}
    $$
    These expressions will be used in the proof of the next theorem.
\end{itemize}
\end{remark}

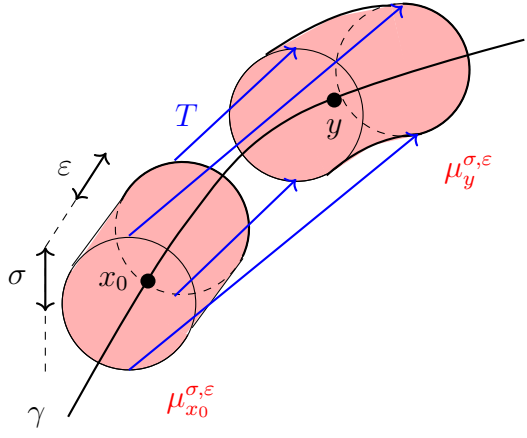
\begin{figure}
    \centering
    \begin{tikzpicture}
        \path[fill=red!30, opacity=0.7] (0.03, 1.9) -- (0.75,2.95) -- (2.3,2.1) -- (1.55,1.1) -- cycle;
        \centerarc[fill=red!30, opacity=0.7](1.5,2.5)(-28:150:25pt)
        \centerarc[fill=red!30, opacity=0.7](0.8,1.5)(-210:-28:25pt)

        \path[fill=red!30, opacity=0.7] (3.4, 3.2) -- (4, 3.6) -- (4.65, 3.75) -- (4.25, 5.45) -- (3.5, 5.3) -- (2.6, 4.8) -- cycle;
        \centerarc[fill=red!30, opacity=0.7](4.4,4.6)(-75:100:25pt)
        \centerarc[fill=red!30, opacity=0.7](3,4)(-245:-63:25pt)

        \node[label=left:$\gamma$] at (0,0){};
            
        \draw[thick] plot[smooth, tension=0.6] coordinates {(0,0) (1.5,2.5) (3, 4) (6,5)};

        \draw[] (0.8,1.5) circle (25pt);

        \node[fill, circle, scale=0.5, label=left:$x_0$] at (1.05, 1.8){};

        \draw[dashed] (1.5,2.5) circle (25pt);

        \centerarc[thick](1.5,2.5)(-25:155:25pt)

        \draw (0.05, 2) -- (0.75, 2.95);
        \node[label=below right:$\color{red}{\mu_{x_0}^{\sigma,\varepsilon}}$] at (1, 0.7){};
        \draw (1.63, 1.18) -- (2.3, 2.1);

        \draw (3,4) circle (25pt);

        \draw[dashed] (4.4, 4.6) circle (25pt);

        \centerarc[thick](4.4,4.6)(-90:90:25pt)

        \node[fill, circle, scale=0.5, label=below:$y$] at (3.5, 4.2){};

        \draw[thick] plot[smooth, tension=1] coordinates {(3.4, 3.2) (4, 3.6) (4.65, 3.75)};

        \draw[thick] plot[smooth, tension=1] coordinates {(2.6, 4.8) (3.5, 5.3) (4.25, 5.45)};

        \node[label=below right:$\color{red}{\mu_{y}^{\sigma,\varepsilon}}$] at (4.65, 3.75){};

        \draw[thick, blue, ->] plot[] coordinates {(0.8, 0.62) (4.6, 3.75)};
        \draw[thick, blue, ->] plot[] coordinates {(1.4, 1.6) (3, 3.15)};
        \draw[thick, blue, ->] plot[] coordinates {(0.8, 2.4) (4.4, 5.45)};
        \draw[thick, blue, ->] plot[] coordinates {(1.4, 3.4) (3, 4.9)};

         \node[label=left:$\color{blue}{T}$] at (2, 4){};

    \draw[dashed] (-0.3, 0.6) -- (-0.3, 2.25);
    \draw[thick, <->] (-0.3, 1.4) -- (-0.3, 2.25);

    \node[label=left:$\sigma$] at (-0.25, 1.85){};

    \draw[dashed] (-0.3,2.25) -- (0.5, 3.5);

    \draw[thick, <->] (0.1, 2.875) -- (0.5, 3.5);

    \node[label=left:$\varepsilon$] at (0.35, 3.3){};
        
    \end{tikzpicture}
    \caption{Space curve case: test measures in red with some transport pairs of $T$ in blue.}
\end{figure}

\begin{theorem}
\label{space-curve}
    Let $\gamma: (-\delta_0,\delta_0) \rightarrow \R^3$ be a space curve with $x_0 = \gamma(0), y=\gamma(\delta)$ and $\mu_{x_0}^{\sigma,\varepsilon}, \mu_y^{\sigma,\varepsilon}$ the test measures defined in \cref{def-test-measures} with coordinate representation of \cref{space-curve-test-measure}. For all $\delta, \sigma, \varepsilon >0$ sufficiently small and with $\sigma \vee \varepsilon \leq \f{\delta}{4}$, it holds that
    $$
    \begin{aligned}
    W_1(\mu_{x_0}^{\sigma,\varepsilon}, \mu_{y}^{\sigma,\varepsilon}) 
    &=\|x_0-y\| \left(1 + \f{\sigma^2}{4R^2} - \f{\varepsilon^2}{6R^2} \right) +O(\delta^4).
    \end{aligned}
    $$
    where $R=\f{1}{\|\ddot{\gamma}(0)\|}$ is the radius of the osculating circle.
\end{theorem}
\begin{proof}
Due to \cref{good-approximation-map}, it is sufficient to work with the distance $W_1(\mu_{x_0}^{\sigma,\varepsilon}, T_* \mu_{x_0}^{\sigma,\varepsilon})$ as it approximates $W_1(\mu_{x_0}^{\sigma,\varepsilon}, \mu_{y}^{\sigma,\varepsilon})$.
The computation of the pairwise distances is similar to the planar curve case, \eqref{planar-case-exp}, with additional terms due to the component $\mathbf{b}(\alpha)$. Concretely, since
$$
\begin{aligned}
\mathbf{b}(\alpha) &= \mathbf{b}(0) + \alpha \dot{\mathbf{b}}(0) + \f{1}{2}\alpha^2 \ddot{\mathbf{b}}(0) + O(\alpha^3) \\
&= \f{\alpha^2 \tau}{2R} \dot{\gamma}(0) - \left(\alpha \tau + \f{1}{2}\alpha^2 \dot{\tau}\right) \mathbf{n}(0) + \left(1 - \f{1}{2}\alpha^2 \tau^2 \right) \mathbf{b}(0) + O(\alpha^3),
\end{aligned}
$$
and using the derivatives \eqref{frenet-serret-first-derivatives} and \eqref{frenet-serret-second-derivatives}, compute
\begin{equation}
\label{fermi-space-curve}
\begin{aligned}
    \phi(\alpha,\beta_1,\beta_2)
    &= x_0 + \left( \alpha - \f{\alpha \beta_1}{R} - \f{\alpha^3}{6R^2} + \f{\beta_1 \alpha^2 \dot{R}}{2R^2} + \f{\beta_2 \alpha^2 \tau}{2R} + O(\delta^4) \right) \dot{\gamma}(0) \\
    & \qquad + \left(\beta_1 + \f{\alpha^2}{2R} + O(\delta^3) \right) \mathbf{n}(0) + \left( \beta_2 + O(\delta^3) \right) \mathbf{b}(0).
\end{aligned}
\end{equation}
Then similarly to \eqref{planar-case-exp} we obtain
\begin{equation}
\label{space-curve-pointwise-distance}
\begin{aligned}
&\|T(\phi(\alpha, \beta_1,\beta_2))-\phi(\alpha,\beta_1,\beta_2)\| \\
&= \bigg\| \gamma(\delta-\alpha) + \beta_1 \mathbf{n}(\delta-\alpha) + \beta_2 \mathbf{b}(\delta-\alpha) -\gamma(\alpha)-\beta_1 \mathbf{n}(\alpha) - \beta_2 \mathbf{b}(\alpha) \bigg\|\\
&=(\delta-2\alpha) \bigg\| \left(1 - \f{\beta_1}{R}-\f{1}{6R^2} (\delta^2-\delta\alpha+\alpha^2) +\f{\beta_1 \dot{R}}{2R^2}\delta + \f{\beta_2 \delta \tau}{2R} + O(\delta^3) \right) \dot{\gamma}(0) \\
&\hspace{2.5cm} + \left(\f{1}{2R} \delta + O(\delta^2) \right) \mathbf{n}(0)  +O(\delta^2) \mathbf{b}(0) \bigg\| \\
&= (\delta-2\alpha) \left( 1 -\f{\beta_1}{R} - \f{\delta^2}{24R^2} + \f{\delta \alpha}{6R^2} - \f{\alpha^2}{6R^2} + \f{\beta_1 \dot{R}}{2R^2}\delta + \f{\beta_2 \delta \tau}{2R}  + O(\delta^3) \right).
\end{aligned}
\end{equation}
The Wasserstein distance upper bound is then computed by integration with respect to $\mu_{x_0}^{\sigma,\varepsilon}$ using the coordinate representation of \cref{space-curve-test-measure} as
$$
\begin{aligned}
W_1(\mu_{x_0}^{\sigma,\varepsilon}, T_*\mu_{x_0}^{\sigma,\varepsilon}) & \leq \int_{\tilde{B}_{\sigma,\varepsilon}} \|T(\phi(\alpha,\beta_1,\beta_2)) - \phi(\alpha,\beta_1,\beta_2)\| d(\phi^{-1}_* \mu_{x_0}^{\sigma,\varepsilon})(\alpha,\beta_1,\beta_2)\\
&= \delta \left(1 - \f{\delta^2}{24R^2} + \f{\sigma^2}{4R^2} -\f{\varepsilon^2}{6R^2} \right) +O(\delta^4)\\
&= \|x_0-y\| \left(1 + \f{\sigma^2}{4R^2} -\f{\varepsilon^2}{6R^2} \right) +O(\delta^4)
\end{aligned}
$$
applying on the last line that $\|x_0-y\| = \delta\left( 1 - \f{\delta^2}{24R^2}\right) +O(\delta^4)$. In the integral on the first line, terms of odd order vanish upon integration, and the remaining terms amount to integration of quadratic polynomials.

We now address the lower bound. 
Analogously to the plane curve case, define the test function for the Kantorovich-Rubinstein duality  as
$$
f(\phi(\alpha,\beta_1,\beta_2)) := \< \phi(\alpha,\beta_1,\beta_2)- x_0, \f{y-x_0}{\|y-x_0\|} \>
$$
which is again clearly 1-Lipschitz in $\R^3$. We wish to apply \cref{test-function-gradient} to show the lower bound and upper bound coincide up to $O(\delta^4)$.
Noting that $y-x_0 = \phi(\delta,0,0)-\phi(0,0,0)$, we deduce from \eqref{space-curve-pointwise-distance} that
$$
\begin{aligned}
y-x_0 &= \delta \left( 1- \f{\delta^2}{6R^2} + O(\delta^3) \right) \dot{\gamma}(0) + \delta \left(\f{\delta}{2R} + O(\delta^2) \right) \mathbf{n}(0) + O(\delta^3) \mathbf{b}(0),\\
\|y-x_0\| &= \delta \left( 1 - \f{\delta^2}{24R^2} + O(\delta^3) \right),
\end{aligned}
$$
and so
$$
\begin{aligned}
\f{y-x_0}{\|y-x_0\|} &= \left(1-\f{\delta^2}{8R^2} + O(\delta^3) \right)\dot{\gamma}(0) + \left(\f{\delta}{2R} + O(\delta^2)\right)\mathbf{n}(0) + O(\delta^2)\mathbf{b}(0).
\end{aligned}
$$
Therefore
$$
\begin{aligned}
    &f(\phi(T(\alpha,\beta_1,\beta_2))) - f(\alpha,\beta_1,\beta_2) \\
    &= \< T(\phi(\alpha,\beta_1,\beta_2)) -\phi(\alpha,\beta_1,\beta_2),  \f{y-x_0}{\|y-x_0\|} \> \\
    &= (\delta-2\alpha) \left( 1 -\f{\beta_1}{R} - \f{\delta^2}{24R^2} + \f{\delta \alpha}{6R^2} - \f{\alpha^2}{6R^2} + \f{\beta_1 \delta \dot{R}}{2R^2} + \f{\beta_2 \delta \tau}{2R}  + O(\delta^3) \right) .
\end{aligned}
$$
This is the same expression as for $\|T(\phi(\alpha, \beta_1,\beta_2))-\phi(\alpha,\beta_1,\beta_2)\|$, hence \cref{test-function-gradient} applies and the lower and upper bounds agree up to $O(\delta^4)$.
\end{proof}

\subsection{Surface}
\label{surface-case}
We now consider a smooth 2-surface $M \subset \R^3$ and $\gamma: (-1,1) \rightarrow M$ a unit speed geodesic in $M$, denoting again $x_0 := \gamma(0), y:= \gamma(\delta)$ for $\delta >0$ sufficiently small. Let $\mathbf{n} \in \Gamma(TM^\perp)$ be the unit normal vector field and $\mathbf{m} \in \Gamma(TM|_\gamma)$ the unit vector field along $\gamma$ orthogonal to the velocity $\dot{\gamma}$. Both $\mathbf{n}$ and $\mathbf{m}$ are unique up to sign.

\begin{figure}
    \centering
    \begin{tikzpicture}
    \draw[thick] (-3,0.5) -- (-4,-4);
    \draw[thick] plot[smooth, tension=0.8] coordinates {(-4,-4) (2, -3) (8, -4)};
    \draw[thick] (9,0.5) -- (8,-4);
    \draw[thick] plot[smooth, tension=0.8] coordinates {(-3,0.5) (3, 1.5) (9, 0.5)};

    \draw[thick] plot[smooth, tension=0.8] coordinates {(-2,-1.5) (0.3, -0.9) (5, -1) (7, -1.5)};

    \draw[thick, blue, ->] (0.2, -0.9) -- (1.2, -0.75);
    \node[label=above right:$\color{blue}{\dot{\gamma}(0)}$] at (1.2,-0.9){};
    \draw[thick, green, ->] (0.2, -0.9) -- (0.4, 0);
    \node[label=right:$\color{green}{\mathbf{m}(0)}$] at (0.4, 0){};
    \draw[thick, red, ->] (0.2, -0.9) -- (0.2, 0.1);
    \node[label=left:$\color{red}{\mathbf{n}(0,0)}$] at (0.2, 0.1){};

    \node[fill, circle, scale=0.5, label=below:$x_0$] at (0.2, -0.9){};

    \draw[dashed, green] (4.9, -1) -- (5.3, 0.6);
    \draw[thick, green, ->] (4.9, -1) -- (5.1, -0.2);
    \node[label=right:$\color{green}{\mathbf{m}(\alpha_1)}$] at (5.1, -0.2){};

    \node[fill, circle, scale=0.5, label=below:$\gamma(\alpha_1)$] at (4.9, -1){};

     \node[fill, circle, scale=0.5, label=right:${\psi(\alpha_1, \alpha_2)}$] at (5.3, 0.6){};
     
    \draw[dashed, red] (5.3, 0.6) -- (5.7, 2.4);
    \draw[thick, red, ->] (5.3, 0.6) -- (5.5, 1.5);
    \node[label=right:$\color{red}{\mathbf{n}(\alpha_1,\alpha_2)}$] at (5.5, 1.5){};

    \node[fill, circle, scale=0.5, label=right:${\phi(\alpha_1,\alpha_2,\beta)}$] at (5.7, 2.4){};

    \node[label=below:$M$] at (8,-4){};

    \node[label=above left:$\gamma$] at (-2,-2){};
    
    \end{tikzpicture}
    \caption{Fermi coordinates along $\gamma$ adapted to the surface $M$ embedded in $\R^3$.}
    \label{fermi-coordinates-surface-fig}
\end{figure}
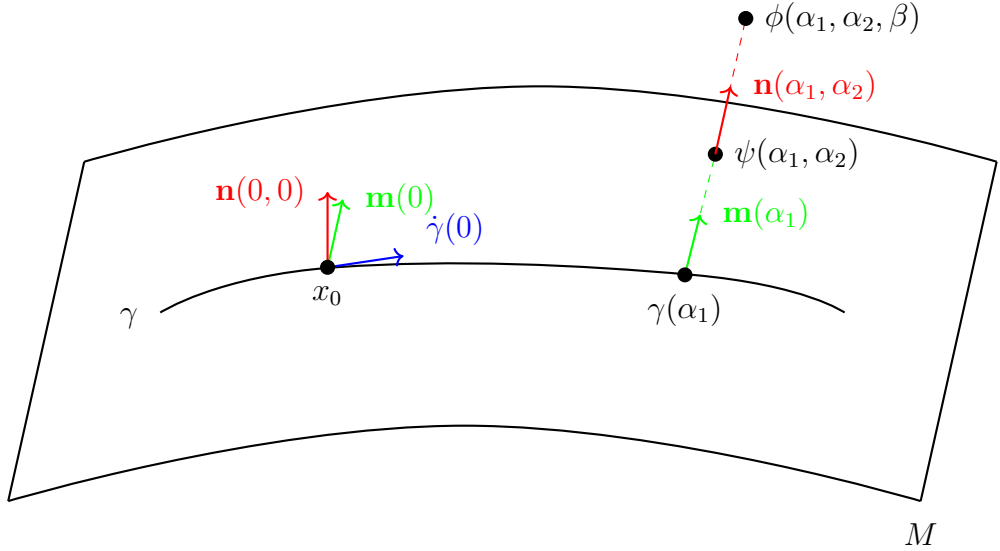

\begin{definition}
\begin{itemize}
\item Define the Fermi coordinates $\psi: (-\delta_0,\delta_0) \times (-\varepsilon_0,\varepsilon_0) \rightarrow M$ along $\gamma$ in $M$ as
$$
\psi(\alpha_1,\alpha_2) = \exp_{M,\gamma(\alpha_1)}(\alpha_2 \mathbf{m}(\alpha_1)).
$$

\item Define the Fermi coordinates $\phi: (-\delta_0,\delta_0) \times (-\varepsilon_0, \varepsilon_0) \times (-\sigma_0, \sigma_0) \rightarrow \R^3$ along $\gamma$ in $\R^3$ adapted to the surface $M$ as
$$
\phi(\alpha_1,\alpha_2,\beta) = \psi(\alpha_1,\alpha_2) + \beta \mathbf{n}(\alpha_1,\alpha_2).
$$
See \cref{fermi-coordinates-surface-fig} for a graphical representation of $\psi$ and $\phi$.

\item For $i,j \in \{1,2\}$ denote the components of the second fundamental form in the Fermi coordinates
$$
\II_{ij}(\alpha) =\<\mathbf{n}(\alpha), \partial_{\alpha_i} \partial_{\alpha_j} \phi(\alpha,\mathbf{0})\>.
$$
\end{itemize}
\end{definition}

\begin{remark}
We point out that we overload the second fundamental form symbol $\II$ depending on the context of use. In the notation \eqref{sff} and in the statements of \cref{surface-coarse-curvature} and \cref{general-submanifold}, the subscript is the point $x_0$ on the manifold and the bracket arguments are tangent vectors. On the other hand, in coordinate computations taking place in the proofs, the subscripts will represent components with respect to the Fermi frame at Fermi coordinates $\alpha,\beta$ in brackets.
\end{remark}

Similarly to the Frenet-Serret frame in the case of a planar curve, we now consider the orthonormal frame $(\dot{\gamma}, \mathbf{m},\mathbf{n})$ with the intent to expand at $x_0$, i.e. $\alpha_1=\alpha_2=\beta=0$. 

\begin{lemma}
\label{derivative-catalogue}
The first derivatives of the normal vector field at $(\alpha_1,\alpha_2) =\mathbf{0}$ are
\begin{equation}
\label{derivative-of-normal}
\partial_{\alpha_1}\mathbf{n}(\mathbf{0}) = -\II_{11}(\mathbf{0}) \dot{\gamma}(0)- \II_{12}(\mathbf{0})\mathbf{m}(0), \quad \partial_{\alpha_2}\mathbf{n}(\mathbf{0}) = -\II_{22}(\mathbf{0})\mathbf{m}(0)- \II_{12}(\mathbf{0}) \dot{\gamma}(0).
\end{equation}
Hence the derivatives of $\phi$ at $(\alpha_1,\alpha_2,\beta) =\mathbf{0}$ up to third order are
$$
\begin{aligned}
\partial_\beta \phi(\mathbf{0}) &= \mathbf{n}(\mathbf{0}), \quad \partial_\beta^k \phi(\mathbf{0}) = 0 \ \textrm{ for } k \geq 2, \\
\partial_{\alpha_1} \phi(\mathbf{0}) &= \dot{\gamma}(0), \quad
\partial_{\alpha_2} \phi(\mathbf{0}) = \mathbf{m}(0), \\
\partial_{\alpha_i} \partial_{\alpha_j} \phi(\mathbf{0}) &=  \II_{ij}(\mathbf{0}) \mathbf{n}(\mathbf{0}), \quad i,j \in \{1,2\}, \\
\partial_\beta \partial_{\alpha_i} \phi(\mathbf{0}) &= -\II_{1i}(\mathbf{0}) \dot{\gamma}(0) - \II_{i2} (\mathbf{0}) \mathbf{m}(0), \quad i \in \{1,2\},\\
\partial_{\alpha_1}^3 \phi(\mathbf{0}) &= -\II_{11}(\mathbf{0})^2 \dot{\gamma}(0) - \II_{11}(\mathbf{0}) \II_{12}(\mathbf{0}) \mathbf{m}(0) + \partial_{\alpha_1} \II_{11}(\mathbf{0}) \mathbf{n}(\mathbf{0}),\\
\partial_{\alpha_1}\partial_{\alpha_2}^2 \phi(\mathbf{0}) &=  - \II_{11}(\mathbf{0}) \II_{22}(\mathbf{0}) \dot{\gamma}(0) - \II_{12}(\mathbf{0}) \II_{22}(\mathbf{0}) \mathbf{m}(0) +  \partial_{\alpha_1} \II_{22}(\mathbf{0}) \mathbf{n}(\mathbf{0}),\\
\partial_{\alpha_2}\partial_{\alpha^1}^2 \phi(\mathbf{0}) &= - \II_{11}(\mathbf{0}) \II_{12}(\mathbf{0}) \dot{\gamma}(0) - \II_{12}(\mathbf{0})^2 \mathbf{m}(0)  + \partial_{\alpha_2} \II_{11}(\mathbf{0}) \mathbf{n}(\mathbf{0}),\\
\partial_{\alpha_2}^3 \phi(\mathbf{0}) &= - \II_{12}(\mathbf{0}) \II_{22}(\mathbf{0}) \dot{\gamma}(0) - \II_{22}(\mathbf{0})^2 \mathbf{m}(0) + \partial_{\alpha_2} \II_{22}(\mathbf{0}) \mathbf{n}(\mathbf{0}).
\end{aligned}
$$
\end{lemma}

\begin{proof}
The derivatives involving $\partial_\beta$ are clear, recalling the definition
$$
\phi(\alpha_1,\alpha_2,\beta) := \psi(\alpha_1,\alpha_2) + \beta \mathbf{n}(\alpha_1,\alpha_2),
$$
and the first derivatives in $\alpha_1,\alpha_2$ follow from the definition of $\psi(\alpha_1,\alpha_2)$.

For $\partial_{\alpha_1} \mathbf{n}(\mathbf{0})$ we check its components with respect to the frame $(\dot{\gamma}, \mathbf{m}, \mathbf{n})$,
    $$
    \begin{aligned}
    \langle \partial_{\alpha_1} \mathbf{n}(\mathbf{0}), \mathbf{n}(\mathbf{0}) \rangle &= \f{1}{2} \partial_{\alpha_1} \< \mathbf{n}, \mathbf{n} \>(\mathbf{0}) = 0, \\
    \<\partial_{\alpha_1} \mathbf{n}(\mathbf{0}), \dot{\gamma}(0) \> &= \partial_{\alpha_1} \<\mathbf{n}, \partial_{\alpha_1} \psi\>(\mathbf{0}) - \< \mathbf{n}(\mathbf{0}), \partial_{\alpha_1}^2 \psi(\mathbf{0})\> =  - \II_{11}(\mathbf{0}), \\
    \<\partial_{\alpha_1} \mathbf{n}(\mathbf{0}), \mathbf{m}(0) \> &= \partial_{\alpha_1} \<\mathbf{n}, \partial_{\alpha_2} \psi\>(\mathbf{0})  - \< \mathbf{n}(\mathbf{0}), \partial_{\alpha_1} \partial_{\alpha_2}\psi (\mathbf{0}) \>= -\II_{12}(\mathbf{0})\\
    \end{aligned}
    $$
    and similarly for $\partial_{\alpha_2} \mathbf{n}(\mathbf{0})$.
    
    For the second derivatives in $\alpha_1,\alpha_2$ at $(\alpha_1,0,0)$ for any $\alpha_1 \in (-\delta_0,\delta_0)$ and $j=1,2$,
    $$
    \begin{aligned}
    \partial_{\alpha_1} \partial_{\alpha_j} \phi (\alpha_1,0,0) &= \partial_{\alpha_1} \partial_{\alpha_j} \psi (\alpha_1,0) \\
    &= \nabla^{\R^3}_{\partial_{\alpha_1}} \partial_{\alpha_j} \psi(\alpha_1,0) - \nabla^M_{\partial_{\alpha_1}} \partial_{\alpha_j} \psi (\alpha_1,0)
    \\
    &= \II(\partial_{\alpha_1}\psi(\alpha_1,0), \partial_{\alpha_j} \psi(\alpha_1,0))
    = \II_{1j}(\alpha_1,0) \mathbf{n}(\alpha_1,0)
    \end{aligned}
    $$
    having introduced the term $\nabla^M_{\partial_{\alpha_1}} \partial_{\alpha_j} \psi(\alpha_1,0)$, which vanishes for $j=1,2$, because $\alpha_1 \mapsto \psi(\alpha_1,0)$ is a geodesic on $M$ and $\mathbf{m}(\alpha_1)$ is the parallel translation of $\mathbf{m}(0)$ along $\gamma$. By the same argument, for any $(\alpha_1, \alpha_2) \in (-\delta_0,\delta_0) \times (-\varepsilon_0,\varepsilon_0)$,
    $$
    \partial_{\alpha_2}^2 \phi (\alpha_1,\alpha_2,0) = \II_{22}(\alpha_1,\alpha_2)\mathbf{n}(\alpha_1,\alpha_2)
    $$
    because $\alpha_2 \mapsto \psi(\alpha_1,\alpha_2)$ is a geodesic for every $\alpha_1 \in (-\delta_0,\delta_0)$.
    
    For the second derivatives in $\beta$ and one of $\alpha_1$ and $\alpha_2$, deduce $\partial_\beta \partial_{\alpha_i} \phi(\mathbf{0}) = \partial_{\alpha_i} \mathbf{n}(\mathbf{0})$
    and plug in for $\partial_{\alpha_i} \mathbf{n}(\mathbf{0})$.
    
    For the third derivatives at $(\alpha_1,\alpha_2,\beta) =\mathbf{0}$, write by the chain rule
    $$
    \begin{aligned}
    \partial_{\alpha_1}^3 \phi (\mathbf{0}) &= \partial_{\alpha_1}|_{\alpha_1=0} (\partial_{\alpha_1}^2 \phi(\alpha_1,0,0)) = \II_{11}(\mathbf{0}) \partial_{\alpha_1}\mathbf{n}(\mathbf{0}) + \partial_{\alpha_1}\II_{11}(\mathbf{0})\mathbf{n}(\mathbf{0}), \\
    \partial_{\alpha_1}\partial_{\alpha_2}^2 \phi (\mathbf{0}) &= \partial_{\alpha_1}|_{\alpha_1=0} (\partial_{\alpha_2}^2 \phi(\alpha_1,0,0))= \II_{22}(\mathbf{0}) \partial_{\alpha_1} \mathbf{n}(\mathbf{0}) + \partial_{\alpha_1} \II_{22}(\mathbf{0}) \mathbf{n}(\mathbf{0}),\\
    \partial_{\alpha_2}\partial_{\alpha_1}^2 \phi (\mathbf{0}) &= \partial_{\alpha_1}|_{\alpha_1=0} (\partial_{\alpha_1} \partial_{\alpha_2} \phi(\alpha_1,0,0)) =\II_{12}(\mathbf{0}) \partial_{\alpha_1} \mathbf{n}(\mathbf{0}) + \partial_{\alpha_1} \II_{12}(\mathbf{0}) \mathbf{n}(\mathbf{0}), \\
    \partial_{\alpha_2}^3 \phi(\mathbf{0}) & = \partial_{\alpha_2}|_{\alpha_2=0} (\partial_{\alpha_2}^2 \phi(0,\alpha_2,0)) = \II_{22}(\mathbf{0}) \partial_{\alpha_2} \mathbf{n}(0) + \partial_{\alpha_2} \II_{22}(\mathbf{0}) \mathbf{n}(\mathbf{0})
    \end{aligned}
    $$
    and plug in for $\partial_{\alpha_i} \mathbf{n}(\mathbf{0})$ in each.
\end{proof}

Denote $D_u V$ the plain derivative in the direction $u \in \R^n$ of a vector field $V$ as a smooth map from an open subset of $\R^n$ to $\R^n$.

\begin{notation}
Denote $\tilde{B}_{\sigma,\varepsilon} := \{(\alpha_1,\alpha_2,\beta): \alpha_1^2+\alpha_2^2 < \varepsilon^2, |\beta| < \sigma \} \subset \R^3$.

Consider the family of surfaces
$
\{ \phi(U, \beta) : \beta \in (-\sigma_0, \sigma_0) \}.
$
For any $\beta \in (-\sigma_0,\sigma_0)$, we denote the unit normal vector field of the surface as $\mathbf{n}(\alpha_1,\alpha_2)$, which is unique up to sign. The corresponding mean curvature is:
$$
H(\phi(\alpha_1,\alpha_2,\beta)) = \< \mathbf{n}(\alpha_1,\alpha_2), \sum_{i=1}^2 \nabla^{\R^3}_{e_i} e_i(\phi(\alpha_1,\alpha_2,\beta)) \>
$$
where $(e_1,e_2)$ is an orthonormal frame on each $\phi(U,\beta)$.
\end{notation}

\begin{remark}
\label{surface-measures-coordinates}
\begin{itemize}
    \item [(1)] As a special case of \cref{test-measures-in-fermi-coordinates}, the test measures at $y=\gamma(\delta)$ in these Fermi coordinates are
$$
\label{test-measure-surface}
\begin{aligned}
&(\phi^{-1}_* \mu_{y}^{\sigma,\varepsilon})(d\alpha_1,d\alpha_2,d\beta) \\
&= \f{\mathbbm{1}_{\tilde{B}_{\sigma,\varepsilon}}(\delta + \alpha_1,\alpha_2,\beta)}{\int_{\tilde{B}_{\sigma,\varepsilon}} 1+r(\alpha,\beta) d(\phi^{-1}_* \mu_y^{\sigma,\varepsilon})(\alpha,\beta)}   \bigg( 1 - \beta (\II_{11}(\mathbf{0}) + \II_{22}(\mathbf{0})) + r(\alpha,\beta) \bigg) d\alpha_1 d\alpha_2 d\beta
\end{aligned}
$$
where $r(\alpha,\beta)=O(\delta^2)$ is a second order remainder.

\item [(2)] The proposed transport map of \cref{proposed-transport-map} reduces, in this case, to
$$
\begin{aligned}
T(\phi(\alpha_1,\alpha_2,\beta)) &= \phi(\delta-\alpha_1, \alpha_2 + O(\delta^3), \beta +O(\delta^3)).
\end{aligned}
$$
See Figures \ref{surface-case-fig}, \ref{top-down-perspective} and \ref{cross-sectional-perspective} for a pictorial representation of this map.
\end{itemize}
\end{remark}

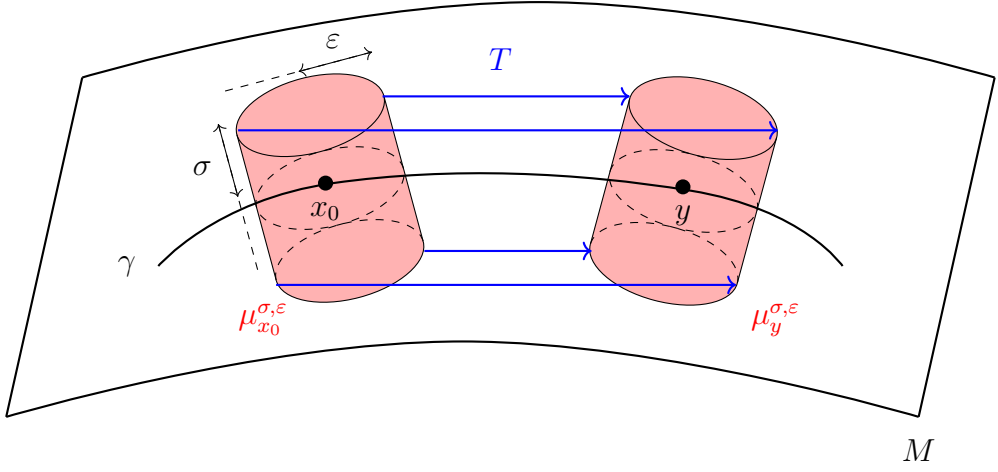
\begin{figure}
    \centering
    \begin{tikzpicture}
    \fill [rotate=15, fill=red!30, opacity=0.7] (-1,0) -- (-1,-2) arc (180:360:1 and 0.5) -- (1,0) arc (0:180:1 and 0.5);
    \fill [rotate=-15, fill=red!30, opacity=0.7] (4,1.3) -- (4,-0.7) arc (180:360:1 and 0.5) -- (6,1.3) arc (0:180:1 and 0.5);

    \draw[thick] plot[smooth, tension=0.8] coordinates {(-2,-2) (0.3, -0.9) (5, -1) (7, -2)};

    \node[label=left:$\gamma$] at (-2,-2){};
    
    \draw[rotate=15] (0,0) ellipse (1 and 0.5);
    \draw[rotate=15, dashed] (0,-1) ellipse (1 and 0.5);
    \draw[rotate=15] (-1,0) -- (-1,-2);
    \draw[rotate=15] (-1,-2) arc (180:360:1 and 0.5);
    \draw[rotate=15, dashed] (-1,-2) arc (180:360:1 and -0.5);
    \draw[rotate=15] (1,-2) -- (1,0);

    \draw[rotate=-15] (5,1.3) ellipse (1 and 0.5);
    \draw[rotate=-15, dashed] (5,0.3) ellipse (1 and 0.5);
    \draw[rotate=-15] (4,1.3) -- (4,-0.7);
    \draw[rotate=-15] (4,-0.7) arc (180:360:1 and 0.5);
    \draw[rotate=-15, dashed] (4,-0.7) arc (180:360:1 and -0.5);
    \draw[rotate=-15] (6,-0.7) -- (6,1.3);  

    \draw[thick] (-3,0.5) -- (-4,-4);
    \draw[thick] plot[smooth, tension=0.8] coordinates {(-4,-4) (2, -3) (8, -4)};
    \draw[thick] (9,0.5) -- (8,-4);
    \draw[thick] plot[smooth, tension=0.8] coordinates {(-3,0.5) (3, 1.5) (9, 0.5)};

    \node[fill, circle, scale=0.5, label=below:$x_0$] at (0.2, -0.9){};

    \node[fill, circle, scale=0.5, label=below:$y$] at (4.9, -0.95){};

    \node[label=below:$M$] at (8,-4){};

    \node[label=below left:$\color{red}{\mu_{x_0}^{\sigma,\varepsilon}}$] at (0,-2.2){};
    \node[label=below right:$\color{red}{\mu_{y}^{\sigma,\varepsilon}}$] at (5.5,-2.2){};

    \draw[thick, blue, ->] plot[] coordinates {(-0.45, -2.25) (5.6, -2.25)};
    \draw[thick, blue, ->] plot[] coordinates {(1.5, -1.8) (3.7, -1.8)};
    \draw[thick, blue, ->] plot[] coordinates {(0.95, 0.25) (4.2, 0.25)};
    \draw[thick, blue, ->] plot[] coordinates {(-0.95, -0.2) (6.15, -0.2)};

    \node[label=above:$\color{blue}{T}$] at (2.5, 0.3){};

    \draw[rotate=15, dashed] (-1.2, 0.2) -- (-1.2, -1.8);
    \draw[rotate=15, <->] (-1.2,0.2) -- (-1.2,-0.8);

    \node[label=left:$\sigma$] at (-1, -0.7){};

    \draw[rotate=15, dashed] (-1, 0.6) -- (1, 0.6);
    \draw[rotate=15, <->] (0, 0.6) -- (1, 0.6);

    \node[label=above:$\varepsilon$] at (0.3, 0.6){};
    
\end{tikzpicture}
    \caption{Test measures in red with some transport pairs of $T$ in blue.}
    \label{surface-case-fig}
\end{figure}

\begin{figure}
    \centering
    \begin{tikzpicture}[scale=0.9]

        \draw (-4.5, -3.5) -- (9.5, -3.5);
        \draw (9.5, -3.5) -- (9.5, 3.5);
        \draw (-4.5, -3.5) -- (-4.5, 3.5);
        \draw (-4.5, 3.5) -- (9.5, 3.5);

        \node[label=above left:$M$] at (9.5, -3.5){};

        \draw (0, 0) circle [radius=2cm];
        \node[fill, circle, scale=0.5, label=below left:$x_0$] at (0, 0){};

        \draw (5, 0) circle [radius=2cm];
        \node[fill, circle, scale=0.5, label=below left:$y$] at (5, 0){};

        \draw[thick] (-3,0) -- (8,0);
        \draw[dashed, ->] (-4,0) -- (9,0);
        \draw[dashed] (5,-3) -- (5,3);
        \node[label=below left:$\gamma$] at (8,0){};
        \node[label=below:$\alpha_1$] at (9,0){};

        \draw[dashed, ->] (0, -3) -- (0, 3);
        \node[label=left:$\alpha_2$] at (0,3){};

        \draw[orange] (1, -1.75) -- (1, 1.75);
        \draw[orange] (4, -1.75) -- (4, 1.75);

        \draw[blue, ->] (1, 1) -- (4, 1);
        \node[label=above:$\color{blue}{T}$] at (2.5,1){};

        \draw[blue, ->] (-1, -1) -- (6, -1);
        \node[label=below:$\color{blue}{T}$] at (2.5,-1){};

        \draw[green] (-1, -1.75) -- (-1, 1.75);
        \draw[green] (6, -1.75) -- (6, 1.75);

    \end{tikzpicture} 
    \caption{Top-down perspective for the transport map $T$.}
    \label{top-down-perspective}
\end{figure}
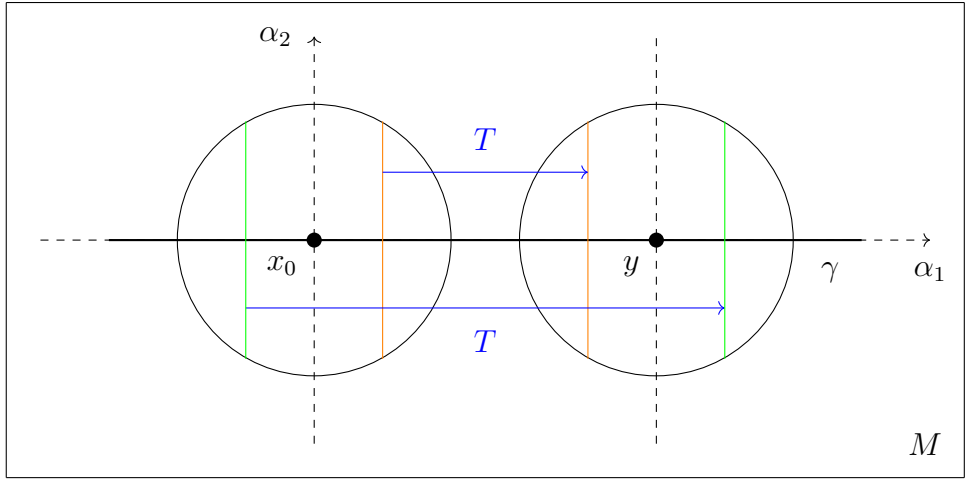

\begin{figure}
    \centering
    \begin{tikzpicture}[scale=0.9]

    \draw[dashed, ->] (0, 1.2) -- (0.4, -1.2);
    \draw[dashed] (0, 1.2) -- (-0.4, 3.6);
    \node[label=right:$\beta$] at (0.4, -1.2){};
        
    \draw[thick] plot[smooth, tension=0.8] coordinates {(-3,0.5) (3, 1.5) (9, 0.5)};
    \draw[dashed, ->] (9, 0.5) -- (10, 0.2);
    \draw[dashed] (-3, 0.5) -- (-4, 0.2);

    \node[label=above:$\gamma \subset M$] at (9,0.5){};
    \node[label=below:$\alpha_1$] at (10,0.2){};

    \draw (-1,-0.5) -- (-1.8, 2.3);

    \draw (1.5,0) -- (1.2, 2.8);

    \draw plot[smooth, tension=0.8] coordinates {(-1, -0.5) (0.25, -0.175) (1.5, 0)};

    \draw plot[smooth, tension=0.8] coordinates {(-1.8, 2.3) (-0.3, 2.65) (1.2, 2.8)};

    \draw (7,-0.5) -- (7.8, 2.3);

    \draw (4.5,0) -- (4.8, 2.8);

    \draw plot[smooth, tension=0.8] coordinates {(7, -0.5) (5.75, -0.175) (4.5, 0)};

    \draw plot[smooth, tension=0.8] coordinates {(7.8, 2.3) (6.3, 2.65) (4.8, 2.8)};

    \node[fill, circle, scale=0.5, label=below left:$x_0$] at (0, 1.2){};
    \node[fill, circle, scale=0.5, label=below:$y$] at (6, 1.2){};

    \draw[red] plot[smooth, tension=0.8] coordinates {(-1.6, 1.6) (-0.1625, 1.93) (1.275, 2.1)};

    \draw[red] plot[smooth, tension=0.8] coordinates {(7.6, 1.6) (6.1625, 1.93) (4.725, 2.1)};

    \draw[blue, ->] (-0.1625, 1.93) -- (6.1625, 1.93);
    \node[label=above:$\color{blue}{T}$] at (3, 1.93){};

    \draw[brown] plot[smooth, tension=0.8] coordinates {(-1.2, 0.2) (0.1125, 0.525) (1.425, 0.7)};

    \draw[brown] plot[smooth, tension=0.8] coordinates {(7.2, 0.2) (6.1125, 0.5) (4.575, 0.7)};

    \draw[blue, ->] (0.1125, 0.525) -- (6.1125, 0.5);
    \node[label=below:$\color{blue}{T}$] at (3, 0.5){};

    \end{tikzpicture}
    \caption{Cross-sectional perspective for the transport map $T$.}
    \label{cross-sectional-perspective}
\end{figure}
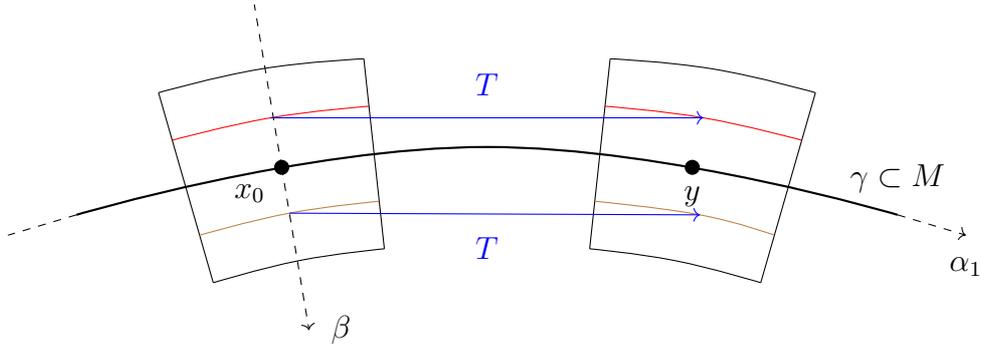

\newpage

\begin{theorem}
\label{surface-coarse-curvature}
Let $M$ be an isometrically embedded surface in $\R^3$, let $x_0$ be a point and $(e_1,e_2)$ an orthonormal basis of principal curvature directions at $x_0$. Let $\gamma$ be a unit speed geodesic in $M$ with $\gamma(0)=x_0$, $\dot{\gamma}(0)=e_1$ and denote $y=\gamma(\delta)$. 
For all $\delta, \varepsilon, \sigma >0$ sufficiently small with $\sigma \vee \varepsilon \leq \f{\delta}{4}$, it holds that
$$
\begin{aligned}
W_1(\mu_{x_0}^{\sigma,\varepsilon}, \mu_{y}^{\sigma,\varepsilon})
&= \|y-x_0\| \left( 1 +\left(\f{\sigma^2}{3} - \f{\varepsilon^2}{8}\right) \<\II_{x_0}(e_1,e_1), H(x_0)\>\right)+ O(\delta^4).
\end{aligned}
$$
\end{theorem}

\begin{remark}
\label{pseudo-flat}
If we set $\varepsilon=\f{2\sqrt{2}}{\sqrt{3}}\sigma$, we note that the bracket on the right reduces to 1. This is due to the effects of second fundamental form and the curvature of the submanifold cancelling out, so it would appear in such special case that the coarse extrinsic curvature is flat, even though the second fundamental form may be non-vanishing. Such a special case is due to having an additional degree of freedom because of the additional $\sigma$ parameter and the sign of the $\sigma^2$ term happens to oppose that of the $\varepsilon^2$ term. The extrinsic curvature should thus be seen as encapsulated by varying both $\sigma$ and $\varepsilon$ in $W_1(\mu_{x_0}^{\sigma,\varepsilon}, \mu_y^{\sigma,\varepsilon})$.
\end{remark}

\begin{proof}
The conclusion of \cref{asymptotically-optimal-transport} holds, so we may compute $W_1(\mu_{x_0}^{\sigma,\varepsilon}, T_* \mu_{x_0}^{\sigma,\varepsilon})$ instead. 
For every point $\phi(\alpha_1,\alpha_2,\beta)$, expanding up to third order and using the list of derivatives of \cref{derivative-catalogue}, we collect terms as components of the frame $(\dot{\gamma}, \mathbf{m},\mathbf{n})$ at $\mathbf{0}$,
\begin{align}
&\phi(\alpha_1,\alpha_2,\beta) \nonumber\\
&= x_0 + \sum_i \alpha_i \partial_{\alpha_i} \phi(\mathbf{0}) +\beta \partial_\beta \phi(\mathbf{0}) + \f{1}{2} \sum_{i,j} \alpha_i \alpha_j \partial_{\alpha_i}\partial_{\alpha_j} \phi(\mathbf{0}) + \sum_i \beta \alpha_i \partial_\beta \partial_{\alpha_i} \phi(\mathbf{0}) \nonumber\\
& \quad + \f{1}{6} \sum_{i,j,k} \alpha_i \alpha_j \alpha_k \partial_{\alpha_i} \partial_{\alpha_j} \partial_{\alpha_k} \phi(\mathbf{0}) + \f{1}{2} \sum_{i,j} \beta \alpha_i \alpha_j \partial_{\alpha_i} \partial_{\alpha_j} \partial_\beta \phi (\mathbf{0}) + O(\delta^4) \nonumber \\
&= x_0 + \bigg(\alpha_1 - \II_{11}(\mathbf{0}) \beta \alpha_1 - \f{1}{6}\II_{11}(\mathbf{0})^2 \alpha_1^3 - \f{1}{2}\II_{11}(\mathbf{0}) \II_{22}(\mathbf{0}) \alpha_1 \alpha_2^2  + O(\delta^4) \bigg) \dot{\gamma}(0) \nonumber \\
\label{phi-fermi}
&\qquad + \left( \alpha_2 - \II_{22}(\mathbf{0}) \beta \alpha_2 + O(\delta^3)\right) \mathbf{m}(0) \\
&\qquad + \left( \beta + \f{1}{2} \II_{11}(\mathbf{0}) \alpha_1^2 + \f{1}{2} \II_{22}(\mathbf{0}) \alpha_2^2 + O(\delta^3) \right) \mathbf{n}(0) \nonumber \\
&\qquad + \f{1}{2} \sum_{i,j} \beta \alpha_i \alpha_j \partial_\beta \partial_{\alpha_i} \partial_{\alpha_j} \phi(\mathbf{0}). \nonumber
\end{align}
While the terms $\beta \alpha_i \alpha_j \partial_\beta \partial_{\alpha_i} \partial_{\alpha_j} \phi (\mathbf{0})$ are only of order 3, they are linear in $\beta$, and hence will not influence the integral with respect to $\mu_{x_0}^{\sigma, \varepsilon}$ up to $O(\delta^4)$.
In the same way an expression for the proposed transport 
$$
T(\phi(\alpha_1,\alpha_2,\beta)) = \phi(\delta-\alpha_1,\alpha_2+O(\delta^3),\beta+O(\delta^3))
$$
can be obtained by making corresponding substitutions for the components in the above expression for $\phi(\alpha_1,\alpha_2,\beta)$. Then the pointwise transport vector is 
\begin{equation}
\label{pointwise-transport-surface-vector}
\begin{aligned}
&T(\phi(\alpha_1,\alpha_2,\beta)) - \phi(\alpha_1,\alpha_2,\beta)  \\
&= (\delta-2\alpha_1) \\
& \quad \bigg[ \bigg(1 - \II_{11}(\mathbf{0}) \beta - \f{1}{6} \II_{11}(\mathbf{0})^2 (\delta^2-\delta\alpha_1 +\alpha_1^2) - \f{1}{2} \II_{11}(\mathbf{0}) \II_{22}(\mathbf{0}) \alpha_2^2 + O(\delta^3) \bigg) \dot{\gamma}(0) \\
&\hspace{1cm} + O(\delta^2) \mathbf{m}(0) + \left(\f{1}{2} \II_{11}(\mathbf{0}) \delta + O(\delta^2) \right) \mathbf{n}(0) \\
&\hspace{1cm} + \beta \alpha_2 \partial_\beta \partial_{\alpha_1} \partial_{\alpha_2} \phi(\mathbf{0}) + \f{1}{2} \beta \delta \partial_\beta \partial_{\alpha_1}^2 \phi(\mathbf{0}) + O(\delta^3) \bigg]
\end{aligned}
\end{equation}
and its magnitude is
\begin{equation}
\label{pointwise-transport-surface-magnitude}
\begin{aligned}
&\|T(\phi(\alpha_1,\alpha_2,\beta)) - \phi(\alpha_1,\alpha_2,\beta)\| \\
&= (\delta-2\alpha_1) \bigg( 1 - \II_{11}(\mathbf{0}) \beta - \f{1}{6} \II_{11}(\mathbf{0})^2 (\delta^2-\delta\alpha_1 +\alpha_1^2) -\f{1}{2} \II_{11}(\mathbf{0}) \II_{22}(\mathbf{0}) \alpha_2^2 \\
& \hspace{1.5cm} + \f{1}{8} \II_{11}(\mathbf{0})^2 \delta^2 + \beta \alpha_2 \<\partial_{\alpha_1} \partial_{\alpha_2} \mathbf{n}(\mathbf{0}), \dot{\gamma}(0) \> + \f{1}{2} \beta \delta \< \partial_{\alpha_1}^2 \mathbf{n}(\mathbf{0}), \dot{\gamma}(0)\>\bigg) + O(\delta^4).
\end{aligned}
\end{equation}
Using the density of the test measure $\mu_{x_0}^{\sigma,\varepsilon}$ in Fermi coordinates given by \cref{test-measure-surface}, the upper bound is
$$
\begin{aligned}
    &W_1(\mu_{x_0}^{\sigma,\varepsilon}, T_*\mu_{x_0}^{\sigma,\varepsilon}) \\
    &\leq \int_{B_{\sigma,\varepsilon}} \|Tz-z\| \mu_{x_0}^{\sigma,\varepsilon}(dz) \\
    &= \int_{\tilde{B}_{\sigma,\varepsilon}} \|T(\phi(\alpha,\beta))-\phi(\alpha,\beta)\|  (\phi^{-1}_* \mu_{x_0}^{\sigma,\varepsilon}) (d\alpha, d\beta) + O(\delta^4)\\
    &= \delta \bigg( 1 - \f{1}{24} \II_{11}(\mathbf{0})^2 \delta^2 + \left(\f{1}{3}\II_{11}(\mathbf{0})^2 +\f{1}{3}\II_{11}(\mathbf{0}) \II_{22}(\mathbf{0}) \right) \sigma^2 \\
    &\hspace{1.5cm} - \f{1}{8}\left(\II_{11}(\mathbf{0})^2  + \II_{11}(\mathbf{0}) \II_{22}(\mathbf{0})\right) \varepsilon^2 \bigg) + O(\delta^4) \\
    &= \|x_0-y\| \bigg( 1 + \f{1}{3} \left(\II_{11}(\mathbf{0})^2 + \II_{11}(\mathbf{0}) \II_{22}(\mathbf{0}) \right) \sigma^2 \\
    &\hspace{2.7cm} - \f{1}{8}\left(\II_{11}(\mathbf{0})^2  + \II_{11}(\mathbf{0}) \II_{22}(\mathbf{0}) \right) \varepsilon^2 \bigg) + O(\delta^4).
\end{aligned}
$$
In the third equality, we plugged in for $\|T(\phi(\alpha,\beta)) -\phi(\alpha,\beta)\|$ as computed above and used that terms of odd order in one of $\alpha_1,\alpha_2,\beta$ integrate to 0 and again absorbed higher order terms into $O(\delta^4)$. On the last line, we used that
$$
\|x_0-y\| = \delta \left(1 - \f{1}{24} \II_{11}(\mathbf{0})^2 \delta^2 \right) + O(\delta^4).
$$
We proceed with showing the lower bound. Define
$$
p(\alpha_2,\beta) :=\f{\phi(\delta,\alpha_2,\beta)-\phi(0,\alpha_2,\beta)}{\|\phi(\delta,\alpha_2,\beta)-\phi(0,\alpha_2,\beta)\|}
$$
and the test function
\begin{equation}
\label{test-function-definition}
f(\phi(\alpha_1,\alpha_2,\beta)) := \< \phi(\alpha_1,\alpha_2,\beta) - x_0, p(\alpha_2,\beta)\>
\end{equation}
for the Kantorovich-Rubinstein duality, with the intention of applying \cref{test-function-gradient} to conclude. We first expand
\begin{equation}
\label{projection-vector}
\begin{aligned}
&\phi(\delta,\alpha_2,\beta) -\phi(0,\alpha_2,\beta) \\
&= \delta \left[ \bigg( 1 - \II_{11}(\mathbf{0}) \beta - \f{1}{6} \II_{11}(\mathbf{0}) \delta^2 - \f{1}{2} \II_{11}(\mathbf{0}) \II_{22}(\mathbf{0}) \alpha_2^2 \right.\\
&\qquad \quad + \f{1}{3} \beta \delta \<\partial_{\alpha_1}^2 \mathbf{n}(\mathbf{0}), \dot{\gamma}(0)\> + \f{1}{3} \beta \alpha_2 \<\partial_{\alpha_1} \partial_{\alpha_2} \mathbf{n}(\mathbf{0}), \dot{\gamma}(0)\> + O(\delta^3) \bigg) \dot{\gamma}(0) \\
& \qquad \quad + \left. O(\delta^2) \mathbf{m}(0) + \left( \f{1}{2} \II_{11}(\mathbf{0}) \delta  + O(\delta^2) \right) \mathbf{n}(0)\right].
\end{aligned}
\end{equation}
Deduce 
$$
\begin{aligned}
& \|\phi(\delta,\alpha_2,\beta) - \phi(0,\alpha_2,\beta)\| \\
& = \delta \bigg( 1 - \II_{11}(\mathbf{0}) \beta - \f{1}{24} \II_{11}(\mathbf{0})^2 \delta^2 - \f{1}{2} \II_{11}(\mathbf{0}) \II_{22}(\mathbf{0}) \alpha_2^2 \\
& \qquad + \f{1}{3} \beta \delta \<\partial_{\alpha_1}^2 \mathbf{n}(\mathbf{0}), \dot{\gamma}(0)\> + \f{1}{3} \beta \alpha_2 \<\partial_{\alpha_1} \partial_{\alpha_2} \mathbf{n}(\mathbf{0}), \dot{\gamma}(0)\> \bigg) + O(\delta^4),
\end{aligned}
$$
and therefore
\begin{equation}
\label{normalized-projection-vector}
\begin{aligned}
p(\alpha_2,\beta) &= \left(1- \f{1}{8} \II_{11}(\mathbf{0})^2 \delta^2 + O(\delta^3)\right)\dot{\gamma}(0) \\
& \quad + O(\delta^2) \mathbf{m}(0) + \left( \f{1}{2} \II_{11}(\mathbf{0}) \delta + O(\delta^2)\right) \mathbf{n}(0).
\end{aligned}
\end{equation}
Then it can be verified using expansions \eqref{pointwise-transport-surface-vector} and \eqref{normalized-projection-vector} to compute the inner product that
$$
\begin{aligned}
f(Tz)-f(z)
&= \< T(\phi(\alpha_1,\alpha_2,\beta))-\phi(\alpha_1,\alpha_2,\beta),p(\alpha_2,\beta)\> \\
&= \|T(\phi(\alpha_1,\alpha_2,\beta)) -\phi(\alpha_1,\alpha_2,\beta)\| + O(\delta^4).
\end{aligned}
$$
by comparison with \eqref{pointwise-transport-surface-magnitude}.

It remains to show that the magnitude of the gradient of $f$ satisfies
\begin{equation}
\label{gradient-magnitude}
\sup_{z \in B_{2\delta}(x_0)}\|\nabla f(z)\| = 1 + O(\delta^3).
\end{equation}
For this we need to expand the inverse matrix of the metric in Fermi coordinates. Using the expansion \eqref{phi-fermi}, compute
$$
\begin{aligned}
    \partial_{\alpha_1}\phi(\alpha_1,\alpha_2,\beta) 
    &= \bigg(1-\II_{11}(\mathbf{0}) \beta - \f{1}{2} \II_{11}(\mathbf{0})^2 \alpha_1^2 - \f{1}{2} \II_{11}(\mathbf{0}) \II_{22}(\mathbf{0}) \alpha_2^2 \\
    &\qquad + \f{1}{3} \beta \alpha_1 \< \partial_{\alpha_1}^2 \mathbf{n}(\mathbf{0}), \dot{\gamma}(0)\> + \f{1}{3} \beta \alpha_2 \< \partial_{\alpha_1} \partial_{\alpha_2} \mathbf{n}(\mathbf{0}), \dot{\gamma}(0)\> + O(\delta^3)\bigg)\dot{\gamma}(0) \\
    &\quad + O(\delta^2) \mathbf{m}(0) + \left( \II_{11}(\mathbf{0}) \alpha_1 + O(\delta^2) \right) \mathbf{n}(\mathbf{0}), \\
    \partial_{\alpha_2}\phi(\alpha_1,\alpha_2,\beta) 
    &= O(\delta^2) \dot{\gamma}(0) + \left(1- \II_{22}(\mathbf{0}) \beta + O(\delta^2) \right) \mathbf{m}(0)  + \left(\II_{22}(\mathbf{0}) \alpha_2 + O(\delta^2)\right) \mathbf{n}(\mathbf{0}),\\
    \partial_\beta \phi (\alpha_1,\alpha_2,\beta) 
    &= -\left( \II_{11}(\mathbf{0}) \alpha_1 + O(\delta^2)\right) \dot{\gamma}(0) - \left(\II_{22}(\mathbf{0}) \alpha_2 + O(\delta^2)\right) \mathbf{m}(0) \\
    &\quad + (1+ O(\delta^2))\mathbf{n}(\mathbf{0}).
    \end{aligned}
$$
We shall label the term
$$
r(\alpha) := \f{1}{3} \alpha_1 \< \partial_{\alpha_1}^2 \mathbf{n}(\mathbf{0}), \dot{\gamma}(0)\> + \f{1}{3} \alpha_2 \< \partial_{\alpha_1} \partial_{\alpha_2} \mathbf{n}(\mathbf{0}), \dot{\gamma}(0)\>.
$$
Then the metric matrix has the shape
$$
G = 
\begin{pmatrix}
    g_{11} & g_{12} & 0 \\
    g_{21} & g_{22} & 0 \\
    0 & 0 & 1
\end{pmatrix}
$$
with
$$
\begin{aligned}
g_{11} &= 1 - 2\II_{11}(\mathbf{0}) \beta + \II_{11}(\mathbf{0})^2 \beta^2 - \II_{11}(\mathbf{0}) \II_{22}(\mathbf{0}) \alpha_2^2 + 2\beta r(\alpha) + O(\delta^3), \\
g_{22} &= 1 - 2 \II_{22}(\mathbf{0}) \beta + O(\delta^2),\\
g_{12} &= O(\delta^2).
\end{aligned}
$$
Note that the matrix is of the form
$$
G = I + A
$$
with $A=O(\delta)$, which means the expansion of its inverse is
$$
G^{-1} = I-A +A^2 + O(\delta^3).
$$
We compute
$$
A^2 = 4\beta^2 
\begin{pmatrix}
    \II_{11}(\mathbf{0})^2  & 0 & 0\\
    0 & \II_{22}(\mathbf{0})^2 & 0 \\
    0 & 0 & 0
\end{pmatrix}
+ O(\delta^3),
$$
and thus
$$
\begin{aligned}
g^{11} &= 1 + 2 \II_{11}(\mathbf{0}) \beta + 3\II_{11}(\mathbf{0})^2 \beta^2 + \II_{11}(\mathbf{0}) \II_{22}(\mathbf{0})\big) \alpha_2^2 - 2\beta r(\alpha) + O(\delta^3),\\
g^{12} &= O(\delta^2),\\
g^{22} &= 1 + 2 \beta \II_{22}(\mathbf{0}) + O(\delta^2).
\end{aligned}
$$
From \eqref{normalized-projection-vector} we deduce the derivatives of the projection vector field in coordinates are
$$
\begin{aligned}
\partial_{\alpha_2} p(\alpha_2,\beta) &= O(\delta^2)\dot{\gamma}(0) +O(\delta) \mathbf{m}(0) + O(\delta) \mathbf{n}(0),\\
\partial_\beta p(\alpha_2,\beta) &= O(\delta^2)\dot{\gamma}(0) + O(\delta) \mathbf{m}(0) +O(\delta) \mathbf{n}(0).
\end{aligned}
$$
Then the first derivatives of the test function defined in \eqref{test-function-definition} are
$$
\begin{aligned}
    \partial_{\alpha_1} (f \circ \phi)(\alpha,\beta) &= \<\partial_{\alpha_1} \phi(\alpha_1,\alpha_2,\beta), p(\alpha_2,\beta)\>\\
    &= 1- \II_{11}(\mathbf{0}) \beta - \f{1}{2} \II_{11}(\mathbf{0})^2 \alpha^2_1 - \f{1}{2} \II_{11}(\mathbf{0}) \II_{22}(\mathbf{0}) \alpha_2^2 \\
    & \quad - \f{1}{8} \II_{11}(\mathbf{0})^2 \delta^2 + \f{1}{2} \II_{11}(\mathbf{0})^2 \delta\alpha_1 + \beta r(\alpha) + O(\delta^3),\\
    \partial_{\alpha_2} (f \circ \phi)(\alpha,\beta) &= \< \partial_{\alpha_2}\phi(\alpha_1,\alpha_2,\beta), p(\alpha_2,\beta)\> \\
    &\quad + \< \phi(\alpha_1,\alpha_2,\beta) - x_0, \partial_{\alpha_2} p(\alpha_2,\beta)\>\\
    &= O(\delta^2)\\
    \partial_{\beta} (f \circ \phi)(\alpha,\beta) &= \< \partial_\beta \phi(\alpha_1,\alpha_2,\beta), p(\alpha_2,\beta)\> \\
    &\quad + \< \phi(\alpha_1,\alpha_2,\beta) - x_0, \partial_\beta p(\alpha_2,\beta)\>\\
    &=  -\II_{11}(\mathbf{0})\alpha_1 + \f{1}{2}\II_{11}(\mathbf{0}) \delta + O(\delta^2).
\end{aligned}
$$
Then the magnitude of the gradient is
$$
\begin{aligned}
    \|\nabla f(\phi(\alpha_1,\alpha_2,\beta) \|^2 &= (g^{11}\circ \phi)(\partial_{\alpha_1}(f\circ \phi))^2 + 2 (g^{12}\circ \phi)\partial_{\alpha_1}(f\circ \phi)\partial_{\alpha_2} (f \circ \phi) \\
    & \quad + (g^{22}\circ \phi)(\partial_{\alpha_2}(f\circ \phi))^2+ (\partial_\beta (f \circ \phi))^2,
\end{aligned}
$$
and we find the individual summands
$$
\begin{aligned}
(g^{11}\circ \phi)(\partial_{\alpha_1}(f \circ \phi))^2 &= 1 + \left(-\alpha_1^2 + \delta \alpha_1 -\f{1}{4}\delta^2\right) \II_{11}(\mathbf{0})^2  + O(\delta^3),\\
(g^{12} \circ \phi) \partial_{\alpha_1} (f\circ \phi) \partial_{\alpha_2}(f \circ \phi) &= O(\delta^3),\\
(g^{22} \circ \phi) (\partial_{\alpha_2}(f \circ \phi))^2 &= O(\delta^3),\\
(g^{33} \circ \phi) (\partial_{\beta}(f \circ \phi))^2 &= \left(\alpha_1^2 -\delta\alpha_1 + \f{1}{4}\delta^2\right) \II_{11}(\mathbf{0})^2 + O(\delta^3),
\end{aligned}
$$
which indeed gives
$$
\|\nabla f(\phi(\alpha_1,\alpha_2,\beta) \| = 1+O(\delta^3)
$$
as the first and second order terms cancel out.
Hence \cref{test-function-gradient} applies and we conclude the lower bound coincides up to $O(\delta^4)$ with the upper bound.
\end{proof}

\section{General Riemannian submanifolds}
\label{section:riemannian-submanifolds}
We now consider a Riemannian submanifold $M$ of arbitrary dimension $m$ and codimension $k$ embedded isometrically in $\R^{m+k}$. Theorems \ref{space-curve} and \ref{surface-coarse-curvature} are thus special cases of \cref{general-submanifold} below.
We begin by defining an orthonormal frame of $\R^{m+k}$-valued vector fields on a sufficiently small open domain $U$ in the submanifold $M$, which is used to define the Fermi coordinates on $U$ in this general setting.

\subsection{Frame extension}
We take the ambient manifold to be $\R^{m+k}$.
Recall the second fundamental form at a point $x \in M$ is
$$
\II_x(w_1,w_2) = \nabla^{\R^{m+k}}_{w_1} W - \nabla^M_{w_1} W \quad \forall w_1,w_2 \in T_{x}M,
$$
where $W$ is an arbitrary local vector field on $M$ with $W(x) = w_2$. The mean curvature at $x$ is
$$
H(x) = \sum_{j=1}^m \II_x(e_j,e_j) 
$$
for an arbitrary orthonormal basis $(e_j)_{j=1}^m$ of $T_xM$. Both $\II_{x}(w_1,w_2)$ and $H(x)$ are normal to the submanifold, i.e. 
$$
\<\II_x(w_1,w_2), u\> = \< H(x), u\> = 0 \quad \forall u \in T_xM.
$$
Recall from \cref{definition-fermi-coordinates} that the Fermi coordinates in $M$ along $\gamma$ are given by
$$
\psi(\alpha) = \exp_{M, \gamma(\alpha_1)}\left( \sum_{j=2}^m \alpha_j e_j(\alpha_1)\right),
$$
where $(e_j(\alpha_1)_{j=1}^m$ is the parallel transport along $\gamma$ of an orthonormal basis $(e_j(0))_{j=1}^m$ of $T_{x_0}M$ with $e_1(0)=\dot{\gamma}(0)$. We refer back to Section 2 for properties of the Fermi chart.
 
Denote $\hat{\alpha} = (\alpha_2,\ldots,\alpha_m)$ so that $\alpha=(\alpha_1,\hat{\alpha})$. Extend the frame $(e_j(\alpha_1))_{j=1}^m$ defined along $\alpha_1 \mapsto \gamma(\alpha_1)$ to $U \subset M$ by imposing 
$$
\f{D}{ds} e_j(\alpha_1, s\hat{\alpha}) = 0,
$$
i.e. by parallel translating in $M$ along the geodesic $s \mapsto \psi(\alpha_1, s\hat{\alpha})$.

Given an initial orthonormal basis $(\mathbf{n}_i)_{i=1}^k$ of $T_{x_0}M^\perp$, first extend it to a frame along $\gamma$ by requiring that
$$
\partial_{\alpha_1} \<\mathbf{n}_i(\alpha_1), e_j(\alpha_1)\> = 0 \quad \textrm{and} \quad \left(\partial_{\alpha_1} \mathbf{n}_i(\alpha_1) \right)^\perp = 0 \quad \forall \alpha_1 \in (-\delta_0,\delta_0).
$$
The first requirement implies
$$
\begin{aligned}
\<\partial_{\alpha_1} \mathbf{n}_i(\alpha_1), e_j(\alpha_1)\> = - \< \mathbf{n}_i(\alpha_1), \partial_{\alpha_1} e_j(\alpha_1) \> = - \< \mathbf{n}_i(\alpha_1), \II(e_1(\alpha_1),e_j(\alpha_1)) \>
\end{aligned}
$$
which together with the second requirement implies the first order ODE
\begin{equation}
\begin{aligned}
\label{normal-frame-dynamics}
\partial_{\alpha_1} \mathbf{n}_i(\alpha_1) &= - \sum_{j=1}^m \< \mathbf{n}_i(\alpha_1), \partial_{\alpha_1} e_j(\alpha_1) \> e_j(\alpha_1) \\
&= - \sum_{j=1}^m\< \mathbf{n}_i(\alpha_1), \II(e_1(\alpha_1),e_j(\alpha_1)) \> e_j(\alpha_1).
\end{aligned}
\end{equation}
The solution exists and is unique by standard ODE theory.
Having defined the frame $(\mathbf{n}_i(\alpha_1))_{i=1}^k$ along the geodesic $\alpha_1 \mapsto \gamma(\alpha_1)$, we may also extend it to the submanifold by requiring that for every $\alpha_1 \in (-\delta_0,\delta_0)$ and $\hat{\alpha} \in \tilde{B}^{m-1}_{\varepsilon_0}$,
$$
\f{d}{ds} \<\mathbf{n}_i(\alpha_1, s\hat{\alpha}), e_j(\alpha_1, s\hat{\alpha}) \> = 0 \quad \textrm{and} \quad \left( \f{d}{ds} \mathbf{n}_i(\alpha_1, s\hat{\alpha}) \right)^\perp = 0 \quad \forall s \in [0,1].
$$
Similarly to the above, the first requirement implies that for all $j =1,\ldots,m$,
$$
\begin{aligned}
\<\f{d}{ds} \mathbf{n}_i(\alpha_1, s\hat{\alpha}), e_j(\alpha_1, s\hat{\alpha})\> &= - \< \mathbf{n}_i(\alpha_1, s\hat{\alpha}), \f{d}{ds} e_j(\alpha_1, s\hat{\alpha}) \>,
\end{aligned}
$$
and from the second requirement we conclude the frame satisfies the first order ODE
\begin{equation}
\label{normal-frame-ode}
\f{d}{ds} \mathbf{n}_i(\alpha_1, s\hat{\alpha}) = - \sum_{j=1}^m  \< \mathbf{n}_i(\alpha_1, s\hat{\alpha}), \f{d}{ds} e_j(\alpha_1, s\hat{\alpha}) \>  e_j(\alpha_1,s\hat{\alpha})
\end{equation}
along each geodesic $s \mapsto \phi(\alpha_1, s\hat{\alpha})$ in $M$.

With these concrete vector fields, recall the Fermi coordinates in $\R^{m+k}$ along $\gamma$ adapted to the submanifold $M$ were defined in \cref{definition-fermi-coordinates} as
$$
\phi(\alpha,\beta) = \psi(\alpha) + \sum_{i=1}^k \beta_i \mathbf{n}_i(\alpha)
$$
and note that $\phi(\alpha,\mathbf{0}) = \psi(\alpha)$.

For every $\alpha_1 \in (-\delta_0,\delta_0)$, the map $\psi(\alpha_1, \cdot): \tilde{B}^{m-1}_{\varepsilon_0} \rightarrow M$ is the exponential chart on its image. It is known that the Christoffel symbols vanish at the centre for such charts, i.e.
$$
\nabla^M_{\partial_{\alpha_i} \psi} \partial_{\alpha_j}\psi(\alpha_1,\mathbf{0}) = 0 \quad \forall i,j=2,\ldots,m.
$$
Moreover, since $\partial_{\alpha_j} \psi(\alpha_1,0) = e_j(\alpha_1)$ for $j=1,\ldots,m$ is parallel transport of $e_j(0)$ along $\gamma$, also
$$
\nabla^M_{\partial_{\alpha_1}\psi} \partial_{\alpha_j} \psi(\alpha_1,\mathbf{0}) = 0 \quad \forall j=1,\ldots,m,
$$
noting that $\dot{\gamma}(\alpha_1) = \partial_{\alpha_1} \psi(\alpha_1,\mathbf{0})$.

Denote the components of the second fundamental form with respect to the Fermi coordinates as
\begin{equation}
\label{fermi-second-fundamental-form}
\II_{i j\ell}(\alpha) = \< \partial_{\alpha_j}\partial_{\alpha_\ell} \psi(\alpha) - \nabla^M_{\partial_{\alpha_j} \psi} \partial_{\alpha_\ell} \psi(\alpha), \mathbf{n}_i(\alpha)\>.
\end{equation}
Note that the first index represents the normal direction and the latter two represent manifold directions. Then we can write for every $j,\ell=1,\ldots,m$,
\begin{equation}
\label{general-second-derivatives}
\begin{aligned}
\partial_{\alpha_j} \partial_{\alpha_\ell} \psi(\alpha_1,\mathbf{0}) = \partial_{\alpha_j} \partial_{\alpha_\ell} \psi(\alpha_1,\mathbf{0}) - \nabla^M_{\partial_{\alpha_j} \psi} \partial_{\alpha_\ell}\psi(\alpha_1,\mathbf{0}) =  \sum_{i=1}^k \II_{ij \ell}(\alpha_1,\mathbf{0}) \mathbf{n}_i(\alpha_1,\mathbf{0}).
\end{aligned}
\end{equation}
In addition, \eqref{normal-frame-dynamics} can be written as
$$
\partial_{\alpha_1} \mathbf{n}_i(\alpha_1)  =- \sum_{j=1}^m \II_{ij1}(\alpha_1) e_j(\alpha_1).
$$
Thus the third derivatives with at least one in $\alpha_1$ are
\begin{equation}
\label{general-third-derivatives}
\begin{aligned}
\partial_{\alpha_1} \partial_{\alpha_r} \partial_{\alpha_\ell} \psi(\alpha_1,\mathbf{0})
 &= \sum_{i=1}^k \partial_{\alpha_1} (\II_{ir\ell}(\alpha_1,\mathbf{0})) \mathbf{n}_i(\alpha_1,\mathbf{0}) \\
&\quad - \sum_{i=1}^k \sum_{j=1}^m \II_{ir\ell}(\alpha_1,\mathbf{0}) \II_{ij1}(\alpha_1,\mathbf{0}) e_j(\alpha_1).
\end{aligned}
\end{equation}

\subsection{Main theorem}
In the statement of the theorem, $\II_{x_0}(w_1,w_2)$ is the vector of second fundamental form. In the proof exclusively, $\II_{ij}(\alpha,\beta)$ denotes the $ij$-component of the second fundamental form with respect to the Fermi frame at Fermi coordinates $\alpha,\beta$.

\coarseextrinsiccurvature

\begin{remark}
    \label{rmk:special-cases}
We point out two special cases:
\begin{itemize}
\item If the submanifold has dimension 1 then the condition on the second fundamental form is trivially satisfied as there are no submanifold directions other than that of the curve itself. In this case
$$
H(x_0) = \II_{x_0}(e_1,e_1) = \nabla^{\R^{m+k}}_{\dot{\gamma}} \dot{\gamma}(0) = \ddot{\gamma}(0),
$$
and hence $\<\II_{x_0}(e_1,e_1),H(x_0)\> = \|\ddot{\gamma}(0)\|^2$. This is the square curvature of the curve and for $m=1$, $k=2$ agrees with \cref{space-curve}.
\item If the submanifold has codimension 1 with a normal vector field $\mathbf{n}$ on the submanifold, then the orthonormal eigenbasis of $\<\II_{x_0}(\cdot,\cdot),\mathbf{n}(x_0)\>$ satisfies the condition $\II_{x_0}(e_i,e_j) = \mathbf{0}$ for $i\neq j$. Such a basis always exists as $\II_{x_0}$ is symmetric and consists of the so-called principal curvature directions. Thus for $m=2$, $k=1$, we obtain \cref{surface-coarse-curvature} as a special case.

\item In general codimension, however, such a basis may not exist for a general submanifold, hence the assumption on the second fundamental form needs to be made and is highly restrictive.

If this assumption was dropped, the upper bound for the Wasserstein distance via the proposed transport map would still apply. However, the computation of the lower bound using a projection plane, as done in the proof of \cref{surface-coarse-curvature} and applied again in the proof below, would yield additional lower order terms not agreeing with the upper bound. This is symptomatic of the non-optimality of the transport map up to third order. The more general computation including the off-diagonal terms to show this is straightforward but rather lengthy and is thus omitted.

Qualitatively, the issue is that the off-diagonal terms of the second fundamental form introduce a deformation of the supports of the test measures which is not easily remedied and leaves the fully general case open. The deformation arises because the principal curvature directions above the reference point $x_0$ for each leaf of the foliation of the tubular neighbourhood change their vertical alignment as we consider leaves further away from the base submanifold $M$. On the other hand, the diagonal assumption on the second fundamental form ensures an aligned stacking of principal curvature directions of leaves above $x_0$, leading to the favourable cylinder-like support of the test measures.

\item For the interpretation of the special case of the parameters $\varepsilon=\sqrt{\f{2(m+2)}{k+2}}\sigma$, we refer back to \cref{pseudo-flat}.
\end{itemize}
\end{remark}

\begin{proof}[Proof of \cref{general-submanifold}]
Expand the Fermi chart up to and including third order as
$$
\begin{aligned}
    \phi(\alpha,\beta) &= x_0 + \sum_{j=1}^m \alpha_j \partial_{\alpha_j} \phi(\mathbf{0}) + \sum_{i=1}^k \beta_i \partial_{\beta_i} \phi(\mathbf{0}) + \f{1}{2} \sum_{j,\ell=1}^m \alpha_j \alpha_\ell \partial_{\alpha_j} \partial_{\alpha_\ell} \phi (\mathbf{0}) \\
    & \quad + \f{1}{6} \sum_{i,j,\ell=1}^m \alpha_i \alpha_j \alpha_\ell \partial_{\alpha_i} \partial_{\alpha_j} \partial_{\alpha_\ell} \phi(\mathbf{0}) + \sum_{i=1}^k \sum_{j=1}^m \beta_i \alpha_j \partial_{\beta_i} \partial_{\alpha_j} \phi(\mathbf{0})\\
    &\quad + \f{1}{2} \sum_{i=1}^k \sum_{j,\ell=1}^m \beta_i \alpha_j \alpha_\ell \partial_{\beta_i} \partial_{\alpha_j} \partial_{\alpha_\ell} \phi(\mathbf{0}) + O(\delta^4) .
    \end{aligned}
    $$
From the definition of the Fermi chart and \eqref{general-second-derivatives}, \eqref{general-third-derivatives}, we have the derivatives at the origin on the right hand side:
$$
\begin{aligned}
\partial_{\alpha_j} \phi(\mathbf{0}) &= e_j(0), \quad \partial_{\beta_i} \phi(\mathbf{0}) = \mathbf{n}_i(\mathbf{0}),\\
\partial_{\alpha_j} \partial_{\alpha_\ell} \phi (\mathbf{0})&=\sum_{i=1}^m\II_{i j \ell} (\mathbf{0}) \mathbf{n}_i(\mathbf{0}), \\
 \partial_{\alpha_1} \partial_{\alpha_j} \partial_{\alpha_\ell} \phi(\mathbf{0})
& =\sum_{i=1}^k \partial_{\alpha_1} (\II_{ij \ell}(\alpha_1,\mathbf{0})) \mathbf{n}_i(\alpha_1,\mathbf{0}) 
- \sum_{i=1}^k \sum_{r=1}^m \II_{ij\ell}(\mathbf{0}) \II_{ir 1}(\mathbf{0}) e_r(0).
\end{aligned}
$$
With these we obtain: 
    $$
\begin{aligned}
    \phi(\alpha,\beta) 
    &= x_0 + \sum_{j=1}^m \alpha_j e_j(0) + \sum_{i=1}^k \beta_i \mathbf{n}_i(\mathbf{0}) + \f{1}{2} \sum_{i=1}^k \sum_{r,\ell=1}^m \alpha_r \alpha_\ell \II_{i r \ell} (\mathbf{0}) \mathbf{n}_i(\mathbf{0}) \\
    &\quad - \f{1}{6} \sum_{j=1}^m \alpha_1^3 \II_{i11}(\mathbf{0}) \II_{ij1}(\mathbf{0}) e_j(0) - \f{1}{2} \sum_{j=1}^m \sum_{\ell=2}^m \alpha_1^2 \alpha_\ell \II_{i\ell 1}(\mathbf{0}) \II_{ij1}(\mathbf{0}) e_j(0) \\
    & \quad - \f{1}{2} \sum_{j=1}^m \sum_{\ell, r=2}^m \alpha_1 \alpha_r \alpha_\ell \II_{ir\ell}(\mathbf{0}) \II_{ij1}(\mathbf{0}) e_j(0) + \f{1}{6} \sum_{i,j,\ell=2}^m \alpha_i \alpha_j \alpha_\ell \partial_{\alpha_i} \partial_{\alpha_j} \partial_{\alpha_\ell} \phi(\mathbf{0})\\
    &\quad - \sum_{r=1}^k \sum_{\ell, j=1}^m \beta_r \alpha_\ell \II_{r \ell j}(\mathbf{0}) e_j(0) + \f{1}{2} \sum_{r=1}^k \sum_{l,q=1}^m \beta_r \alpha_\ell \alpha_q \partial_{\alpha_\ell} \partial_{\alpha_q} \mathbf{n}_r(\mathbf{0}) + O(\delta^4).
    \end{aligned}
    $$
    In the above, the sum of third derivative terms in $\alpha$ was split into those that involve at least one power in $\alpha_1$, for which we have a formula, and those that don't.
    The other third derivatives $\partial_{\alpha_i} \partial_{\alpha_r} \partial_{\alpha_\ell} \phi(\mathbf{0})$ for $i,r,\ell \geq 2$ are not easily written in Fermi coordinates, but will not be needed for our computations.
    Rearranging the terms, we write $\phi$ in terms of the basis $(e_1(0),\ldots,e_m(0)$, $\mathbf{n}_1(\mathbf{0}),\ldots, \mathbf{n}_k(\mathbf{0}))$ and apply the assumption $\II_{ij1}(\mathbf{0})=0$:
    \begin{equation}
    \label{general-fermi-chart-expansion}
    \begin{aligned}
   \phi(\alpha,\beta) 
    &= x_0 + \bigg( \alpha_1 - \sum_{r=1}^k \beta_r \alpha_1 \II_{r 1 1}(\mathbf{0}) - \f{1}{6}  \alpha_1^3 \sum_{i=1}^k \II_{i11}(\mathbf{0})^2 \\
    & \hspace{1.3cm} - \f{1}{2} \sum_{\ell, r=2}^m \alpha_1 \alpha_r \alpha_\ell \II_{ir\ell}(\mathbf{0}) \II_{i11}(\mathbf{0}) + \f{1}{6} \sum_{i,j,\ell=2}^m \alpha_i \alpha_j \alpha_\ell \<\partial_{\alpha_i} \partial_{\alpha_j} \partial_{\alpha_\ell} \phi(\mathbf{0}), e_1(0) \>  \\
    & \hspace{1.3cm} \left. + \f{1}{2} \sum_{r=1}^k \sum_{l,q=1}^m \beta_r \alpha_\ell \alpha_q \< \partial_{\alpha_\ell} \partial_{\alpha_q} \mathbf{n}_r(\mathbf{0}),e_1(0) \> + O(\delta^4)\right) e_1(0) \\
    & \hspace{1.3cm} + \sum_{j=2}^m \left(\alpha_j - \sum_{r=1}^k \sum_{\ell=2}^m \beta_r \alpha_\ell \II_{r \ell j}(\mathbf{0}) + O(\delta^3) \right) e_j(0) \\
    & \hspace{1.3cm} + \sum_{i=1}^k \left(\beta_i + \f{1}{2} \alpha_1^2 \II_{i11}(\mathbf{0})+ \f{1}{2} \sum_{r,\ell=2}^m \alpha_r \alpha_\ell \II_{i r \ell} (\mathbf{0})  + O(\delta^3)\right) \mathbf{n}_i(\mathbf{0}).
\end{aligned}
\end{equation}
We will henceforth denote
$$
r_i(\alpha):= \f{1}{2} \alpha_1 \< \partial_{\alpha_1}^2 \mathbf{n}_i(\mathbf{0}), e_1(0)\> + \sum_{\ell=2}^m \alpha_\ell \< \partial_{\alpha_\ell} \partial_{\alpha_1} \mathbf{n}_i(\mathbf{0}),e_1(0) \>.
$$

Let $T$ be the transport map defined in \cref{proposed-transport-map}. With asymptotic notation for the third order terms,
$$
T(\phi(\alpha_1,\hat{\alpha},\beta)) = \phi(\delta-\alpha_1,\hat{\alpha}, \beta+O(\delta^3)).
$$
In the expansion of $\phi$ above, from the third derivatives in $\alpha$ we only needed to specify those involving $\alpha_1$, because the transport map $T$ changes only the first coordinate up to $O(\delta^3)$. These derivatives were given by \eqref{general-third-derivatives}.
Then the pointwise transport vector is 
\begin{equation}
\label{pointwise-transport-vector}
\begin{aligned}
&T(\phi(\alpha,\beta)) - \phi(\alpha, \beta)\\
&= \phi(\delta-\alpha_1, \hat{\alpha}+O(\delta^3), \beta + O(\delta^3)) - \phi(\alpha,\beta)\\
&= (\delta - 2\alpha_1) \bigg[
\bigg(
1 - \f{1}{6} (\delta^2 - \delta \alpha_1 + \alpha_1^2) \sum_{i=1}^k \II_{i11}(\mathbf{0})^2 - \f{1}{2} \sum_{i=1}^k \sum_{r,\ell=2}^m \alpha_r \alpha_\ell \II_{i r \ell} (\mathbf{0}) \II_{i11}(\mathbf{0}) \\
& \hspace{3cm} - \sum_{i=1}^k \beta_i \II_{i11}(\mathbf{0}) + \sum_{i=1}^k \beta_i r_i(\alpha) + O(\delta^3) 
\bigg) e_1(0) \\
& \hspace{2cm} +\sum_{j=2}^m  O(\delta^2) e_j(0) + \sum_{i=1}^k \left( \f{\delta}{2} \II_{i11}(\mathbf{0}) + O(\delta^2) \right) \mathbf{n}_i(\mathbf{0})
\bigg].
\end{aligned}
\end{equation}
Therefore, using the expansion $\sqrt{1+x}=1+\f{1}{2}x-\f{1}{8}x^2 + O(x^3)$, the pointwise transport distance is
\begin{equation}
\label{pointwise-transport-magnitude}
\begin{aligned}
& \|T(\phi(\alpha,\beta) - \phi(\alpha,\beta)\| \\
&= (\delta - 2\alpha_1) \bigg( 1 -\f{1}{6} (\delta^2 - \delta \alpha_1 + \alpha_1^2) \sum_{i=1}^k \II_{i11}(\mathbf{0})^2 - \f{1}{2} \sum_{i=1}^k \sum_{r,\ell=2}^m \alpha_r \alpha_\ell \II_{i r \ell} (\mathbf{0}) \II_{i11}(\mathbf{0}) \\
& \hspace{2.7cm} + \f{\delta^2}{8} \sum_{i=1}^k \II_{i11}(\mathbf{0})^2 - \sum_{i=1}^k \beta_i \II_{i11}(\mathbf{0}) + \sum_{i=1}^k \beta_i r_i(\alpha) + O(\delta^3)
\bigg).
\end{aligned}
\end{equation}
\cref{test-measures-in-fermi-coordinates} expressed the density of the test measure $\mu_{x_0}^{\sigma,\varepsilon}$ in Fermi coordinates up to second order. Denoting the second order remainder of the density as $r(\alpha,\beta)$, the density simplifies to give
\begin{equation}
\label{test-function-density}
\begin{aligned}
&(\phi_*^{-1} \mu_{x_0}^{\sigma,\varepsilon})(d\alpha,d\beta) \\
&= \f{\mathbbm{1}_{\tilde{B}_{\sigma,\varepsilon}}(\alpha,\beta)}{\int_{\tilde{B}_{\sigma,\varepsilon}} (1+r(\alpha', \beta')) (\phi_*^{-1} \mu_{x_0}^{\sigma,\varepsilon})(d\alpha',d\beta')} \left(1 - \sum_{i=1}^k \sum_{j=1}^m \beta_i \II_{ijj}(\mathbf{0}) + r(\alpha, \beta) \right)d\alpha d\beta,
\end{aligned}
\end{equation}
where the form of the normalizing factor in the denominator is deduced from the two facts 
$$
\int_{\tilde{B}_{\sigma,\varepsilon}} \sum_{i=1}^k \sum_{j=1}^m \beta_i \II_{ijj}(\mathbf{0}) d(\phi^{-1}_* \mu_{x_0}^{\sigma,\varepsilon})(\alpha,\beta)=0, \quad \int_{\tilde{B}_{\sigma,\varepsilon}} d(\phi^{-1}_* \mu_{x_0}^{\sigma,\varepsilon})(\alpha,\beta)=1.
$$
We deduce the upper bound in the statement of \cref{general-submanifold} by computing the integral on the right side of the inequality
$$
W_1(\mu_{x_0}^{\sigma,\varepsilon}, \mu_y^{\sigma,\varepsilon}) \leq \int_{\tilde{B}_{\sigma,\varepsilon}} \|T(\phi(\alpha,\beta) - \phi(\alpha,\beta)\| (\phi_*^{-1} \mu_{x_0}^{\sigma,\varepsilon})(d\alpha,d\beta)
$$
up to and including third order terms.
Using the product of expressions \eqref{test-function-density} and \eqref{pointwise-transport-magnitude}, this amounts to integrating a quadratic polynomial in $\alpha,\beta$. First, as terms with odd power in one of the coordinates vanish, we simplify the integral to
$$
\begin{aligned}
&\int_{\tilde{B}_{\sigma,\varepsilon}} \|T(\phi(\alpha,\beta) - \phi(\alpha,\beta)\| (\phi_*^{-1} \mu_{x_0}^{\sigma,\varepsilon})(d\alpha,d\beta)\\
&= \delta \int_{\tilde{B}_{\sigma,\varepsilon}} \left( 1 - \f{\delta^2}{6} \sum_{i=1}^k \II_{i11}(\mathbf{0})^2  + \f{1}{2} \sum_{j=1}^m \alpha_j^2 \sum_{i=1}^k \big(\II_{ij 1}(\mathbf{0})^2 - \II_{i11}(\mathbf{0}) \II_{ijj}(\mathbf{0})\big)\right.\\
& \hspace{2.4cm} \left. + \sum_{i=1}^k \beta_i^2 \left( \sum_{j=1}^m \II_{i11}(\mathbf{0}) \II_{ijj}(\mathbf{0}) + \f{1}{2} \sum_{j=2}^m \II_{ij1}(\mathbf{0})^2 \right) \right) \; d\alpha \; d\beta + O(\delta^4).
\end{aligned}
$$
We now use the fact that the average integral of the square of any coordinate over a $d$-dimensional ball of arbitrary radius $r >0$ is
$$
\dashint_{B^d_r} x^2_i dx_1 \ldots dx_d = \frac{1}{|B_r^d|d} \int_0^r |\partial B^{d}_s| s^2 ds = \f{1}{r^d} \int_0^r s^{d+1}ds = \frac{r^2}{d+2},
$$
where $\dashint_{B^d_r}$ denotes the integral normalised by the volume of the ball and using that
$$
|B_r^d| = \frac{\pi^{\frac{d}{2}}}{\Gamma\left(\frac{d}{2}+1\right)} r^{d}, \quad |\partial B_{s}^{d}|=\frac{2 \pi^{\frac{d}{2}}}{\Gamma\left(\frac{d}{2}\right)} s^{d-1}.
$$
This in particular gives
$$
\dashint_{\tilde{B}_{\sigma,\varepsilon}} \alpha_j^2 \; d\alpha \; d\beta = \f{\varepsilon^2}{m+2}, \quad \dashint_{\tilde{B}_{\sigma,\varepsilon}} \beta_i^2 \; d\alpha \; d\beta = \f{\sigma^2}{k+2}.
$$
Then
$$
\begin{aligned}
&\int_{\tilde{B}_{\sigma,\varepsilon}} \|T(\phi(\alpha,\beta) - \phi(\alpha,\beta)\| (\phi_*^{-1} \mu_{x_0}^{\sigma,\varepsilon})(d\alpha,d\beta) \\
&= \delta \left(1 - \f{\delta^2}{24} \sum_{i=1}^k \II_{i11}(\mathbf{0})^2 + \left( \f{\sigma^2}{k+2} - \f{\varepsilon^2}{2(m+2)} \right) \sum_{i=1}^k  \sum_{j=1}^m \II_{i11}(\mathbf{0}) \II_{ij j}(\mathbf{0}) \right) + O(\delta^4) \\
&= \delta \left(1 - \f{\delta^2}{24} \sum_{i=1}^k \II_{i11}(\mathbf{0})^2 + \left( \f{\sigma^2}{k+2} - \f{\varepsilon^2}{2(m+2)} \right) \<\II_{x_0}(e_1,e_1), H(x_0)\>\right) + O(\delta^4).
\end{aligned}
$$
Furthermore, from \eqref{pointwise-transport-magnitude} for $\alpha = 0, \beta=0$ we deduce
$$
\|x_0-y\| = \delta \left(1 - \f{\delta^2}{24} \sum_{i=1}^k \II_{i11}(\mathbf{0})^2  \right) + O(\delta^4).
$$
Therefore, we can rewrite in terms of the Euclidean distance:
$$
\begin{aligned}
W_1(\mu_{x_0}^{\sigma,\varepsilon}, \mu_y^{\sigma,\varepsilon}) &\leq \int_{\tilde{B}_{\sigma,\varepsilon}} \|T(\phi(\alpha,\beta) - \phi(\alpha,\beta)\| (\phi_*^{-1} \mu_{x_0}^{\sigma,\varepsilon})(d\alpha,d\beta) \\
&=\|x_0-y\| \left( 1  +\left( \f{\sigma^2}{k+2} - \f{\varepsilon^2}{2(m+2)}\right) \<\II_{x_0}(e_1,e_1), H_{x_0}\>) \right)+ O(\delta^4).
\end{aligned}
$$
We now address the lower bound. Denoting
$$
p(\hat{\alpha},\beta) :=
\f{\phi(\delta,\hat{\alpha},\beta)-\phi(0,\hat{\alpha},\beta)}{\|\phi(\delta,\hat{\alpha},\beta)-\phi(0,\hat{\alpha},\beta)\|},
$$
emphasizing that this vector does not depend on $\alpha_1$, we propose
$$
\begin{aligned}
f(\phi(\alpha, \beta)) := \< \phi(\alpha,\beta)- x_0, p(\hat{\alpha},\beta) \>
\end{aligned}
$$
as the test function for Kantorovich-Rubinstein duality, with the intention of applying \cref{test-function-gradient} to conclude the upper bound is also a lower bound up to $O(\delta^4)$.
We deduce from \eqref{pointwise-transport-vector} that
\begin{equation}
\label{projection-vector-general}
\begin{aligned}
p(\hat{\alpha},\beta) 
&= \left(1 - \f{\delta^2}{8} \sum_{i=1}^k \II_{i11}(\mathbf{0})^2 + O(\delta^3) \right)e_1(0) \\
& \quad + \sum_{j=2}^m O(\delta^2) e_j(0) + \sum_{i=1}^k \left( \f{\delta}{2} \II_{i11}(\mathbf{0}) + O(\delta^2)  \right) \mathbf{n}_i(\mathbf{0}).
\end{aligned}
\end{equation}
Then it can be verified, using the expansions \eqref{pointwise-transport-vector} and \eqref{projection-vector-general} to compute the inner product up to and including third order terms, that
$$
\begin{aligned}
f(T(\phi(\alpha,\beta))) - f(\phi(\alpha,\beta))
&= \<T(\phi(\alpha,\beta)) - \phi(\alpha,\beta) + O(\delta^4), p(\hat{\alpha},\beta)\> \\
&= \|T(\phi(\alpha,\beta)) - \phi(\alpha,\beta)\| + O(\delta^4)
\end{aligned}
$$
by comparison with \eqref{pointwise-transport-magnitude}.

Finally, we wish to compute the square magnitude of the gradient of the test function in order to verify that its supremum over $B_{2\delta}(x_0)$ is $1+O(\delta^3)$ for \cref{test-function-gradient} to apply.
For this we need to establish the Riemannian metric in Fermi coordinates 
$g_{ij} = \<\partial_{\alpha_i}\phi, \partial_{\alpha_j} \phi\>.$
The first derivatives of the Fermi chart are deduced by differentiating \eqref{general-fermi-chart-expansion} as 
$$
\begin{aligned}
\partial_{\alpha_1} \phi(\alpha,\beta) &= \left( 1 - \sum_{i=1}^k \beta_i \II_{i11}(\mathbf{0}) - \f{1}{2} \sum_{i=1}^k \sum_{r,\ell=2}^m \alpha_r \alpha_\ell \II_{ir\ell}(\mathbf{0}) \II_{i 11}(\mathbf{0})   + O(\delta^3) \right) e_1(0)\\
& \quad + \sum_{\ell=2}^m O(\delta^2) e_\ell(0) + \sum_{i=1}^k \left( \alpha_1 \II_{i1 1}(\mathbf{0}) + O(\delta^2)\right)\mathbf{n}_i(\mathbf{0}),\\
\partial_{\alpha_j} \phi(\alpha,\beta) &= \sum_{\ell=2}^m \left(\delta_{\ell j} -\sum_{i=1}^k \beta_i \II_{i\ell j}(\mathbf{0}) + O(\delta^2) \right) e_\ell(0) \\
& \qquad + \sum_{i=1}^k \left( \sum_{\ell=2}^m \alpha_\ell \II_{i\ell j}(\mathbf{0}) + O(\delta^2)\right)\mathbf{n}_i(\mathbf{0}) \quad \textrm{for } 2 \leq j \leq m,\\
\partial_{\beta_i} \phi(\alpha,\beta) &= - \sum_{j=2}^m \left( \sum_{\ell=2}^m \alpha_\ell \II_{i \ell j}(\mathbf{0})+ O(\delta^2) \right) e_j(0) + \sum_{r=1}^k \left(\delta_{ir} + O(\delta^2) \right) \mathbf{n}_i(\mathbf{0}).
\end{aligned}
$$
Then the entries of the inverse metric matrix are computed from these to be
\begin{equation}
\label{metric-matrix-entries}
\begin{aligned}
g_{11}(\phi(\alpha,\beta)) &= 1 - 2 \sum_{i=1}^k \beta_i \II_{i11}(\mathbf{0})  + \sum_{i,r=1}^k \beta_i \beta_r \II_{i1 1}(\mathbf{0})^2 \\
& \qquad - \sum_{i=1}^k \sum_{r,\ell=2}^m \alpha_r \alpha_\ell \II_{ir\ell}(\mathbf{0}) \II_{i11}(\mathbf{0}) + O(\delta^3),\\
g_{j \ell}(\phi(\alpha,\beta)) &= \delta_{j\ell} - 2 \sum_{i=1}^k \beta_i \II_{i\ell j}(\mathbf{0}) + O(\delta^2) \quad \textrm{for } j, \ell \leq m,\\
g_{ij}(\phi(\alpha,\beta)) &= O(\delta^2) \quad \textrm{for } m+1 \leq i \leq m+k, \; j \leq m,\\
g_{ir}(\phi(\alpha,\beta)) &= \<\mathbf{n}_i(\alpha), \mathbf{n}_r(\alpha)\> = \delta_{ir} \quad \textrm{for } m+1 \leq i, \; r \leq m + k.
\end{aligned}
\end{equation}
Note that $\partial_{\alpha_1}\phi$ and $g_{11}$ needed to be expanded up to second order due to the particular role of the first coordinate. For the rest, expansion up to first order is sufficient.
The above means the metric matrix has the block structure
$$
G = 
\begin{pmatrix}
    (g_{j \ell})_{j,\ell \leq m} & O(\delta^2) \\
    O(\delta^2) & I_k
\end{pmatrix}.
$$
In particular, denoting 
$$
\begin{aligned}
a_{j\ell} &= -2 \sum_{i=1}^k \beta_i \II_{ij\ell}(\mathbf{0}), \\
b &= \sum_{\ell=1}^m \sum_{i,r=1}^k \beta_i \beta_r \II_{i\ell 1}(\mathbf{0}) \II_{r\ell 1}(\mathbf{0}) - \sum_{i=1}^k \sum_{r,\ell=2}^m \alpha_r \alpha_\ell \II_{ir\ell}(\mathbf{0}) \II_{i11}(\mathbf{0})
\end{aligned}
$$
and the matrix
$$
A = \begin{pmatrix}
    a_{11} + b + O(\delta^3) & O(\delta^2) & \hdots & O(\delta^2)\\
    O(\delta^2) & a_{22} + O(\delta^2) &  & \vdots \\
    \vdots & & \ddots & \\
    O(\delta^2) & \hdots &  & a_{mm} + O(\delta^2)
\end{pmatrix},
$$
having used that $a_{1j} =0$ for $j=2,\ldots,m$ as $\II_{ij1}(\mathbf{0})=0$ by assumption, we can write
$$
G= I_{m+k}+ \begin{pmatrix}A & O(\delta^2) \\ O(\delta^2) & \mathbf{0} \end{pmatrix}.
$$
Noting that the second matrix is $O(\delta)$, the expansion of its inverse is
$$
G^{-1} = \begin{pmatrix}
    I_m - A + A^2 + O(\delta^3) & O(\delta^2) \\
    O(\delta^2) & I_k + O(\delta^4)
\end{pmatrix}
$$
due to the block structure. Computing
$$
(A^2)_{j\ell} = \sum_{q=1}^m a_{j q} a_{\ell q} +O(\delta^3)= 4 \sum_{q=1}^m \sum_{i,r=1}^k \beta_i \beta_r \II_{ijq}(\mathbf{0}) \II_{r\ell q}(\mathbf{0}) + O(\delta^3),
$$
we deduce
\begin{equation}
\begin{aligned}
g^{11}(\phi(\alpha,\beta)) &= 1 -a_{11} + a_{11}^2 - b + O(\delta^3)\\
&= 1 + 2 \sum_{i=1}^k \beta_i \II_{i11}(\mathbf{0}) +3 \sum_{i,r=1}^k \beta_i \beta_r \II_{i1 1}(\mathbf{0}) \II_{r1 1}(\mathbf{0}) \\
& \qquad + \sum_{i=1}^k \sum_{j,\ell=2}^m \alpha_j \alpha_\ell \II_{ij\ell}(\mathbf{0}) \II_{i11}(\mathbf{0}) + O(\delta^3)
\end{aligned}
\end{equation}
by plugging in for $a_{j \ell}$ and $b$, and also
\begin{equation}
g^{j\ell}(\phi(\alpha,\beta)) = \delta_{j \ell} + 2 \sum_{i=1}^k \beta_i \II_{ij\ell}(\mathbf{0}) + O(\delta^2) \quad \forall j,l \leq m.
\end{equation}
We remark that for $j, \ell \geq 2$ the expansion of $g^{j\ell}$ up to the linear term suffices for the computations to follow, while the expansion of $g^{11}$ up to second order is necessary.

We now compute the expansions of the derivatives of the test function. The first derivatives of the projection vector field in coordinates can be computed from \eqref{projection-vector-general} as
$$
\begin{aligned}
\partial_{\alpha_j} p(\hat{\alpha},\beta) &= O(\delta) \quad \forall 2 \leq j \leq m, \\
\partial_{\beta_i} p(\hat{\alpha}, \beta)  &= O(\delta) \quad \forall 1 \leq i \leq k.
\end{aligned}
$$
Then computing the inner products, using \eqref{general-fermi-chart-expansion} for the derivatives of the chart,
$$
\begin{aligned}
    \partial_{\alpha_1} (f\circ \phi)(\alpha,\beta) &= \< \partial_{\alpha_1} \phi(\alpha,\beta), p(\hat{\alpha},\beta)\> \\
    &= 1 - \sum_{i=1}^k \beta_i \II_{i11}(\mathbf{0}) - \f{1}{2} \sum_{i=1}^k \sum_{j,\ell=2}^m \alpha_j \alpha_\ell \II_{ij\ell}(\mathbf{0}) \II_{i 11}(\mathbf{0}) \\
    & \qquad + \f{\alpha_1 \delta}{2} \sum_{i=1}^k \II_{i11}(\mathbf{0})^2 - \f{\delta^2}{8} \sum_{i=1}^k \II_{i11}(\mathbf{0})   + O(\delta^3),
    \end{aligned}
$$
    and for $ 2 \leq j \leq m$, 
    $$
\begin{aligned}
    \partial_{\alpha_j} (f \circ \phi)(\alpha,\beta) &= \langle \partial_{\alpha_j} \phi(\alpha,\beta), p(\hat{\alpha},\beta) \rangle + \< \phi(\alpha,\beta) - x_0, \partial_{\alpha_j} p(\hat{\alpha},\beta)\> \\
    &=  O(\delta^2),
\end{aligned}
$$
and for $i \leq k$, 
  $$
\begin{aligned}
    \partial_{\beta_i} ( f \circ \phi)(\alpha,\beta) &= \<\partial_{\beta_i} \phi(\alpha,\beta), p(\hat{\alpha},\beta)\> + \< \phi(\alpha,\beta) - x_0, \partial_{\beta_i} p(\hat{\alpha},\beta)\> \\
    &= - \alpha_1 \II_{i 1 1}(\mathbf{0}) + \f{\delta}{2} \II_{i11}(\mathbf{0}) + O(\delta^2).\\
\end{aligned}
$$
We wish to compute
$$
\begin{aligned}
\|\nabla f(\phi(\alpha,\beta))\|^2 &= \sum_{j,\ell=1}^m g^{j\ell}(\phi(\alpha,\beta)) \partial_{\alpha_j}(f \circ \phi)(\alpha,\beta) \partial_{\alpha_\ell} (f \circ \phi)(\alpha,\beta)\\
& \quad + 2 \sum_{i=1}^{k} \sum_{j=1}^m g^{m+i,j}(\phi(\alpha,\beta)) \partial_{\beta_i} (f \circ \phi)(\alpha,\beta) \partial_{\alpha_j} (f \circ \phi)(\alpha,\beta) \\
& \quad + \sum_{i=1}^{k} (\partial_{\beta_i} (f \circ \phi)(\alpha,\beta))^2.
\end{aligned}
$$
The individual summands are
$$
\begin{aligned}
(g^{11}\circ \phi) (\partial_{\alpha_1}(f \circ \phi))^2 &= 1 - \left(\alpha_1^2  - \alpha_1 \delta + \f{\delta^4}{4}\right) \sum_{i=1}^k \II_{i11}(\mathbf{0})^2 + O(\delta^3),\\
(g^{j\ell} \circ \phi) \partial_{\alpha_j} (f\circ \phi) \partial_{\alpha_\ell} (f\circ \phi) &= O(\delta^3) \textrm{ for } 1 \leq j\leq m, 2 \leq \ell \leq m,\\
(g^{m+i,j} \circ \phi) \partial_{\beta_i} (f \circ \phi) \partial_{\alpha_j} (f \circ \phi) &= O(\delta^3) \textrm{ for } i \leq k, j \leq m,\\
(\partial_{\beta_i} (f \circ \phi))^2 &= \left(\alpha_1^2  - \alpha_1 \delta + \f{\delta^4}{4}\right) \II_{i11}(\mathbf{0})^2 + O(\delta^3) \textrm{ for } i \leq k.
\end{aligned}
$$
All first and second order terms vanish upon summation, hence we may conclude that $\|\nabla f(\phi(\alpha,\beta))\|^2 = 1 + O(\delta^3)$ as required.
\end{proof}

\section{Applications}
\label{section:applications}
\subsection{Poisson point processes on manifolds}
In applications one may wish to recover curvature information from coarse curvature of a random point cloud represented by a Poisson point process. Such an approach has already been investigated in \cite{hoorn-2023} and \cite{arnaudon2023coarse} for the Ricci curvature and generalised Ricci curvature, respectively.

We first recall the definition of a Poisson point process. Let $(\mathcal{X}, \mathcal{B}, \mu)$ be a $\sigma$-finite measure space, $\mathcal{M}(\mathcal{X})$ the set of measures on $\mathcal{X}$ and $(\Omega, \mathcal{F}, \mathbb{P})$ a probability space.

\begin{definition}
\label{poisson-point-process}
A Poisson point process on $\mathcal{X}$ with intensity measure $\mu$  is a random measure $\mathcal{P}: \Omega \times \mathcal{B} \rightarrow [0,\infty]$ (equivalently $\mathcal{P}: \Omega \rightarrow \mathcal{M}(\mathcal{X})$) such that the following three properties hold:
\begin{enumerate}
    \item For all $\mu$-finite measurable sets $A \in \mathcal{B}$: $\mathcal{P}(\cdot,A)$ is a $\textrm{Poisson}(\mu(A))$ random variable,
    \item For all disjoint, measurable $\mu$-finite sets $A,B \in \mathcal{B}$: $\mathcal{P}(\cdot,A)$ and $\mathcal{P}(\cdot, B)$ are independent random variables,
    \item For all $\omega \in \Omega$: $\mathcal{P}(\omega, \cdot)$ is a measure on $\mathcal{X}$.
\end{enumerate}
\end{definition}

It turns out (see \cite[Chap. 6]{MR3791470}) that all Poisson point processes with a finite intensity measure take the form of a random empirical measure, i.e.
$$
\label{poisson-process-deltas}
\mathcal{P}(\omega,\cdot) =  \sum_{i=1}^{N(\omega)} \delta_{X_i(\omega)}
$$
where $N$ is a $\textrm{Poisson}(\mu(\mathcal{X}))$ random variable, $(X_i)_{i\in \mathbb{N}}$ are independent $\mu$-distributed random variables on $\mathcal{X}$ and $(X_i)_{i\in \mathbb{N}}, N$ are independent. Denote the random set of points thus generated by $\mathcal{P}$ as
$$
\mathcal{V}(\omega) = \{X_{i}(\omega) : 1\leq i \leq N(\omega)\}.
$$

\begin{notation}
Let $(\mathcal{P}_n)_{n\in \mathbb{N}}$ be a sequence of Poisson point processes on the ambient space $\R^{m+k}$ with uniform intensity measure $n \vol_{\R^{m+k}}(dz)$. Denote by $\mathcal{V}_n(\omega) \subset \mathbb{R}^{m+k}$ the discrete random set of points generated by $\mathcal{P}_n$. Let $x_0 \in M$,  $(\delta_n)_{n \in \mathbb{N}}, (\sigma_n)_{n \in \mathbb{N}}, (\varepsilon_n)_{n \in \mathbb{N}}$ sequences of positive reals and $y_n := \exp_{x_0}(\delta_n v)$ for a fixed unit vector $v \in T_{x_0}M$. As the discrete counterpart to the test measures $\mu_x^{\sigma,\varepsilon}$, for any point $x \in M$ denote the random empirical measures adapted to the submanifold, 
$$
\eta^{\sigma_n, \varepsilon_n}_x (z) = \begin{cases}
    \f{1}{\# (B_{\sigma_n,\varepsilon_n}(x) \cap \mathcal{V}_n)} & \textrm{ if } z \in B_{\sigma_n,\varepsilon_n}(x) \cap \mathcal{V}_n\\
    0 & \textrm{ otherwise}.
\end{cases}
$$
If $\sigma_n \vee \varepsilon_n \leq \f{\delta_n}{4}$ then
$
B_{\sigma_n,\varepsilon_n}(x_0) \cup B_{\sigma_n, \varepsilon_n}(y_n) \subset x_0 + [-2\delta_n,2\delta_n]^{m+k}
$.
\end{notation}

Using the following result proved in \cite[Corollary 3]{hoorn-2023}, it is possible to quantify the approximation of the test measures by the empirical measures in the Wasserstein metric:
\begin{lemma}
For all $n \in \mathbb{N}$, it holds that
\begin{equation}
\label{empirical-approximation}
\sup_{x \in B_{\delta_n}(x_0)} \E[W_1(\eta_{x}^{\sigma_n, \varepsilon_n}, \mu_x^{\sigma_n, \varepsilon_n})] = O\left(\log(n) n^{-\f{1}{m+k}}\right).
\end{equation}
\end{lemma}
We may then deduce that coarse curvature of point clouds with the empirical measures as test measures has the same limit as coarse extrinsic curvature if the intensity of the point process increases fast enough relative to the parameter $\delta_n$. Denote
$$
\hat{\kappa}_{\sigma_n, \varepsilon_n}(x_0, y_n) =1 - \f{W_1(\eta_{x_0}^{\sigma_n,\varepsilon_n}, \eta_{y_n}^{\sigma_n,\varepsilon_n})}{\delta_n}, \quad \kappa_{\sigma_n,\varepsilon_n}(x_0,y_n) = 1 - \f{W_1(\mu_{x_0}^{\sigma_n,\varepsilon_n}, \mu_{y_n}^{\sigma_n,\varepsilon_n})}{\delta_n}.
$$
This leads immediately to a corollary of \cref{general-submanifold}:
\begin{proposition}
    \label{point-cloud-curvature}
    Under the assumptions of \cref{general-submanifold}, if the sequences $(\delta_n)_{n \in \mathbb{N}}$, $(\sigma_n)_{n \in \mathbb{N}}$ and $(\varepsilon_n)_{n \in \mathbb{N}}$ satisfy $\sigma_n \vee \varepsilon_n \leq \f{\delta_n}{4}$ and $\log(n) n^{-\f{1}{m+k}}=o(\delta_n^3)$, then
    $$
    \lim_{n \rightarrow \infty} \f{1}{\delta_n^2} \E\left[\left|\hat{\kappa}_{\sigma_n,\varepsilon_n}(x_0, y_n) - \left(\f{\varepsilon_n^2}{2(m+2)} - \f{\sigma_n^2}{k+2} \right) \<\II_{x_0}(e_1,e_1), H(x_0)\> \right| \right] = 0.
    $$
\end{proposition}

\begin{proof}
By the triangle inequality and \eqref{empirical-approximation},
    $$
    \begin{aligned}
    \E\left[|\hat{\kappa}_{\sigma_n,\varepsilon_n}(x_0, y_n) - \kappa_{\sigma_n,\varepsilon_n}(x_0,y_n)| \right] &= \f{1}{\delta_n} \E\left[ |W_1(\eta_{x_0}^{\sigma_n,\varepsilon_n}, \eta_{y_n}^{\sigma_n,\varepsilon_n}) - W_1(\mu_{x_0}^{\sigma_n,\varepsilon_n}, \mu_{y_n}^{\sigma_n,\varepsilon_n})| \right]
    \\
    &\leq \f{1}{\delta_n} \E[W_1(\eta_{x_0}^{\sigma_n,\varepsilon_n}, \mu_{x_0}^{\sigma_n,\varepsilon_n}) + W_1(\eta_{y}^{\sigma_n,\varepsilon_n}, \mu_{y}^{\sigma_n,\varepsilon_n})] \\
     &= \f{1}{\delta_n} O\left(\log (n) n^{-\f{1}{m+k}}\right) = o(\delta_n^2).
     \end{aligned}
    $$
    At the same time, from \cref{general-submanifold} we have
    $$
    \kappa_{\sigma_n, \varepsilon_n}(x_0,y_n) = \left(\f{\varepsilon_n^2}{2(m+2)} - \f{\sigma_n^2}{k+2} \right) \<\II_{x_0}(e_1,e_1), H(x_0)\>,
    $$
    which gives the final result upon substitution and taking the limit as $n \rightarrow \infty$.
\end{proof}

\subsection{Retrieving mean curvature}
\cref{general-submanifold} could in practice be exploited in the two settings already alluded to in the introduction, which considered the planar curve case for illustrative purposes. In the scope of generality of \cref{general-submanifold}, we have
$$
\lim_{\substack{\sigma,\varepsilon \leq \delta/4 \\ \delta \rightarrow 0}} \left( \f{\varepsilon^2}{2(m+2)} - \f{\sigma^2}{k+2}\right)^{-1} \left(1- \f{W_1(\mu_{x_0}^{\sigma,\varepsilon}, \mu_y^{\sigma,\varepsilon})}{\|y-x_0\|} \right) = \<\II_{x_0}(e_1,e_1), H(x_0)\>.
$$
In particular, we may distinguish two limit regimes:
\begin{enumerate}[label=(\arabic*)]
    \item Assuming $\sigma= \Theta(\delta)$ and $\varepsilon = o(\sigma)$, 
    $$
    \begin{aligned}
    -\lim_{\delta \rightarrow 0} \f{k+2}{\sigma^2}\left(1- \f{W_1(\mu_{x_0}^{\sigma,\varepsilon}, \mu_y^{\sigma,\varepsilon})}{\|y-x_0\|} \right) = \<\II_{x_0}(e_1,e_1), H(x_0)\>.
    \end{aligned}
    $$
    This represents a situation where one can obtain a sample from the ambient measure in a tubular neighbourhood of the surface.
    Decreasing $\varepsilon$ corresponds to localization of the geometric information thus retrieved. 
    
    \item Assuming $\varepsilon = \Theta(\delta)$ and $\sigma =o(\varepsilon)$,
    $$
    \lim_{\delta \rightarrow 0} \f{2(m+2)}{\varepsilon^2}\left(1- \f{W_1(\mu_{x_0}^{\sigma,\varepsilon}, \mu_y^{\sigma,\varepsilon})}{\|y-x_0\|} \right) = \<\II_{x_0}(e_1,e_1), H(x_0)\>.
    $$
    In this case, we have a noisy sample from the surface and obtain convergence of the coarse extrinsic curvature under attenuation of the noise as $\sigma$ decreases. 
\end{enumerate}
Note that these expressions depend on the vector $v$ with $y_\delta = \exp_{M, x_0}(\delta v)$. We can remove this directionality by adding up coarse curvatures in all directions of an orthonormal frame at $x_0$, thus obtaining an expression involving the mean curvature.

Denote the square norm of the mean curvature vector as
$$
\|H(x_0)\|^2 = \sum_{i=1}^k \<H(x_0), \mathbf{n}_i(x_0)\>^2
$$
for an arbitrary orthonormal basis $(\mathbf{n}_i(x_0))_{i=1}^k$ of the normal space $T_{x_0}M^\perp \subset T_{x_0}N$.

\coarsemeancurvature

\begin{proof}
    We express the coarse curvatures using the expansion of \cref{general-submanifold} and sum up, noting that $j = 1,\ldots,m$ indexing each direction plays the role of the first coordinate, 
    $$
    \begin{aligned}
    \sum_{j=1}^m \left( 1- \f{W_1(\mu_{x_0}^{\sigma,\varepsilon}, \mu_{y_j}^{\sigma,\varepsilon})}{\|x_0-y_j\|}\right) &= \left( \f{\varepsilon^2}{2(m+2)} - \f{\sigma^2}{k+2} \right)\sum_{j=1}^m \< \II_{x_0}(e_j,e_j), H(x_0)\> + O(\delta^3) \\
    & = \left( \f{\varepsilon^2}{2(m+2)} - \f{\sigma^2}{k+2} \right) \|H(x_0)\|^2 +O(\delta^3),
    \end{aligned}
    $$
    completing the proof.
\end{proof}
This implies that given the family of coarse curvatures
$$
\left\{ 1- \f{W_1(\mu_{x_0}^{\sigma,\varepsilon}, \mu_{y_j}^{\sigma,\varepsilon})}{\|x_0-y_j\|} : \sigma, \varepsilon, \delta > 0, j = 1,\ldots, m \right\},
$$
one can retrieve the square magnitude of the mean curvature vector of the surface at $x_0$ as
$$
\lim_{\substack{\sigma,\varepsilon \leq \delta/4\\ \delta \rightarrow 0}} \left( \f{\varepsilon^2}{2(m+2)} - \f{\sigma^2}{k+2} \right)^{-1} \sum_{j=1}^m \left( 1- \f{W_1(\mu_{x_0}^{\sigma,\varepsilon}, \mu_{y_j}^{\sigma,\varepsilon})}{\|x_0-y_j\|}\right)  = \|H_{x_0}\|^2.
$$

In conclusion, we introduced the notion of coarse extrinsic curvature of Riemannian submanifolds embedded isometrically in a Euclidean space and verified that in a scaled limit of the parameters we retrieve meaningful geometric information about the submanifold. As illustrative examples, in the case of a curve we retrieve the inverse squared radius of the osculating circle at a given point, while in the case of a 2-surface we obtain an expression in terms of the second fundamental form and mean curvature. Such coarse extrinsic curvatures can be combined to yield the square magnitude of the mean curvature as a scaled limit.

\section*{Funding declaration}

Xue-Mei Li acknowledges
support from Engineering and Physical Sciences Research Council grant EP/V026100/1; Benedikt Petko was supported by the Engineering and Physical Sciences Research Council Centre for Doctoral Training in Mathematics of Random Systems: Analysis, Modelling 
and Simulation (EP/S023925/1).

\bibliographystyle{unsrt}
\bibliography{CoarseExtrinsicCurvature}

\end{document}